\documentclass[a4paper]{amsart}

\usepackage[english]{babel}
\usepackage[utf8x]{inputenc}
\usepackage[T1]{fontenc}
\usepackage{mathtools}
\usepackage{bm}

\usepackage{amsmath}
\usepackage{graphicx}
\usepackage[colorinlistoftodos]{todonotes}
\usepackage[colorlinks=true, allcolors=blue]{hyperref}
\usepackage{float}
\usepackage{fancyhdr}
\usepackage{amssymb}
\usepackage{multicol}
\usepackage{dsfont}

\usepackage{microtype}
\usepackage{listings}
\usepackage{enumerate}
\usepackage{amsthm}
\usepackage{amsmath}

\usepackage{tikz}
\usepackage{pgfplots}
\pgfplotsset{compat=1.11}
\usepgfplotslibrary{fillbetween}
\usetikzlibrary{intersections}
\pgfdeclarelayer{bg}
\pgfsetlayers{bg,main}
\usepackage{url}
\usepackage{hyperref}
\usepackage{amssymb}
\usepackage{MnSymbol}
\usepackage{tikz}
\usepackage{xcolor}
\usepackage{pgfkeys}
\usepackage{caption}
\usetikzlibrary{}

\usepackage[all,cmtip]{xy}

\theoremstyle{plain}
\newtheorem{theorem}{Theorem}[section]
\newtheorem*{theorem*}{Theorem}
\newtheorem{lemma}[theorem]{Lemma}
\newtheorem{proposition}[theorem]{Proposition}
\newtheorem{corollary}[theorem]{Corollary}
\newtheorem{definition}[theorem]{Definition}

\newtheorem{exmp}[theorem]{Example}

\newtheorem{rem}[theorem]{Remark}
\newtheorem*{theorema}{Theorem A}
\newtheorem*{theoremb}{Theorem B}
\newtheorem*{theoremc}{Theorem C}

\theoremstyle{remark}
\newtheorem*{note}{Note}
\newtheorem*{claim}{Claim}

	\pgfkeys{
		/churro/.is family
		, /churro
		, default/.style = 
		{ scale=1
		, x=0
		, y=0
		, color=white
		}
		, scale/.estore in=\Churroscale 
		, x/.estore in=\Churrox
		, y/.estore in=\Churroy
		, color/.estore in=\Churrocolor
	}
	\newcommand{\churro}[1][]{		
		\pgfkeys{/churro, default, #1}
		\begin{scope}[scale=\Churroscale,xshift=\Churrox*28.45/\Churroscale pt, yshift=\Churroy*28.45/\Churroscale pt] 
			\draw [fill=\Churrocolor] 
			(0,0) to[out=0,in=-90] (0.3,0.3)
			to[out=90,in=-120] (0.2,1.25)
			to[out=60,in=120] (0.8,1.25)
			to[out=-60,in=90] (0.7,0.3)
			to[out=-90,in=180] (1,0);
			
			\draw 
			(0.3,1.1) to[out=-80,in=-100] (0.7,1.1);
			
			\draw 
			(0.35,1.06) to[out=70,in=110] (0.65,1.06);
		\end{scope}
	}

\usetikzlibrary{decorations.pathreplacing}
\usetikzlibrary{decorations.markings}

	\pgfkeys{
		/cube/.is family
		, /cube
		, default/.style = 
		{ scale=1
		, x=0
		, y=0
		, color=white
		}
		, scale/.estore in=\Cubescale 
		, x/.estore in=\Cubex
		, y/.estore in=\Cubey
		, color/.estore in=\Cubecolor
	}

	\pgfkeys{
		/huequito/.is family
		, /huequito
		, default/.style = { x=0, y=0 }
		, x/.estore in=\Huequitox
		, y/.estore in=\Huequitoy
	}
	\newcommand{\huequito}[1][]{		
		\pgfkeys{/huequito, default, #1}
		\begin{scope}[xshift=\Huequitox*28.45 pt, yshift=\Huequitoy*28.45 pt]
			\draw 
			(-.2,.04) to[out=-80,in=-100] (.2,.04);
			
			\draw 
			(-.15,0) to[out=70,in=110] (.15,0);
		\end{scope}
	}

\newcommand{\Hom}{\operatorname{Hom}}
\newcommand{\id}{\operatorname{id}}
\newcommand{\im}{\operatorname{im}}
\renewcommand{\int}{\operatorname{int}}

\newcommand{\rk}{\operatorname{rk}}
\newcommand{\Aut}{\operatorname{Aut}}
\newcommand{\incl}{\operatorname{incl}}
\newcommand{\res}{\operatorname{res}}
\newcommand{\colim}{\operatorname{colim}}

\newcommand{\coker}{\operatorname{coker}}
\newcommand{\Diff}{\operatorname{Diff}}
\newcommand{\Alg}{\operatorname{Alg}}
\newcommand{\hofib}{\operatorname{hofib}}
\newcommand{\Arf}{\operatorname{Arf}}

\newcommand{\bR}{\overline{R}}
\newcommand{\gr}{\operatorname{gr}}
\newcommand{\Sp}{\operatorname{Sp}}

\renewcommand{\mod}{\operatorname{mod}}
\newcommand{\FF}{\operatorname{\mathbb{F}_2}}

\newcommand{\Lk}{\operatorname{Lk}}	
\newcommand{\Emb}{\operatorname{Emb}}

\newcommand{\rad}{\operatorname{rad}}

\newcommand{\Diffhalf}{\operatorname{Diff_{\frac{1}{2}\partial}}}
\newcommand{\Mhalf}{\mathcal{M}[I^{2n-1}]}

\newcommand{\Fr}{\operatorname{Fr}}
\newcommand{\Tors}{\operatorname{Tors}}
\newcommand{\Ann}{\operatorname{Ann}}

\title[Homological stability of diffeomorphism groups]{Homological stability of diffeomorphism groups of high dimensional manifolds via $E_k$-algebras}

\author{Ismael Sierra}

\begin{document}
\large
\maketitle

\begin{abstract}
We will study homological stability of the diffeomorphism groups of the manifolds $W_{g,1}:=D^{2n} \# (S^n \times S^n)^{\#g }$ using $E_k$-algebras. 
This will lead to new improvements in the stability results, especially when working with rational coefficients. 
Moreover, we will prove a new type of stability result---quantised homological stability---which says that either the best stability result is a linear bound of slope $1/2$ or the stability is at least as good as a line of slope $2/3$. 
\end{abstract}

\section{Introduction} \label{section 1}

\subsection{Statement of results} \label{section 1.1}

For a $d$-dimensional manifold with boundary $W$ let us denote the group of diffeomorphisms of $W$ fixing pointwise a neighbourhood of the boundary by $\Diff_{\partial}(W)$. 
The classifying space of this group is of geometric importance since it classifies smooth fibre bundles with fibre $W$ and a trivialization over $\partial W$, and hence $H^*(B \Diff_{\partial}(W))$ is the ring of characteristic classes of such bundles. 
Since homology essentially contains the same amount of information as cohomology, computing $H_*(B\Diff_{\partial}(W))$ has lots of potential applications. 

We will focus on the following manifolds, which generalize orientable surfaces with one boundary component to high dimensions:  
for a fixed $n \geq 1$ we define the following $2n$-dimensional manifold for each $g \geq 0$
$$W_{g,1}:= D^{2n} \# (S^n \times S^n)^{\#g }.$$ 
We can view $W_{g,1}$ as the boundary connected sum of $W_{g-1,1}$ with $W_{1,1}$, and hence extension by the identity gives inclusions $\Diff_{\partial}(W_{g-1,1}) \hookrightarrow \Diff_{\partial}(W_{g,1})$. 
By \cite[Theorems 1.1, 1.2]{stablemoduli} the \textit{stable homology} $\colim_{g}H_d(B\Diff_{\partial}(W_{g,1}))$ has an expression in terms of a certain Thom spectrum, and has been explicitly computed with rational coefficients.
Furthermore, by \cite[Theorem 1.2]{high}, for $n \geq 3$ the maps
$$H_d(B\Diff_{\partial}(W_{g-1,1})) \rightarrow H_d(B\Diff_{\partial}(W_{g,1}))$$
are isomorphisms for $2d \leq g-4$. 
Thus, one can access the homology groups $H_d(B\Diff_{\partial}(W_{g,1}))$ in the range $d \lesssim g/2$. 

The goal of this paper is to improve the homological stability bound of the family of groups $\{\Diff_{\partial}(W_{g,1})\}_{g \geq 1}$, and hence improving the range in which their homology groups can be computed. 
The main result of the paper is

\begin{theorema} \label{theorem A}
For $n \geq 3$ odd, consider the stabilization map 
$$H_d(B \Diff_{\partial}(W_{g-1,1});\mathds{k}) \rightarrow H_d(B \Diff_{\partial}(W_{g,1});\mathds{k})$$
Then
\begin{enumerate}[(i)]
    \item If $n=3$ or $7$ and $\mathds{k}=\mathbb{Z}$ it is surjective for $3d \leq 2g-1$ and an isomorphism for $3d \leq 2g-4$. 
    \item If $n \neq 3$ or $7$ and $\mathds{k}=\mathbb{Z}$ it is surjective for $2d \leq g-2$ and an isomorphism for $2d \leq g-4$. 
    \item If $n \neq 3$ or $7$ and $\mathds{k}=\mathbb{Z}[\frac{1}{2}]$ it is surjective for $3d \leq 2g-4$ and an isomorphism for $3d \leq 2g-7$. 
    \item If $n=3$ or $7$ and $\mathds{k}=\mathbb{Q}$ it is surjective for $d< \frac{3n-6}{3n-5}g$ and an isomorphism for $d< \frac{3n-6}{3n-5}g-1$.
    \item If $n \neq 3$ or $7$ and $\mathds{k}=\mathbb{Q}$ it is surjective for $d< \frac{3n-6}{3n-5}(g-1)$ and an isomorphism for $d< \frac{3n-6}{3n-5}(g-1)-1$. 
\end{enumerate}
\end{theorema}

Let us briefly comment on the different parts of this result: part (ii) is essentially the same stability result found in \cite[Theorem 1.2]{high}, however our proof is done using different methods. 
Part (i) gives an improvement of the previously known stability result for $n=3$ or $7$. 
Part (iii) improves the previous stability results when $n \neq 3$ or $7$ but only away from the prime $2$. 
Finally, parts (iv) and (v) give a significant improvement in homological stability when restricting to rational coefficients. 

In order to prove Theorem \hyperref[theorem A]{A} we will use the cellular $E_k$-algebras approach developed in \cite{Ek}, which has already been applied in other contexts to improve several homological stability results. 
This paper is in some sense a high-dimensional analogue of the ideas of \cite{E2}, where cellular $E_2$-algebras are used prove the best possible homological stability results of mapping class groups of orientable surfaces.
One of the main novelties of this paper is that it deals with the topological groups $\Diff_{\partial}(W_{g,1})$ rather than discrete groups, being the first example of application of the cellular $E_k$-algebras method in this context. 

One of the main features of the $E_k$-algebras method is that it can give different homological stability bounds for different coefficients, as one can see in Theorem \hyperref[theorem A]{A}. 
In our example, the stability range we get with $\mathbb{Q}$ coefficients is $d \lesssim \frac{3n-6}{3n-5}g$, which suggests the conjecture that the optimal stability bound should have slope at least $1$, meaning that if should be of the form $d \leq \lambda g +c$ for $\lambda \geq 1$ and $c$ a constant. 

Finally, let us remark that the restriction to $n$ odd is due to a technical step explained in Remark \ref{rem n odd}. 
It has to do with the fact that the intersection form on $W_{g,1}$ is either symmetric or skew-symmetric depending on the parity of $n$, and these two types of forms behave very differently.
All the other steps in Sections \ref{section 2}, \ref{section 3}, \ref{section 4} can be carried out for the case $n$ even, and in fact it is likely that one could eliminate the restriction of $n$ odd by using a different ``algebraic arc complex'' to the one of Section \ref{section 6} of this paper.  

Using the ingredients of the proof of Theorem \hyperref[theorem A]{A} and results from \cite{Sierra2022-st} we will also prove the following ``quantisation stability'' result, which, as explained in Remark \ref{remark quantised} differs from other similar results in the literature. 

\begin{theoremb} \label{theorem B}
Let $n \geq 5$ be odd, $n \neq 7$, then one of the following two options holds
\begin{enumerate}[(i)]
    \item The relative homology groups 
    $H_{2k}(B \Diff_{\partial}(W_{4k,1}), B \Diff_{\partial}(W_{4k-1,1});\mathbb{Z})$ are non-zero for all $k \geq 1$. 
    \item $H_d(B \Diff_{\partial}(W_{g-1,1});\mathbb{Z}) \rightarrow H_d(B \Diff_{\partial}(W_{g,1});\mathbb{Z})$ is surjective for $3d \leq 2g-6$ and an isomorphism for $3d \leq 2g-9$. 
\end{enumerate}
\end{theoremb}

In other words, it says that either part (ii) of Theorem \hyperref[theorem A]{A} is essentially optimal or the true stability result has slope at least $2/3$. 
We do not know which of the two alternatives holds. 

\subsection{Overview of cellular $E_k$-algebras} \label{section 1.2}

The purpose of this section is to explain the methods from \cite{Ek} used in this paper: we aim for an informal discussion and refer to \cite{Ek} for details.

In the $E_k$-algebras part of the paper we will work in the category $\mathsf{sMod}_{\mathds{k}}^{\mathsf{G}}$ of $\mathsf{G}$-graded simplicial $\mathds{k}$-modules, for $\mathds{k}$ a commutative ring and $\mathsf{G}$ a discrete symmetric monoid. 
Formally, $\mathsf{sMod}_{\mathds{k}}^{\mathsf{G}}$ denotes the category of functors from $\mathsf{G}$, viewed as a category with objects the elements of $\mathsf{G}$ and only identity morphisms, to $\mathsf{sMod}_{\mathds{k}}$. 
This means that each object $M$ consists of a simplicial $\mathds{k}$-module $M_{\bullet}(x)$ for each $x \in \mathsf{G}$. 
The tensor product $\otimes$ in this category is given by Day convolution, i.e.
$$ (M \otimes N)_p(x)= \bigoplus_{y+z=x}{M_p(y) \otimes_{\mathds{k}} N_p(z)}$$
where $+$ denotes the monoidal structure of $\mathsf{G}$. 

In a similar way one can define the category of $\mathsf{G}$-graded spaces, denoted by $\mathsf{Top}^{\mathsf{G}}$ and endow it with a monoidal structure by Day convolution using cartesian product of spaces. 

The \textit{little k-cubes} operad in $\mathsf{Top}$ has $n$-ary operations given by rectilinear embeddings $I^k \times \{1,\cdots,n\} \hookrightarrow I^k$ such that the interiors of the images of the cubes are disjoint.
(The space of 0-ary operations is empty.) 
We define the little $k$-cubes operad in $\mathsf{sMod}_{\mathds{k}}$ by applying the symmetric monoidal functor
$(-)_{\mathds{k}}: \mathsf{Top} \rightarrow \mathsf{sMod}_{\mathds{k}}$ given by the free $\mathds{k}$-module on the singular simplicial set of a space.
Moreover, $(-)_{\mathds{k}}$ can be promoted to a functor $(-)_{\mathds{k}}: \mathsf{Top}^{\mathsf{G}} \rightarrow \mathsf{sMod}_{\mathds{k}}^{\mathsf{G}}$ between the graded categories, and we define the little $k$-cubes operad in these by concentrating it in grading $0$, where $0 \in \mathsf{G}$ denotes the identity of the monoid.  
We shall denote this operad by $\mathcal{C}_k$ in all the categories  $\mathsf{Top}, \mathsf{Top}^{\mathsf{G}}, \mathsf{sMod}_{\mathds{k}}^{\mathsf{G}}$ which we use, and define an $E_k$-\textit{algebra} to mean an algebra over this operad. 

The $E_k$-\textit{indecomposables} of an $E_k$-algebra $\textbf{R}$ in $\mathsf{sMod}_{\mathds{k}}^{\mathsf{G}}$ is defined by the exact sequence of graded simplicial $\mathds{k}$-modules

$$ \bigoplus_{n \geq 2}{\mathcal{C}_k(n) \otimes \textbf{R}^{\otimes n}} \rightarrow \textbf{R} \rightarrow Q^{E_k}(\textbf{R}) \rightarrow 0.$$

The functor $\textbf{R} \mapsto Q^{E_k}(\textbf{R})$ is not homotopy-invariant but has a derived functor $Q_{\mathbb{L}}^{E_k}(-)$ which is.
See \cite[Section 13]{Ek} for details and how to define it in more general categories such as $E_k$-algebras in $\mathsf{Top}$ or $\mathsf{Top}^{\mathsf{G}}$. 
The $E_k$-\textit{homology} groups of $\textbf{R}$ are defined to be
$$H_{x,d}^{E_k}(\textbf{R}):=H_d(Q_{\mathbb{L}}^{E_k}(\textbf{R})(x))$$
for $x \in \mathsf{G}$ and $d \in \mathbb{N}$. 
One formal property of the derived indecomposables, see \cite[Lemma 18.2]{Ek}, is that it commutes with $(-)_{\mathds{k}}$, so for $\textbf{R} \in \Alg_{E_k}(\mathsf{Top}^{\mathsf{G}})$ its $E_k$-homology with $\mathds{k}$ coefficients is the same as the $E_k$-homology of $\textbf{R}_{\mathds{k}}$. 

Thus, in order to study homological stability of $\textbf{R}$ with different coefficients we can work with the $E_k$-algebras $\textbf{R}_{\mathds{k}}$ instead, which enjoy better properties as they are cofibrant and the category of graded simplicial $\mathds{k}$-modules offers some technical advantages as explained in \cite[Section 11]{Ek}.
However, at the same time, we can do computations in $\mathsf{Top}$ of the homology or $E_k$-homology of $\textbf{R}$ and then transfer them to $\mathsf{sMod}_{\mathds{k}}$. 

In \cite[Section 6]{Ek} the notion of a \textit{CW $E_k$-algebra} is defined, built in terms of free $E_k$-algebras by iteratively attaching cells in the category of $E_k$-algebras in order of dimension. 

Let $\Delta^{x,d} \in \mathsf{sSet}^{\mathsf{G}}$ be the standard $d$-simplex placed in grading $x$ and let $\partial \Delta^{x,d} \in \mathsf{sSet}^{\mathsf{G}}$ be its boundary. 
By applying the free $\mathds{k}$-module functor we get objects $\Delta_{\mathds{k}}^{x,d}, \partial \Delta_{\mathds{k}}^{x,d} \in \mathsf{sMod}_{\mathds{k}}^{\mathsf{G}}$. 
We then define the graded spheres in $\mathsf{sMod}_{\mathds{k}}^{\mathsf{G}}$ via $S_{\mathds{k}}^{x,d}:=\Delta_{\mathds{k}}^{x,d}/ \partial \Delta_{\mathds{k}}^{x,d}$, where the quotient denotes the cofibre of the inclusion of the boundary into the disc. 
In $\mathsf{sMod}_{\mathds{k}}^{\mathsf{G}}$, the data for a cell attachment to an $E_k$-algebra $\textbf{R}$ is an \textit{attaching map} $e: \partial \Delta_{\mathds{k}}^{x,d} \rightarrow \textbf{R}$, which is the same as a map $\partial \Delta^d_{\mathds{k}} \rightarrow \textbf{R}(x)$ of simplicial $\mathds{k}$-modules. 
To attach the cell we form the pushout in $\Alg_{E_k}(\mathsf{sMod}_{\mathds{k}}^{\mathsf{G}})$

\centerline{\xymatrix{ \mathbf{E_k}(\partial \Delta_{\mathds{k}}^{x,d}) \ar[r] \ar[d] & \textbf{R} \ar[d] \\
\mathbf{E_k}(\Delta_{\mathds{k}}^{x,d}) \ar[r] & \textbf{R} \cup_{e}^{E_k} D_{\mathds{k}}^{x,d}.
}}

A weak equivalence $\textbf{C} \xrightarrow{\sim} \textbf{R}$ from a CW $E_k$-algebra is called a \textit{CW-approximation} to
$\textbf{R}$, and a key result, \cite[Theorem 11.21]{Ek}, is that if $\textbf{R}(0) \simeq 0$ then $\textbf{R}$ admits a CW-approximation. 
Moreover, whenever $\mathds{k}$ is a field, we can construct a CW-approximation in which the number of $(x,d)$-cells needed is precisely the dimension of $H_{x,d}^{E_k}(\textbf{R})$.
By “giving the $d$-cells filtration $d$”, see \cite[Section 11]{Ek} for a more precise discussion of what this
means, one gets a skeletal filtration of this $E_k$-algebra and a spectral sequence computing the homology of $\textbf{R}$. 

In order to discuss homological stability of $E_k$-algebras we will need some preparation. 
For the rest of this section let $\textbf{R} \in \Alg_{E_k}(\mathsf{sMod}_{\mathds{k}}^{\mathsf{G}})$, where $\mathsf{G}$ is equipped with a symmetric monoidal functor $\rk: \mathsf{G} \rightarrow \mathbb{N}$; and suppose we are given a homology class $\sigma \in H_{x,0}(\textbf{R})$ with $\rk(x)=1$.
By definition $\sigma$ is a homotopy class of maps $\sigma: S_{\mathds{k}}^{x,0} \rightarrow \mathbf{R}$.

Following \cite[Section 12.2]{Ek}, there is a strictly associative algebra $\mathbf{\overline{R}}$ which is equivalent to the unitalization $\mathbf{R^+}:=\mathds{1} \oplus \textbf{R}$, where $\mathds{1}$ is the monoidal unit in simplicial modules. 
Then, $\sigma$ gives a map $\sigma \cdot-: S_{\mathds{k}}^{x,0} \otimes \mathbf{\overline{R}} \rightarrow \mathbf{\overline{R}}$ by using the associative product of $\mathbf{\overline{R}}$. 
We then define $\mathbf{\overline{R}}/\sigma$ to be the cofibre of this map.
Observe that a-priori $\sigma \cdot -$ is not a (left) $\mathbf{\overline{R}}$-module map, so the cofibre $\mathbf{\overline{R}}/\sigma$ is not a (left) $\mathbf{\overline{R}}$-module.
However, for $k \geq 2$, by the ``adapters construction'' in \cite[Section 12.2]{Ek} and its  applications in \cite[Section 12.2.3]{Ek}, there is a way of defining a cofibration sequence $S^{1,0} \otimes \mathbf{\overline{R}} \xrightarrow{\sigma \cdot - } \mathbf{\overline{R}} \rightarrow \mathbf{\overline{R}}/\sigma$ in the category of left $\mathbf{\overline{R}}$-modules in such a way that forgetting the $\mathbf{\overline{R}}$-module structure recovers the previous construction; we will make use of this fact in some of the proofs Sections \ref{section 2} and \ref{section 5.7}. 

By construction $\sigma \cdot -$ induces maps $\textbf{R}(y) \rightarrow \textbf{R}(x+y)$ between the different graded components of $\textbf{R}$ and the homology of the object $\mathbf{\overline{R}}/\sigma$ captures the relative homology of these. 
Thus, homological stability results of $\textbf{R}$ using $\sigma$ to stabilize can be reformulated as vanishing ranges for $H_{x,d}(\mathbf{\overline{R}}/\sigma)$; the advantage of doing so is that using filtrations for CW approximations of $\textbf{R}$ one also gets filtrations of $\mathbf{\overline{R}}/\sigma$ and hence spectral sequences capable of detecting vanishing ranges. 

Finally, in Sections \ref{section 2} and \ref{section 5.7}, we will also make use of some bar constructions, defined as follows:

For $\mathbf{R}$ as above, $\mathbf{M}$ a right $\mathbf{\overline{R}}$-module and $\mathbf{N}$ a left $\mathbf{\overline{R}}$-module we define the \textit{bar construction} $B(\mathbf{M},\mathbf{\overline{R}},\mathbf{N})$ to be the geometric realization of the semisimplicial object $B_{\bullet}(\mathbf{M},\mathbf{\overline{R}},\mathbf{N})$ with $p$-simplices $\mathbf{M} \otimes \mathbf{\overline{R}}^{\otimes p} \otimes \mathbf{N}$, and face maps given by using the $\mathbf{\overline{R}}$-module structures and multiplication. 
By \cite[Lemma 9.16]{Ek}, $B(\mathbf{M},\mathbf{\overline{R}},\mathbf{N})$ computes the derived tensor product $\mathbf{M} \otimes_{\mathbf{\overline{R}}}^{\mathbb{L}} \mathbf{N}$, and by \cite[Corollary 9.17, Theorem 13.7]{Ek} we can compute $E_1$-indecomposables using bar constructions via
$$Q_{\mathbb{L}}^{E_1}(\mathbf{R})\simeq B(\mathds{1},\mathbf{\overline{R}},\mathbf{R}).$$
This formula will be essential for us in order to understand the $E_k$ homology of the algebras considered in this paper. 

\subsection{Organization of the paper and outline of proof} \label{section 1.3}

In Section \ref{section 2} we will see some generic homological stability results for $E_k$-algebras. 
These results will begin with an $E_k$-algebra $\textbf{X}$ in $\mathsf{sMod}_{\mathds{k}}^{\mathsf{G}}$ for some appropriate $\mathsf{G}$, satisfying the following two types of assumptions, and will conclude homological stability results for $\textbf{X}$. 
\begin{enumerate}[(i)]
    \item The ``\textit{standard connectivity estimate}'', meaning that $H_{x,d}^{E_k}(\textbf{X})=0$ for $d<\rk(x)-1$. 
    \item Certain conditions on $H_{x,d}(\mathbf{X})$ for small values of $d$ and $\rk(x)$, chosen so that they hold for the later applications in Section \ref{section 4}. 
\end{enumerate}

This section is based on ideas of \cite{Sierra2022-st}, and in fact some of the generic stability results used in this paper are taken directly from there.

In Section \ref{section 3} we shall define certain moduli spaces of $2n$-manifolds, and use them to construct a graded $E_{2n-1}$-algebra $\textbf{R}$ in $\mathsf{Top}^{\mathsf{G}}$, where $\mathsf{G}=\mathsf{G}_n$ depends on $n$.
The monoids $\mathsf{G}_n$ come equipped with a certain symmetric monoidal functor $\rk: \mathsf{G}_n \rightarrow \mathbb{N}$ explicitly constructed. 
We will establish some basic properties of $\textbf{R}$, and interpret its path-components as classifying spaces of diffeomorphism groups fixing ``half of the boundary'' of the manifolds. 
Finally we will formulate the following.  

\begin{theoremc} \label{theorem C}
For $n \geq 3$ odd the $E_{2n-1}$-algebra $\textbf{R}$ satisfies $H_{x,d}^{E_1}(\textbf{R})=0$ for $x \in \mathsf{G}_n$ with $d<\rk(x)-1$. 
\end{theoremc}

Moreover, by Corollary \ref{cor theorem key} the same vanishing range holds for $H_{x,d}^{E_{2n-1}}(\textbf{R})$. 

This theorem is one of the main results of this paper.
It is essential in showing Theorems \hyperref[theorem A]{A}, \hyperref[theorem B]{B}, and could be used in the future to further refine those results if one can get more information about the homology groups of $\mathbf{R}$ in small bidegrees. 
However, its proof is highly technical and none of the details of it are needed to prove the stability results, so it will be delayed to Sections \ref{section 5} and \ref{section 6}. 

In Section \ref{section 4} we will prove Theorems \hyperref[theorem A]{A} and \hyperref[theorem B]{B} by applying the generic stability results of Section \ref{section 2} and some results of \cite{Sierra2022-st} to the algebras $\textbf{R}_{\mathds{k}}$. 
This will use Theorem \hyperref[theorem C]{C} and some computations of the homology groups of the moduli spaces of manifolds of Section \ref{section 3} with appropriate coefficients $\mathds{k}$.
Since the components of $\mathbf{R}$ are classifying spaces of diffeomorphism groups relative to half of the boundary and we want to understand diffeomorphism groups relative the full boundary we need to compare these two, which will be done in Section \ref{section 4.2}.
Finally in Sections \ref{section 4.5} and \ref{section 4.6} we put everything together. 

The remaining sections of the paper are devoted to proving Theorem \hyperref[theorem C]{C}. 
Section \ref{section 5} will assume a certain connectivity property of a simplicial complex called the ``\textit{arc complex}'' in analogy to the one used for surfaces in \cite[Section 4]{E2}, and will deduce Theorem \hyperref[theorem C]{C} from it. 
To do so we will begin by defining a certain poset of splittings, similar to the one appearing in \cite[Theorem 3.4]{E2}, and we will show this is highly connected by mimicking the proof of \cite[Theorem 3.4]{E2} and appealing to the high connectivity of the arc complex. 
Then we will define a certain semisimplicial space called \textit{splitting complex} and use a ``discretization technique'', inspired by the proof of \cite[Theorem 5.6]{high}, to show that the splitting complex is highly connected from the high connectivity results of the poset of splittings. 
Finally, we will relate the high connectivity of the splitting complexes to the vanishing of the $E_1$-homology of $\textbf{R}$. 

In Section \ref{section 6} we will show the high connectivity of the arc complex.
This will be done by firstly finding an algebraic analogue of the arc complex, called the \textit{algebraic arc complex}, and showing that it is highly connected. 
Finally, we will use a similar technique to the one in \cite[Lemma 5.5]{high}, based on the Whitney trick, to get the high connectivity of the arc complex from that of the algebraic arc complex. 

\section*{Acknowledgements.}
I am supported by an EPSRC PhD Studentship, grant no. 2261123, and by O. Randal-Williams’ Philip Leverhulme Prize from the Leverhulme Trust.
I would like to give special thanks to my PhD supervisor Oscar Randal-Williams for all his advice and all the helpful discussions and corrections. 
I would also like to thank Alexander Kupers and Manuel Krannich for their suggestions. 

\section{Homological stability results on $E_k$-algebras} \label{section 2}

This section is devoted to generic homological stability results about $E_k$-algebras that will apply to the $E_{2n-1}$-algebra constructed from moduli spaces of manifolds in Section \ref{section 3}.  
The first two results, Lemmas \ref{lem stab 1} and \ref{lem stab 2} are inspired by the generic homological stability result, \cite[Theorem 18.1]{Ek}, and proved in \cite{Sierra2022-st}. 
We included the statements here as they will be crucial in the proof of Theorem \hyperref[theorem A]{A}. 
The third result, Theorem \ref{theorem rational stability general}, is new and its proof uses a different approach and can be found at the end of this section. 
The statements of these results are essential for the rest of the paper but their proofs can be skipped on a first reading of the paper as they rely on technicalities of cellular $E_k$-algebras. 

Let us define $\mathsf{H}$ to be the discrete monoid $\{0\} \cup \mathbb{N}_{>0} \times \mathbb{Z}/2$, where the monoidal structure $+$ is given by addition in both coordinates. 
We denote by $\rk: \mathsf{H} \rightarrow \mathbb{N}$ the symmetric monoidal functor given by projection to the first coordinate. 

The following two lemmas correspond to \cite[Theorem 2.1, Theorem 2.2]{Sierra2022-st} respectively. 

\begin{lemma} \label{lem stab 1}
Let $\mathds{k}$ be a commutative ring and let $\mathbf{X} \in \Alg_{E_2}(\mathsf{sMod}_{\mathds{k}}^{\mathsf{H}})$ be such that $H_{0,0}(\mathbf{X})=0$, $H_{x,d}^{E_2}(\mathbf{X})=0$ for $d<\rk(x)-1$, and 
$H_{*,0}(\mathbf{\overline{X}})=\frac{\mathds{k}[\sigma_0,\sigma_1]}{(\sigma_1^2-\sigma_0^2)}$ as a ring, for some classes $\sigma_{\epsilon} \in H_{(1,\epsilon),0}(\mathbf{X})$. 
Then, for any $\epsilon \in \{0,1\}$ and any $x \in \mathsf{H}$ we have $H_{x,d}(\mathbf{\overline{X}}/\sigma_{\epsilon})=0$ for $2d \leq \rk(x)-2$.
\end{lemma}

\begin{lemma} \label{lem stab 2}
Let $\mathds{k}$ be a commutative $\mathbb{Z}[1/2]$-algebra, let $\textbf{X} \in \Alg_{E_2}(\mathsf{sMod}_{\mathds{k}}^{\mathsf{H}})$ be such that $H_{0,0}(\mathbf{X})=0$,  $H_{x,d}^{E_2}(\textbf{X})=0$ for $d<\rk(x)-1$, and  $H_{*,0}(\mathbf{\overline{X}})=\frac{\mathds{k}[\sigma_0,\sigma_1]}{(\sigma_1^2-\sigma_0^2)}$ as a ring, for some classes $\sigma_{\epsilon} \in H_{(1,\epsilon),0}(\textbf{X})$. 
Suppose in addition that for some $\epsilon \in \{0,1\}$ we have:
\begin{enumerate}[(i)]
    \item $\sigma_{\epsilon} \cdot - : H_{(1,1-\epsilon),1}(\textbf{X}) \rightarrow H_{(2,1),1}(\textbf{X})$ is surjective.
    \item $\coker(\sigma_{\epsilon} \cdot-: H_{(1,\epsilon),1}(\textbf{X}) \rightarrow H_{(2,0),1}(\textbf{X}))$ is generated by $Q_{\mathds{k}}^1(\sigma_0)$ as a $\mathbb{Z}$-module.
    \item $\sigma_{1-\epsilon} \cdot Q_{\mathds{k}}^1(\sigma_0) \in H_{(3,1-\epsilon),1}(\textbf{X})$ lies in the image of $\sigma_{\epsilon}^2 \cdot -:H_{(1,1-\epsilon),1}(\textbf{X}) \rightarrow H_{(3,1-\epsilon),1}(\textbf{X})$.
\end{enumerate}
Then $H_{x,d}(\mathbf{\overline{X}}/\sigma_{\epsilon})=0$ for $3d \leq 2 \rk(x)-4$. 
\end{lemma}

The next result will be used to deal with rational coefficients, and it says that under suitable conditions the ``\textit{standard vanishing line}'' on $E_k$-cells, i.e. $H_{x,d}^{E_k}(\textbf{X})$ for $d<\rk(x)-1$, plus some partial homological stability result, i.e. stability up to some homological degree, leads to a global homological stability result valid for all degrees. 
Before properly stating it, let us give some motivation for some of the assumptions that will appear on the theorem. 

If $\textbf{R} \in \Alg_{E_k}(\mathsf{Top}^{\mathbb{N}})$ satisfies that $\textbf{R}(0)=\emptyset$ and that $\textbf{R}(g)$ is path-connected for $g > 0$ then $H_{1,0}(\textbf{R}_{\mathbb{Q}})=\mathbb{Q}\{\sigma_0\}$ for $\sigma_0$ represented by a point in $\textbf{R}(1)$. 
This is represented by a map $S_{\mathbb{Q}}^{1,0} \rightarrow \textbf{R}_{\mathbb{Q}}$ and hence we get an $E_k$-algebra map $\mathbf{E_k}(S_{\mathbb{Q}}^{1,0} \sigma_0) \rightarrow \textbf{R}_{\mathbb{Q}}$. 
The long exact sequence in homology of the map $\sigma_0 \cdot-$ gives $H_{1,0}(\mathbf{\overline{R}}_{\mathbb{Q}}/\sigma_0)=0$.

On the other hand, let $\textbf{R} \in \Alg_{E_k}(\mathsf{Top}^{\mathsf{H}})$ be such that $\textbf{R}(0)=\emptyset$ and $\textbf{R}(g,\epsilon)$ is path-connected for any $(g,\epsilon) \in \mathbb{N}_{>0} \times \mathbb{Z}/2$.
We can also view $\textbf{R}$ as an object in $\Alg_{E_k}(\mathsf{Top}^{\mathbb{N}})$ via $R(g)=R(g,0) \sqcup R(g,1)$, see \cite[Section 3]{Ek} for details. 
Then $H_{1,0}(\textbf{R}_{\mathbb{Q}})=\mathbb{Q}\{\sigma_0,\sigma_1\}$ for $\sigma_{\epsilon}$ represented by a point in $\textbf{R}(1,\epsilon)$, so we get an $E_k$-algebra map $\mathbf{E_k}(S_{\mathbb{Q}}^{1,0} \sigma_0 \oplus S_{\mathbb{Q}}^{1,0} \sigma_1) \rightarrow \textbf{R}_{\mathbb{Q}}$. 
Moreover, by the monoidal structure of $\mathsf{H}$ both $\sigma_0^2$ and $\sigma_1^2$ lie in the same path-component $\mathbf{R}(2,0)$, so $\sigma_1^2-\sigma_0^2=0 \in H_{2,0}(\mathbf{R}_{\mathbb{Q}})$. 
Hence, we can use the cell attachment construction explained in Section \ref{section 1.2}, or \cite[Section 6]{Ek} for further details, to extend the previous map to an $E_k$-algebra map $\mathbf{E_k}(S_{\mathbb{Q}}^{1,0} \sigma_0 \oplus S_{\mathbb{Q}}^{1,0}\sigma_1) \cup_{\sigma_1^2-\sigma_0^2}^{E_k}{D_{\mathbb{Q}}^{2,1} \rho} \rightarrow \textbf{R}_{\mathbb{Q}}$. 
The long exact sequence in homology of the map $\sigma_0 \cdot -$ gives $H_{1,0}(\mathbf{\overline{R}}_{\mathbb{Q}}/\sigma_0)=\mathbb{Q}\{\sigma_1\}$.

\begin{rem} \label{rem application rational}
The above two cases are of special interest for us since they will apply to the $E_{2n-1}$-algebra \textbf{R} of Definition \ref{definition R}: the first case will apply for $n=3,7$ and the second one in all the other dimensions. 
On each of the two cases we get a map from a certain cellular $E_k$-algebra to the rationalization $\mathbf{R}_{\mathbb{Q}}$.  
Moreover, the above discussion gives information about $H_{1,0}(\mathbf{\overline{R}}_{\mathbb{Q}}/\sigma_0)$ in each of the two cases which will be useful later. 
\end{rem}

\begin{theorem} \label{theorem rational stability general}
Let $k \geq 3$, $\mathbf{A}, \mathbf{X} \in \Alg_{E_k}(\mathsf{sMod}_{\mathbb{Q}}^{\mathbb{N}})$ such that $H_{0,0}(\textbf{X})=0$, and $\mathbf{X}$ satisfies $H_{g,d}^{E_k}(\textbf{X})=0$ for $d<g-1$.
Suppose that there is an $E_k$-algebra map $\textbf{A} \rightarrow \textbf{X}$ and a $D \in \mathbb{Z}_{>0} \cup \{\infty\}$ for which one of the following two cases holds.
\begin{enumerate}[(i)]
    \item $\mathbf{A}=\mathbf{E_k}(S_{\mathbb{Q}}^{1,0} \sigma_0)$ and  $H_{g,d}(\mathbf{\overline{X}}/\sigma_0)=0$ for $d<\min\{g,D\}$
    \item $\mathbf{A}=\mathbf{E_k}(S_{\mathbb{Q}}^{1,0} \sigma_0 \oplus S_{\mathbb{Q}}^{1,0}\sigma_1) \cup_{\sigma_1^2-\sigma_0^2}^{E_k}{D_{\mathbb{Q}}^{2,1} \rho}$ and $H_{g,d}(\mathbf{\overline{X}}/\sigma_0)=0$ for $d< \min \{g,D\}$ except for $(g,d)=(1,0)$ where $H_{1,0}(\mathbf{\overline{X}}/\sigma_0)=\mathbb{Q}\{\sigma_1\}$. 
\end{enumerate}

Then in case (i) we have $H_{g,d}(\mathbf{\overline{X}}/\sigma_0)=0$ for $d<\frac{D}{D+1}g$ and in case (ii) we have $H_{g,d}(\mathbf{\overline{X}}/\sigma_0)=0$ for $d<\frac{D}{D+1}(g-1)$. 
\end{theorem}

\begin{proof}
The proof is divided into several steps: we will firstly check an universal example and then reduce the general statement to that one. 

\textbf{Step 0.}
Consider the case $\textbf{X}=\textbf{A}$, $\textbf{A}\rightarrow \textbf{X}$ given by the identity.
We want to show the conclusion of the theorem for $D=\infty$.

Consider case $(i)$: 
by \cite[Section 16]{Ek} $H_{*,*}(\overline{\mathbf{E_k}(S_{\mathbb{Q}}^{1,0} \sigma_0)})=\Lambda(L)$ is the free graded-commutative algebra on $L$, where $L$ is the free $(k-1)$-Lie algebra on $\sigma_0$. 
Explicitly, $L$ is generated as a graded $\mathbb{Q}$-vector space by $\sigma_0$ in grading $(1,0)$ and by $[\sigma_0,\sigma_0]$ in grading $(2,k-1)$. 
Thus, the map $\sigma_0 \cdot -$ acts injectively on homology, so the homology long exact sequence of the cofibration 
$$S_{\mathbb{Q}}^{1,0} \otimes \overline{\mathbf{E_k}(S_{\mathbb{Q}}^{1,0} \sigma_0)} \xrightarrow{\sigma_0 \cdot} \overline{\mathbf{E_k}(S_{\mathbb{Q}}^{1,0} \sigma_0)} \rightarrow \overline{\mathbf{E_k}(S_{\mathbb{Q}}^{1,0} \sigma_0)} /\sigma_0$$ 
gives that
$H_{*,*}(\overline{\mathbf{E_k}(S_{\mathbb{Q}}^{1,0} \sigma_0)}/\sigma_0)= \Lambda(\sigma_0,[\sigma_0,\sigma_0])/\sigma_0=\Lambda([\sigma_0,\sigma_0])$ is supported in bidegrees in the line of slope $(k-1)/2$ through the origin, and hence it satisfies the required range as $k \geq 3$. 
(We define the slope of an element to be its homological degree divided by its rank.)

For the case $(ii)$ we let $\mathbf{fA} \in \Alg_{E_k}((\mathsf{sMod}_{\mathbb{Q}}^{\mathbb{N}})^{\mathbb{Z}_{\leq}})$ be the cell attachment filtration of $\mathbf{A}$, i.e. 
$$\mathbf{fA}=\mathbf{E_k}(S_{\mathbb{Q}}^{1,0,0} \sigma_0 \oplus S_{\mathbb{Q}}^{1,0,0}\sigma_1) \cup_{\sigma_1^2-\sigma_0^2}^{E_k}{D_{\mathbb{Q}}^{2,1,1} \rho}$$
where the last grading is the filtration degree. 
This satisfies $\colim(\mathbf{fA})=\mathbf{A}$, and by \cite[Theorem 6.4]{Ek} $\gr(\mathbf{fA})=\mathbf{E_k}(S_{\mathbb{Q}}^{1,0,0} \sigma_0 \oplus S_{\mathbb{Q}}^{1,0,0}\sigma_1 \oplus S_{\mathbb{Q}}^{2,1,1} \rho)$. 

The map $\sigma_0: S_{\mathbb{Q}}^{1,0} \rightarrow \mathbf{A}$ lifts to a filtered map $\sigma_0: S_{\mathbb{Q}}^{1,0,0} \rightarrow \mathbf{fA}$, obtained from $\sigma_0: S_{\mathbb{Q}}^{1,0} \rightarrow \mathbf{fA}(0)= \mathbf{E_k}(S_{\mathbb{Q}}^{1,0} \sigma_0 \oplus S_{\mathbb{Q}}^{1,0}\sigma_1)$. 
Thus, using the adapters construction in \cite[Section 12.2]{Ek}, we can form the $\mathbf{\overline{fA}}$-module $\mathbf{\overline{fA}}/\sigma_0$ whose colimit is $\mathbf{\overline{A}}/\sigma_0$. 
\begin{enumerate}[(i)]
    \item By \cite[Corollary 10.7, Section 16.6]{Ek} there is a multiplicative spectral sequence
    $$E^1_{g,p,q}=H_{g,p+q,p}(\overline{\mathbf{E_k}(S_{\mathbb{Q}}^{1,0,0} \sigma_0 \oplus S_{\mathbb{Q}}^{1,0,0}\sigma_1 \oplus S_{\mathbb{Q}}^{2,1,1} \rho)}) \Rightarrow H_{g,p+q}(\mathbf{\overline{A}})$$
    whose $d^1$ differential satisfies $d^1(\rho)=\sigma_1^2-\sigma_0^2$. 
    \item By \cite[Theorem 10.10]{Ek} there is a spectral sequence
    $$F^1_{g,p,q}=H_{g,p+q,p}(\overline{\mathbf{E_k}(S_{\mathbb{Q}}^{1,0,0} \sigma_0 \oplus S_{\mathbb{Q}}^{1,0,0}\sigma_1 \oplus S_{\mathbb{Q}}^{2,1,1} \rho)}/\sigma_0) \Rightarrow H_{g,p+q}(\mathbf{\overline{A}}/\sigma_0).$$
\end{enumerate}
The $\mathbf{\overline{fA}}$-module structure on $\mathbf{\overline{fA}}/\sigma_0$ gives the second spectral sequence $F^t_{*,*,*}$ the structure of a module over the first one $E^t_{*,*,*}$. 

By \cite[Theorem 16.4]{Ek} $E^1_{*,*,*}$ is given as a $\mathbb{Q}$-algebra by the free graded-commutative algebra $\Lambda(L)$, where $L$ is the free $(k-1)$-Lie algebra on generators $\sigma_0,\sigma_1,\rho$. 
Since $\sigma_0 \cdot -$ is injective on $E^1_{*,*,*}$ then $F^1_{*,*,*}=\Lambda(L/\langle \sigma_0 \rangle)$ inherits the structure of a CDGA because its $d^1$ differential is determined by the one on the first spectral sequence. 
We can split $L=\mathbb{Q}\{\sigma_0\} \oplus \mathbb{Q}\{\sigma_1\} \oplus \mathbb{Q}\{\rho\} \oplus L'$, where $L'$ consists of elements whose word-length is at least 2, and satisfies that every element of $L'$ has slope $\geq 1$ (because $k \geq 3$). 

Let us introduce a filtration on the chain complex $(F^1_{*,*,*},d^1)$ by giving $\sigma_1$ and $\rho$ filtration degree 0, elements of $L'$ filtration equal to their homological degree, and then extending the filtration multiplicatively.
The differential $d^1$ preserves this filtration, and its associated graded is given by a tensor product of chain complexes 
$$(\Lambda(\sigma_1,\rho),\delta) \otimes (\Lambda(L'),0)$$
where $\delta(\sigma_1)=0$ and $\delta(\rho)=\sigma_1^2$. 
By Künneth's theorem the homology of this chain complex is $\mathbb{Q}[\sigma_1]/(\sigma_1^2) \otimes \Lambda(L')$, and so there is a spectral sequence of the form 
$$\mathbb{Q}[\sigma_1]/(\sigma_1^2) \otimes \Lambda(L') \Rightarrow H(F^1_{*,*,*},d^1)=F^2_{*,*,*}.$$
Thus, $F^2_{g,p,q}=0$ for $p+q<g-1$ because $\mathbb{Q}[\sigma_1]/(\sigma_1^2) \otimes \Lambda(L')$
already has this vanishing line, and hence $H_{g,d}(\mathbf{\overline{A}}/\sigma_0)=0$ for $d<g-1$, as required.

\textbf{Step 1.}
We will show that under the hypotheses in the statement of the theorem $H_{g,d}^{\mathbf{\overline{A}}}(\mathbf{\overline{X}}/\sigma_0)=0$ for $d < \min\{g,D\}$.

The composition $S_{\mathbb{Q}}^{1,0} \otimes \mathbf{\overline{A}} \xrightarrow{\sigma_0 \cdot -} \mathbf{\overline{A}} \rightarrow \mathbf{\overline{X}} \rightarrow \mathbf{\overline{X}}/\sigma_0$ is nullhomotopic as an $\mathbf{\overline{A}}$-module map because it agrees with $S_{\mathbb{Q}}^{1,0} \otimes \mathbf{\overline{A}} \rightarrow S_{\mathbb{Q}}^{1,0} \otimes \mathbf{\overline{X}} \xrightarrow{\sigma_0 \cdot -} \mathbf{\overline{X}} \rightarrow \mathbf{\overline{X}/\sigma_0}$ and the last two maps form a cofibration sequence. Thus there is an $\mathbf{\overline{A}}$-module map $\mathbf{\overline{A}}/\sigma_0 \rightarrow \mathbf{\overline{X}}/\sigma_0$.
(We are using adapters throughout so all the previous constructions are in the category of left $\mathbf{\overline{A}}$-modules.)

If $\mathbf{M}$ is an $\mathbf{\overline{A}}$-module then we can descendently filter it by its $\mathbb{N}$-grading, giving an associated graded $\gr(\mathbf{M})$ which is isomorphic to $\mathbf{M}$ itself in $\mathsf{sMod}_{\mathbb{Q}}^{\mathbb{N}}$, but which has the trivial $\mathbf{\overline{A}}$-module structure induced by the augmentation $\mathbf{\overline{A}} \rightarrow \mathds{1}$, c.f. \cite[Remark 19.3]{Ek}. 
This induces a filtration on $B(\mathds{1},\mathbf{\overline{A}},\mathbf{M})$ whose associated graded is $B(\mathds{1},\mathbf{\overline{A}},\mathds{1}) \otimes \gr(\mathbf{M})$ and hence there is a spectral sequence 
$$ H_{*,*}(B(\mathds{1},\mathbf{\overline{A}},\mathds{1})) \otimes H_{*,*}(\textbf{M}) \Rightarrow H_{*,*}^{\mathbf{\overline{A}}}(\mathbf{M})$$
where we are suppressing the internal grading.
Applying this to both $\mathbf{M}= \mathbf{\overline{A}}/\sigma_0$ and $\mathbf{M}=\mathbf{\overline{X}}/\sigma_0$ gives spectral sequences 
$$\mathcal{E}^1_{*,*}=H_{*,*}(B(\mathds{1},\mathbf{\overline{A}},\mathds{1})) \otimes H_{*,*}(\mathbf{\overline{A}}/\sigma_0) \Rightarrow H_{*,*}^{\mathbf{\overline{A}}}(\mathbf{\overline{A}}/\sigma_0)$$
and
$$\mathcal{F}^1_{*,*}=H_{*,*}(B(\mathds{1},\mathbf{\overline{A}},\mathds{1})) \otimes H_{*,*}(\mathbf{\overline{X}}/\sigma_0) \Rightarrow H_{*,*}^{\mathbf{\overline{A}}}(\mathbf{\overline{X}}/\sigma_0).$$
Moreover, the $\mathbf{\overline{A}}$-module map $\mathbf{\overline{A}}/\sigma_0 \rightarrow \mathbf{\overline{X}}/\sigma_0$ induces a morphism of spectral sequences $\mathcal{E}^t_{*,*} \rightarrow \mathcal{F}^t_{*,*}$. 
 
By \cite[Corollary 9.17, Theorem 13.7]{Ek} $Q_{\mathbb{L}}^{E_1}(\mathbf{A}) \simeq B(\mathds{1}, \mathbf{\overline{A}},\mathbf{A})$; and since $\mathbf{A} \rightarrow \mathbf{\overline{A}} \rightarrow \mathds{1}$ is a cofibre sequence of $\mathbf{\overline{A}}$-modules then it follows that  $B(\mathds{1}, \mathbf{\overline{A}},\mathbf{A}) \rightarrow B(\mathds{1}, \mathbf{\overline{A}},\mathbf{\overline{A}}) \simeq \mathds{1} \rightarrow B(\mathds{1}, \mathbf{\overline{A}},\mathds{1})$ is a cofibration sequence too.
Moreover, the map $B(\mathds{1}, \mathbf{\overline{A}},\mathbf{A}) \simeq Q_{\mathbb{L}}^{E_1}(\mathbf{A}) \rightarrow \mathds{1}$ is nullhomotopic by grading reasons: $\mathds{1}$ is concentrated in grading 0 whereas $Q_{\mathbb{L}}^{E_1}(\mathbf{\overline{A}})$ is 0 in grading 0. 
Thus it follows that $B(\mathds{1},\mathbf{\overline{A}},\mathds{1}) \simeq \mathds{1} \oplus S_{\mathbb{Q}}^{0,1} \otimes Q_{\mathbb{L}}^{E_1}(\textbf{A})$. 

By definition of $\textbf{A}$ and the transferring vanishing lines down theorem, \cite[Theorem 14.6]{Ek}, $H_{g,d}^{E_1}(\textbf{A})=0$ for $d<g-1$ in cases (i) and (ii); and thus $H_{g,d}(B(\mathds{1},\mathbf{\overline{A}},\mathds{1}))=0$ for $d<g$. 
In other words, in both spectral sequences the first factor of the tensor product vanishes below slope $1$. 

Thus, by assumption there cannot be elements in $\mathcal{F}^1_{*,*}$ with $d< \min \{g,D\}$ except in case (ii), in which case they are of the form $\zeta \otimes \sigma_1$ for some $\zeta \in H_{*,*}(B(\mathds{1},\mathbf{\overline{A}},\mathds{1}))$ lying on the line $d=g$. 
Thus, all such elements lift to $\mathcal{E}^1_{*,*}$ as $\sigma_1$ does. 
Moreover, $\zeta \otimes \sigma_1 \in \mathcal{E}^1_{*,*}$ is a permanent cycle because it has bidegree of the form $(d+1,d)$ for some $d<D$, and so $d^r(\zeta \otimes \sigma_1)$ has bidegree $(d+1,d-1)$, but $\mathcal{E}^1_{d+1,d-1}=0$. 
Since $Q_{\mathbb{L}}^{\mathbf{\overline{A}}}(\sigma_0 \cdot-)$ is nullhomotopic and $S_{\mathbb{Q}}^{1,0} \otimes \mathbf{\overline{A}} \xrightarrow{\sigma_0 \cdot -} \mathbf{\overline{A}} \rightarrow \mathbf{\overline{A}}/\sigma_0$ is a cofibration sequence of $\mathbf{\overline{A}}$-modules,  $Q_{\mathbb{L}}^{\mathbf{\overline{A}}}(\mathbf{\overline{A}}/\sigma_0) \simeq \mathds{1} \oplus S_{\mathbb{Q}}^{1,1}$. 
Thus $H_{*,*}^{\mathbf{\overline{A}}}(\mathbf{\overline{A}}/\sigma_0)=\mathbb{Q}[0,0]\oplus \mathbb{Q}[1,1]$, so $\mathcal{E}^{\infty}_{d+1,d}=0$ and hence all elements $\zeta \otimes \sigma_1 \in \mathcal{E}^1_{*,*}$ are boundaries. 
Using the morphism of spectral sequences and our assumptions on $H_{*,*}(\mathbf{\overline{X}}/\sigma_0)$ it follows that $\zeta \otimes \sigma_1 \in \mathcal{F}^1_{*,*}$ is also a permanent cycle and a boundary. 
Therefore, $\mathcal{F}^{\infty}_{g,d}=0$ for $d<\min\{g,D\}$, giving the result.

\textbf{Step 2.}
We will show that $H_{g,d}^{\mathbf{\overline{A}}}(\mathbf{\overline{X}})=0$ for $d < \min\{g,D\}$.

By \cite[Section 12.2.4]{Ek}, $\mathbf{\overline{X}}/\sigma_0 \simeq \mathbf{\overline{A}}/\sigma_0 \otimes_{\mathbf{\overline{A}}}^{\mathbb{L}} \mathbf{\overline{X}}$ as left $\mathbf{\overline{A}}$-modules. 
Also, in the proof of Step 1 we showed that $Q_{\mathbb{L}}^{\mathbf{\overline{A}}}(\mathbf{\overline{A}}/\sigma_0) \simeq \mathds{1} \oplus S_{\mathbb{Q}}^{1,1}$. 
By \cite[Corollary 9.17]{Ek} $Q_{\mathbb{L}}^{\mathbf{\overline{A}}}(\mathbf{\overline{X}}/\sigma_0)= \mathds{1} \otimes_{\mathbf{\overline{A}}}^{\mathbb{L}} \mathbf{\overline{X}}/\sigma_0$ and  $Q_{\mathbb{L}}^{\mathbf{\overline{A}}}(\mathbf{\overline{X}})= \mathds{1} \otimes_{\mathbf{\overline{A}}}^{\mathbb{L}} \mathbf{\overline{X}}$. 
Thus, 
$$Q_{\mathbb{L}}^{\mathbf{\overline{A}}}(\mathbf{\overline{X}}/\sigma_0)= (\mathds{1} \oplus S_{\mathbb{Q}}^{1,1}) \otimes Q_{\mathbb{L}}^{\mathbf{\overline{A}}}(\mathbf{\overline{X}})$$
which together with Step 1 implies the required result. 

\textbf{Step 3.}
We will show that $H_{g,d}^{E_k}(\textbf{X},\textbf{A})=0$ for $d< \min\{g,D\}$.

Let $\delta:= \min \{d: H_{d+1,d}^{E_k}(\textbf{X},\textbf{A}) \neq 0\} \in \mathbb{N} \cup \{\infty\}$. 
Our goal is to show that $\delta \geq D$. 

Since $H_{g,d}^{E_k}(\textbf{X})=0=H_{g,d}^{E_k}(\textbf{A})$ for $d<g-1$ it follows that  $H_{g,d}^{E_k}(\textbf{X},\textbf{A})=0$ for $d<g-1$ and $d<\frac{\delta}{\delta+1} g$.
By \cite[Theorem 15.9]{Ek} with $\rho(g)=\frac{\delta}{\delta+1} g= \sigma(g)$, the map $H_{g,d}^{\mathbf{\overline{A}}}(\mathbf{\overline{X}},\mathbf{\overline{A}}) \rightarrow H_{g,d}^{E_k}(\textbf{X},\textbf{A})$ is surjective for $d<\frac{\delta}{\delta+1} g+1$. 
Thus, $H_{\delta+1,\delta}^{\mathbf{\overline{A}}}(\mathbf{\overline{X}}) \xrightarrow{\cong} H_{\delta+1,\delta}^{\mathbf{\overline{A}}}(\mathbf{\overline{X}},\mathbf{\overline{A}}) \rightarrow H_{\delta+1,\delta}^{E_k}(\mathbf{X},\mathbf{A}) \neq 0$ is a surjection, and hence  $H_{\delta+1,\delta}^{\mathbf{\overline{A}}}(\mathbf{\overline{X}}) \neq 0$. 
By Step 2 we have $\delta \geq D$, as required. 

\textbf{Step 4.}
Now we can finish the proof.

By Step 3 and the vanishing lines in $E_k$-homology for both $\textbf{A}$ and $\textbf{X}$ we get $H_{g,d}^{E_k}(\textbf{X},\textbf{A})=0$ for $d<\frac{D}{D+1}g$.
Applying \cite[Corollary 15.10]{Ek} with $\rho(g)=\frac{D}{D+1}g$, $\mathbf{M}=\mathbf{\overline{A}}/\sigma_0$, and $\mu(g)$ either $\frac{D}{D+1}g$ in case (i) or $\frac{D}{D+1}(g-1)$ in case (ii); we find that $H_{g,d}(B(\mathbf{\overline{A}}/\sigma_0,\mathbf{\overline{A}},\mathbf{\overline{X}}))=0$ for $d<\mu(g)$. 
Finally, \cite[Section 12.2.4, Lemma 9.16]{Ek} gives that $B(\mathbf{\overline{A}}/\sigma_0,\mathbf{\overline{A}},\mathbf{\overline{X}})\simeq \mathbf{\overline{X}}/\sigma_0$.
\end{proof}

\section{Spaces of manifolds and $E_{2n-1}$-algebras} \label{section 3}

\subsection{Spaces of manifolds} \label{section 3.1}

In this subsection we define the moduli spaces of manifolds that we will study.
Before doing so let us fix some notation: for each $s \in \mathbb{N}$ let us denote 
$$J_s:=\partial I^{s+1} \setminus \Big(\int(I^{s} \times \{1\}) \cup \int(\{1\} \times I^{s})\Big) \subset \partial I^{s+1}$$ 
where $I$ is the unit interval $[0,1]$. 
This represents the boundary of the standard $(s+1)$-cube once we remove the interiors of the top and right faces. 

From now on let $n$ denote an odd natural number with $n \geq 3$. 

\begin{definition} \label{definition moduli}
Let $A \subset I^{2n-1} \times \mathbb{R}^{\infty}$ be a smooth, $(n-2)$-connected, $(2n-1)$-manifold with corners such that $A \cap (\partial I^{2n-1} \times \mathbb{R}^{\infty})=\partial A= \partial{I^{2n-1}} \times \{0\}$, and $A$ agrees with $I^{2n-1} \times \{0\}$ in a neighbourhood of its boundary.

Then let $\mathcal{M}[A]$ be the set of all $W \subset I^{2n} \times \mathbb{R}^{\infty}$ which are smooth $2n$-submanifolds with corners such that:

\begin{enumerate}[(i)]
    \item $W$ is $(n-1)$-connected and s-parallelizable.
    \item $W \cap (\partial I^{2n} \times \mathbb{R}^{\infty})= \partial W \supset J_{2n-1}$ and $W$ agrees with $I^{2n} \times \{0\}$ in a neighbourhood of $J_{2n-1}$. 
    \item $\partial W \cap (\{1\} \times I^{2n-1} \times \mathbb{R}^{\infty})=\{1\} \times A$. 
    \item $\mathcal{D}(W):= \partial W \cap (I^{2n-1} \times \{1\} \times \mathbb{R}^{\infty}) \subset I^{2n-1} \times \{1\} \times \mathbb{R}^{\infty}$ is a smooth $(2n-1)$-submanifold with corners such that
    \begin{enumerate}
        \item $ \mathcal{D}(W) \cap (\partial I^{2n-1} \times \{1\} \times \mathbb{R}^{\infty})= \partial \mathcal{D}(W)= \partial I^{2n-1} \times \{1\} \times \{0\}$ and $\mathcal{D}(W)$ agrees with $I^{2n-1} \times \{1\} \times \{0\}$ in a neighbourhood of its boundary. 
        \item $\mathcal{D}(W)$ is contractible (hence diffeomorphic to $I^{2n-1}$ by the $h$-cobordism theorem). 
    \end{enumerate}
    \item We have a decomposition $\partial W= \partial^-W \cup \mathcal{D}(W)$, where $\partial^-W:= J_{2n-1} \cup \{1\} \times A$.
    
    \item $W$ has a \textit{product structure near its boundary}, in the following sense: 
    for each $(2n-1)$-face $F$ of $I^{2n}$ let $\pi_F: I^{2n} \rightarrow F$ be the orthogonal projection onto that face. 
    Then $W$ agrees with $\{(u,x) \in I^{2n} \times \mathbb{R}^{\infty}: (\pi_F(u),x) \in \partial W \}$
    in a neighbourhood of $F \times \mathbb{R}^{\infty}$.
\end{enumerate}

   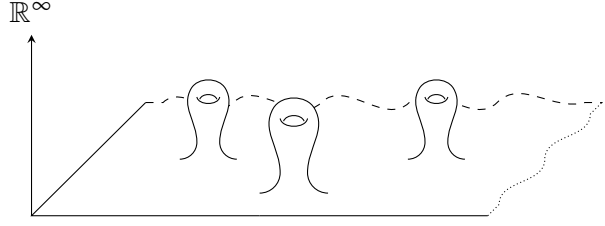
\begin{figure}[H]
      \begin{center}
\begin{tikzpicture}[scale=3]
	\begin{scope}[xshift=.885cm,yshift=.115cm,scale=1.09]
    \draw[dashed,rotate=45,scale=0.5] (0,1)
    -- (.1,.9) 
    to[out=0, in=180] (.5,.5)
    to[out=0, in=180] (1,0)
    to[out=0, in=180] (1.5,-.5)
    to[out=0, in=180] (2,-1)
    to[out=0, in=150] (2.5,-1.5)
    --
    (2.6,-1.6)
    ;
    \end{scope}

    \draw  (1,0) -- (0,0) -- (0.5,0.5) ;
	\churro[scale=0.25, x=0.65, y=0.25]
 \churro[scale=0.3, x=1, y=0.1]
    
	\begin{scope}[xshift=1cm]
    \draw (1,0) -- (0,0);
	\churro[scale=0.25, x=0.65, y=0.25]
    \end{scope}

	\begin{scope}[xshift=2.5cm]
    \draw[densely dotted, rotate=90,scale=0.5] (0,1)
    -- (.1,.9) 
    to[out=0, in=180] (.5,.5)
    to[out=0, in=180] (.9,.1)
    -- (1,0);
    \end{scope}

   \draw[-stealth] (0,0)-- (0,.8); 
    \node at (0,.9) {$\mathbb{R}^{\infty}$};
    
\end{tikzpicture}

    \end{center}
    \caption{Picture for $n=1$: the dotted edge represents $A \subset \{1\} \times I^{2n-1} \times \mathbb{R}^{\infty}$, the dashed edge represents a given $\mathcal{D}(W) \subset I^{2n-1} \times \{1\} \times \mathbb{R}^{\infty}$, and $J_1$ consists of the remaining edges together with the vertex where the dashed and dotted edges meet.}
\end{figure}

The topology of $\mathcal{M}[A]$ is defined in two steps as follows. 

Firstly, for each $W \in \mathcal{M}[A]$ we let
$\Emb_{\partial^- W}^p(W, I^{2n} \times \mathbb{R}^{\infty})$
be the space, with the $C^{\infty}$-topology, of smooth embeddings of $W$ into $I^{2n} \times \mathbb{R}^{\infty}$, sending the interior of $W$ to $\int(I^{2n}) \times \mathbb{R}^{\infty}$, fixing pointwise a neighbourhood of $\partial^- W$, and with a product structure near $\partial W$, in the following sense:
for each $\epsilon \in \Emb_{\partial^- W}^p(W, I^{2n} \times \mathbb{R}^{\infty})$ and each $(2n-1)$-face $F$ of $I^{2n}$ we have 
$$\epsilon(\partial W \cap (F \times \mathbb{R}^{\infty}))= \partial \epsilon(W) \cap (F \times \mathbb{R}^{\infty})$$
and there is a small neighbourhood of $\partial W \cap (F \times \mathbb{R}^{\infty})$ in $W$ such that
$\epsilon(u,x)=(p_1(\epsilon(\pi_F(u),x))+u-\pi_F(u), p_2(\epsilon(\pi_F(u),x)))$ whenever $(u,x) \in I^{2n} \times \mathbb{R}^{\infty}$ lies in that neighbourhood, where $p_1, p_2$ are the projections of $I^{2n} \times \mathbb{R}^{\infty}$ to each of its factors. 

Then we get a function of sets 
$$\Emb_{\partial^- W}^p(W, I^{2n} \times \mathbb{R}^{\infty}) \rightarrow \mathcal{M}[A]$$
via $\epsilon \mapsto \epsilon(W)$; and we give $\mathcal{M}[A]$ the finest topology so that all these functions are continuous. 
\end{definition}

The most natural moduli space to consider is the collection of manifolds inside $I^{2n} \times \mathbb{R}^{\infty}$ agreeing pointwise with $I^{2n} \times \{0\}$ near their boundary, plus some conditions on the manifolds such as the $(n-1)$-connectivity and the s-parallelizability. 
Such a moduli space has a natural $E_{2n}$-algebra structure by ``re-scaling and gluing the manifolds inside the cube''. 

However, in later proofs in Section \ref{section 5} we will need to ``cut arcs of the manifolds'' and study the leftover pieces, so we cannot restrict ourselves to manifolds whose boundary is $\partial I^{2n}$. 
Instead, the only control we will have on their boundary is the $(n-2)$-connectivity. 
This is similar to the fact that \cite[Section 4]{E2} introduces surfaces with more boundary components in order to understand the ones with just one boundary component. 
The need of spaces of manifolds whose boundary is more general than just the standard disc motivates the introduction of the face $A$ in the above definition.

On the other hand, by cutting arcs we can get to manifolds whose boundary is an exotic sphere, and it will be convenient that these form part of the $E_k$-algebra too: see Remark \ref{remark D(W)} for details. 
Hence, we introduced the further freedom of $\mathcal{D}(W)$, so that even when $A$ is the standard face we allow exotic spheres as boundaries and still get an $E_k$-algebra.
This new freedom must be ``movable'', in the sense that $\mathcal{D}(W)$ depends on $W$ itself, in order to get an $E_k$-algebra structure in a natural way. 

In Section \ref{section 3.3} we will give an interpretation of the moduli spaces $\mathcal{M}[A]$ as classifying spaces of diffeomorphism groups fixing some part of the boundary, and in Section \ref{section 4.2} we will relate these to the classifying spaces of diffeomorphism groups relative the full boundary that we are originally interested in.

\subsection{$\Mhalf$ as a (graded) $E_{2n-1}$-algebra} \label{section 3.2}

In this subsection we will give an $E_{2n-1}$-algebra structure to $\Mhalf \in \mathsf{Top}$. 
This will be, up to a small modification, the $E_{2n-1}$-algebra that we will study for the rest of the paper. 

To do so recall the explicit model of the $E_{2n-1}$-operad in Section \ref{section 1.2} given by the space of rectilinear embeddings of little $(2n-1)$-cubes into the standard $I^{2n-1}$. 
The $E_{2n-1}$-algebra structure on $\Mhalf$ is then given by scaling the manifolds using the rectilinear embeddings, noting that any rectilinear embedding $e: I^{2n-1} \hookrightarrow I^{2n-1}$ gives a canonical embedding $e \times I \times \mathbb{R}^{\infty}: I^{2n}  \times \mathbb{R}^{\infty} \hookrightarrow I^{2n}  \times \mathbb{R}^{\infty}$, and then re-gluing them inside $I^{2n} \times \mathbb{R}^{\infty}$; so that outside the little cubes the manifold still agrees with $I^{2n} \times \{0\}$. 
The resulting manifold will then satisfy all the conditions of Definition \ref{definition moduli}, including the s-parallelizability: for each manifold $W$ in $\Mhalf$ we can pick a stable framing which agrees with a given one on a neighbourhood of $\partial^-W \cong D^{2n-1}$, and hence we can glue these stable framings along the above $E_{2n-1}$-product to get a stable framing. 

We want to grade this $E_{2n-1}$-algebra in order to keep track of the different components, so the natural discrete monoid to take is $\mathsf{G}_n:=\pi_0(\Mhalf)$, where the (symmetric) monoidal structure is given by the $E_{2n-1}$-multiplication. 
Thus we can view
$$\Mhalf \in \Alg_{E_{2n-1}}\big(\mathsf{Top}^{\mathsf{G}_n}\big).$$

We will describe $\mathsf{G}_n$ explicitly in Section \ref{section 3.4}. 

\subsection{Interpretation of the moduli spaces as classifying spaces of diffeomorphisms} \label{section 3.3}

For the rest of this subsection we give all the embedding and diffeomorphism spaces the $C^{\infty}$-topology. 
Also, $W$ will always be assumed to be a manifold belonging to one of the spaces $\mathcal{M}[A]$ of Definition \ref{definition moduli}. 

Let us denote by $\Diff_{\partial^- W}(W)$ the group of diffeomorphisms of $W$ fixing pointwise a neighbourhood of $\partial^- W$. 
We say that a diffeomorphism $\phi \in \Diff_{\partial^-W}(W)$ has a \textit{product structure} near the boundary if when composed with the inclusion $W \hookrightarrow I^{2n} \times \mathbb{R}^{\infty}$ it has a product structure in the sense of Definition \ref{definition moduli}.
We denote by $\Diff_{\partial^- W}^p(W)$ the group of diffeomorphisms of $W$ fixing pointwise a neighbourhood of $\partial^- W$ and with a product structure near $\partial W$.

\begin{proposition} \label{prop classifying space}
The path-component of $W \in \mathcal{M}[A]$ agrees with the image of $\Emb_{\partial^- W}^p(W, I^{2n} \times \mathbb{R}^{\infty})$ inside $\mathcal{M}[A]$, and moreover it is a model of $B\Diff_{\partial^- W}^p(W)$.
Thus we can identify 
$$\mathcal{M}[A]= \bigsqcup_{[W]}{B\Diff_{\partial^- W}^p(W)}$$
where the coproduct is taken over path-components. 
\end{proposition}

\begin{proof}
The definition of the topology in $\mathcal{M}[A]$ in Definition \ref{definition moduli} forces the path-component of $W$ to lie in the image of $\Emb_{\partial^- W}^p(W, I^{2n} \times \mathbb{R}^{\infty}) \rightarrow \mathcal{M}[A]$, and since $\Emb_{\partial^- W}^p(W, I^{2n} \times \mathbb{R}^{\infty})$ is path-connected by the Whitney embedding theorem then the result follows. 

By adapting the proof of \cite[Chapter 13]{minchor} insisting that all the maps involved have a product structure, the quotient map 
$$\Emb_{\partial^- W}^p(W, I^{2n} \times \mathbb{R}^{\infty}) \rightarrow \Emb_{\partial^- W}^p(W, I^{2n} \times \mathbb{R}^{\infty})/\Diff_{\partial^- W}^p(W)$$
has slices and hence is a principal $\Diff_{\partial^- W}^p(W)$-bundle.
Moreover, its base is canonically homeomorphic to the image of $\Emb_{\partial^- W}^p(W, I^{2n} \times \mathbb{R}^{\infty}) \rightarrow \mathcal{M}[A]$. 
The result follows since $\Emb_{\partial^- W}^p(W, I^{2n} \times \mathbb{R}^{\infty})$ is weakly contractible by the Whitney embedding theorem.
\end{proof}

The following lemma says that up to homotopy the product structure plays no role. However, it is convenient to have it in our moduli spaces for technical purposes. 

\begin{lemma} \label{lem classifying space}
For any $W$ the inclusion $\Diff_{\partial^- W}^p(W) \hookrightarrow \Diff_{\partial^- W}(W)$ is a homotopy equivalence. 
Thus 
$$\mathcal{M}[A] \simeq \bigsqcup_{[W]}{B\Diff_{\partial^- W}(W)}$$
\end{lemma}

\begin{proof}
We can construct a deformation-retraction $\Diff_{\partial^-W}(W) \rightarrow \Diff_{\partial^-W}^p(W)$ as follows: 
firstly observe that any $\phi \in \Diff_{\partial^-W}(W)$ automatically preserves the product structure everywhere except possibly at the top face. 
Secondly, the product structure of $W$ at its top face gives a collar $c: \mathcal{D}(W) \times [0,\epsilon] \hookrightarrow W$, and we pick a Riemannian metric on $\partial W$.
Finally, we use geodesic interpolation to modify $\phi$ inside the collar until it becomes product-preserving.
\end{proof}

For the particular case that $A=I^{2n-1}$ we shall denote $\Diff_{\frac{1}{2}\partial}(W):=\Diff_{\partial^- W}(W)$, because $\partial W$ splits as the union of two contractible pieces: $\partial^-W$ and $\mathcal{D}(W)$. 
We will think of $\partial^-W$ as the ``\textit{fixed half of the boundary}'' because it is the standard part of it, whereas the other piece depends on the choice of $W$. 
This notation and point of view will be used in the whole Section \ref{section 4}.

\subsection{Identifying the grading} \label{section 3.4}
In this subsection we will determine $\mathsf{G}_n=\pi_0(\Mhalf)$ explicitly, which explains the choice of grading monoids of Section \ref{section 2}. 

Any $W \in \Mhalf$ can be canonically identified with a smooth manifold $W^s$ by canonically smoothing the corners of the standard cube $I^{2n}$. 
$W$ and $W^s$ are homeomorphic and have diffeomorphic interiors.
$W^s$ is s-parallelizable, $(n-1)$-connected and its boundary is a homotopy sphere. 
By \cite[pages 165-167]{wall} we can associate to any such manifold $W^s$ a triple $(H_n({W^s}),\lambda_{W^s},q_{W^s})$, where
$H_n({W^s})$ denotes its middle homology; $\lambda_{W^s}: H_n({W^s}) \otimes H_n({W^s}) \rightarrow \mathbb{Z}$ its intersection product; and $q_{W^s}: H_n({W^s}) \rightarrow \mathbb{Z}/\Lambda_n$, 
its quadratic refinement, defined in terms of normal bundles using the stable parallelizability of $W^s$, where $\Lambda_n= \left\{ \begin{array}{lcc}
             \mathbb{Z} &   if  \; n=3,7
             \\ \mathbb{Z}/2 &  \text{otherwise.} \\
             \end{array}
   \right.$
   
Since $H_n(W)=H_n(W^s)$ we get an equivalent triple $(H_n(W),\lambda_W,q_W)$. 
By Poincaré duality and the odd parity of $n$ the bilinear form $(H_n({W}),\lambda_{W})$ is skew-symmetric and non-degenerate on a free finitely generated $\mathbb{Z}$-module, and hence it is isomorphic to the standard hyperbolic form of genus $g$ for a unique $g=g(W) \in \mathbb{N}$ called the \textit{genus} of $W$. 
When $n=3,7$ the quadratic refinement vanishes identically, whereas when $n \neq 3,7$ it takes values in $\mathbb{Z}/2$. 
For $g \geq 1$ there are precisely two isomorphism classes of $\mathbb{Z}/2$-valued quadratic refinements on the standard hyperbolic form of genus $g$, distinguished by the  so called Arf invariant. 
We define the \textit{Arf invariant} of $W$, denoted $\Arf(W) \in \{0,1\}$, to the the Arf invariant of the corresponding $q_{W}$. 
When $g=0$ the Arf invariant is just set to be $0$. 

\begin{proposition} \label{prop path components}
For $n=3,7$ the genus gives an isomorphism of monoids $\mathsf{G}_n \cong \mathbb{N}$. 
For $n \neq 3,7$ odd, $n \geq 3$, taking genus and Arf invariant gives an isomorphism of monoids $\mathsf{G}_n \cong \mathsf{H}= \{0\} \bigsqcup \mathbb{N}_{>0} \times \mathbb{Z}/2$. 
\end{proposition}

\begin{proof}
Let us denote by $\mathsf{G}'_n$ the monoid $\mathbb{N}$ when $n=3,7$ and the monoid $\mathsf{H}$ otherwise. 
For each $n$ the genus and Arf invariant give a function $\Psi: \Mhalf \rightarrow \mathsf{G}'_n$. 
Since both the genus and the Arf invariant are diffeomorphism invariants then $\Psi$ factors through $\mathsf{G}_n= \pi_0(\Mhalf)$ by Proposition \ref{prop classifying space}. 
Moreover $\Psi$ is monoidal because both the genus and the Arf invariant are additive under orthogonal direct sums, and the $E_{2n-1}$-product of manifolds induces orthogonal direct sum on their associated algebraic data. 

Surjectivity: given any $x \in \mathsf{G}'_n$ we can find a corresponding quadratic form $(H,\lambda,q)$ of the correct genus and Arf invariant. 
By \cite[page 168]{wall} we can find a smooth, $(n-1)$-connected, $s$-parallelizable $2n$-manifold $W^s$ with $(H_n({W^s}),\lambda_{W^s},q_{W^s}) \cong (H,\lambda,q)$. 
The non-degeneracy of $H$ gives by Poincaré duality that $\partial W^s$ is a homotopy sphere because $\partial W^s$ is $1$-connected by construction. 
Any homotopy sphere can be realized as the canonical smoothing of a smooth $(2n-1)$-dimensional manifold with corners $\Sigma \subset \partial I^{2n} \times \mathbb{R}^{\infty}$ such that $\Sigma$ agrees with $\partial I^{2n} \times \{0\}$ in a neighbourhood of $(J_{2n-1} \cup \{1\} \times I^{2n-1}) \times \mathbb{R}^{\infty}$ and $\Sigma \cap (I^{2n-1} \times \{1\} \times \mathbb{R}^{\infty})$ is contractible. 
By the isotopy extension theorem we can then realize $W^s$ as the canonical smoothing of a manifold $W \in \Mhalf$. 

Injectivity: suppose $W, W' \in \Mhalf$ satisfy $\Psi(W)=\Psi(W')$, we will show that $[W]=[W'] \in \pi_0(\Mhalf)$. 
By definition both $W$ and $W'$ agree with $I^{2n} \times \{0\}$ on a neighbourhood $U$ of $\partial^-(W)=\partial^-(W')=J_{2n-1} \cup \{1\} \times I^{2n-1}$ in $I^{2n} \times \mathbb{R}^{\infty}$. 
We can then write both $W$ and $W'$ as $U \cap (I^{2n} \times \{0\})$ with $2g$ $n$-handles attached.
By the same argument of \cite[page 166]{wall} we can find an embedding $e: W' \hookrightarrow W$ which is the identity on $U \cap (I^{2n} \times \{0\})$: this uses that both forms $(H_n({W^s}),\lambda_{W^s},q_{W^s})$ and $(H_n({W'^s}),\lambda_{W'^s},q_{W'^s})$ are isomorphic, which holds because $\Psi(W)=\Psi(W')$ and the genus and Arf invariant completely classify non-degenerate skew-symmetric quadratic forms.  
By Whitehead's theorem the difference $\overline{W \setminus \im(e)}$ is a $h$-cobordism between the interiors of $\mathcal{D}(W)$ and $\mathcal{D}(W')$, and hence it is trivial by the contractibility of the ends and the $h$-cobordism theorem. 
We then fix a trivialization of the $h$-cobordism relative to the internal boundary and use the trivialization to push this boundary to the external one, giving a diffeomorphism $W' \xrightarrow{e'} W$ which is the identity on a neighbourhood of $\partial^-(W)=\partial^-(W')=J_{2n-1} \cup \{1\} \times I^{2n-1}$ in $I^{2n} \times \mathbb{R}^{\infty}$. 
Observe that $e'$ respects the product structure except possibly near $\mathcal{D}(W')$.
By an argument similar to the one of Lemma \ref{lem classifying space} we can use geodesic interpolation to produce a diffeomorphism $\Tilde{e}: W' \rightarrow W$ which preserves the product structure near the boundary and is the identity near $\partial^-(W)=\partial^-(W')$.

Thus, $W \in \Mhalf$ lies in the image of $\Emb_{\partial^- W'}^p(W', I^{2n} \times \mathbb{R}^{\infty})$. 
Since the later space is path-connected by the Whitney embedding theorem then $W$ and $W'$ lie on the same path-component. 
\end{proof}

\begin{rem}\label{rem representatives of path components}
We can view the manifold $W_{g,1}$ as (the smoothing of) an element in $\Mhalf$. 
Since $W_{g,1}$ has genus $g$ and Arf invariant $0$ (when defined) then:
\begin{enumerate}[(i)]
    \item When $n=3,7$ each path-component of $\Mhalf$ is represented by a $W_{g,1}$. 
    \item When $n \neq 3,7$ the path-component $(g,0) \in \mathsf{G}_n$ is represented by $W_{g,1}$. 
\end{enumerate}
\end{rem}

\subsection{The graded $E_{2n-1}$-algebra $\textbf{R}$} \label{section 3.5}

We should think of $\Mhalf$ as the $E_{2n-1}$-algebra that we study, but the CW-approximation theorem of $E_k$-algebras, \cite[Theorem 11.21]{Ek}, requires us to work with $E_k$-algebras in graded simplicial modules which are 0 in grading 0. 
In order to achieve that, we would like our algebra in spaces to be empty in grading 0.

\begin{definition}\label{definition R}
We define $\textbf{R} \in \Alg_{E_{2n-1}}(\mathsf{Top}^{\mathsf{G}_n})$ to be $\{W \in \Mhalf: g(W)>0 \}$, i.e. the subspace of manifolds in $\Mhalf$ of strictly positive genus, so that by definition $\textbf{R}(0)=\emptyset$. 
\end{definition}

By construction, the associative unital replacement $\mathbf{\overline{R}}(x)$ is path-connected  $\forall x \in \mathsf{G}_n$, and then using the monoidal structure of $\mathsf{G}_n$ of Proposition \ref{prop path components} we get

\begin{corollary}\label{cor 0 connected}
$H_{*,0}(\mathbf{\overline{R}})$ is given (as a a ring) by 
\begin{enumerate}[(i)]
    \item $\mathbb{Z}[\sigma]$, where $\sigma \in H_0(\mathbf{R}(1);\mathbb{Z})$ is the canonical generator, if $n=3,7$. 
    \item $\frac{\mathbb{Z}[\sigma_0,\sigma_1]}{(\sigma_1^2-\sigma_0^2)}$, where $\sigma_{\epsilon} \in H_0(\mathbf{R}(1,\epsilon);\mathbb{Z})$ are the canonical generators, for $n \neq 3,7$. 
\end{enumerate}
\end{corollary}

The key property that $\textbf{R}$ has and that allows the study of its homological stability properties is an a-priori vanishing line on its $E_{2n-1}$-cells.
We will denote by $\rk: \mathsf{G}_n \rightarrow \mathbb{N}$ the obvious monoidal functor taking the genus. 

\begin{theoremc}
The $E_{2n-1}$-algebra $\textbf{R}$ satisfies $H_{x,d}^{E_1}(\textbf{R})=0$ for $x \in \mathsf{G}_n$ with $d<\rk(x)-1$.
\end{theoremc}

\begin{corollary}\label{cor theorem key}
$\textbf{R}$ satisfies the standard connectivity estimate, i.e. that
$H_{x,d}^{E_{2n-1}}(\textbf{R})=0$ for $x \in \mathsf{G}_n$ with $d<\rk(x)-1$. 
Moreover, the algebras $\mathbf{R_{\mathds{k}}}$ also satisfy the standard connectivity estimate. 
\end{corollary}

\begin{proof}
It is an immediate consequence of Theorem \hyperref[theorem C]{C} and \cite[Theorem 14.4]{Ek}. 
The moreover part follows from \cite[Lemma 18.2]{Ek}.
\end{proof}

The proof of this theorem will be delayed until Sections \ref{section 5} and \ref{section 6}, once we have already deduced the main homological stability results from it.

\section{Homological stability results} \label{section 4}

The goal of this section is to show the homological stability results of Theorems \hyperref[theorem A]{A} and \hyperref[theorem B]{B}. 
In Sections \ref{section 4.1}, \ref{section 4.2} we will compare homological stability of diffeomorphism groups relative the full boundary to homological stability of $\mathbf{R}$, in Sections \ref{section 4.3}, \ref{section 4.4} we will do some homology computations, and finally in Sections \ref{section 4.5}, \ref{section 4.6} we put everything together using the results of Section \ref{section 2}. 

Let us begin by giving the precise definition of \textit{stabilization maps} that will be used for the rest of this section. 
Given $W_i \in \Mhalf$, $i \in \{0,1\}$ we can use the $E_1$-algebra structure on $\Mhalf$ and the element $[0,1/2] \sqcup [1/2,1] \in \mathcal{C}_1(2)$ to define the \textit{boundary connected sum} operation, which produces a new manifold $W_0 \natural W_1 \in \Mhalf$.

\begin{definition}\label{def stab}
For a fixed $W_0 \in \textbf{R}$ the \textit{(left) stabilization maps corresponding to $W_0$} are 
\begin{enumerate}[(i)]
    \item $B(\id_{W_0} \natural -): B \Diffhalf(W_1) \rightarrow B \Diffhalf(W_0 \natural W_1)$
    \item $B(\id_{W_0} \natural -): B \Diff_{\partial}(W_1) \rightarrow B \Diff_{\partial}(W_0 \natural W_1)$
\end{enumerate}
defined for any $W_1 \in \Mhalf$. 
\end{definition}

\begin{rem} \label{comparison stab R half bry}
Under the identification $\textbf{R} \simeq \bigsqcup_{[W]}{B \Diffhalf(W)}$ of Lemma \ref{lem classifying space}, stabilization map (i) agrees with the stabilization map corresponding to $W_0 \in \textbf{R}$ using the $E_{2n-1}$-algebra structure in the sense of Section \ref{section 1.2}.
\end{rem}

\subsection{A technical result} \label{section 4.1}

In order to compare the two notions of stabilisation maps of Definition \ref{def stab} we need to understand $B \Diff_{\partial}(W) \rightarrow B \Diffhalf(W)$ for an arbitrary $W \in \Mhalf$. 
The following result will be used throughout the whole of Section \ref{section 4}. 

\begin{theorem}
\label{theorem fibration}
If $W \in \Mhalf$ then
$$B \Diff_{\partial}(W) \xrightarrow{B \incl} B \Diffhalf(W) \xrightarrow{B\res} B \Diff_{\partial,0}(\mathcal{D}(W))$$
is a homotopy fibration sequence, where $\Diff_{\partial,0}(\mathcal{D}(W)) \subset \Diff_{\partial}(\mathcal{D}(W))$ denotes the path-component of the identity. 
\end{theorem}

\begin{proof}
The main step is to show that for $W \in \Mhalf$ the image of 
$$\pi_0(\res): \pi_0(\Diffhalf(W)) \rightarrow \pi_0(\Diff_{\partial}(\mathcal{D}(W)))$$
is trivial.

Assuming the above, $B\res: B \Diffhalf(W) \rightarrow B \Diff_{\partial}(\mathcal{D}(W))$ induces a map $B\res: B \Diffhalf(W) \rightarrow B \Diff_{\partial,0}(\mathcal{D}(W))$.
By the isotopy extension theorem the map $\Diffhalf(W) \xrightarrow{\res} \Diff_{\partial,0}(\mathcal{D}(W))$ is a Serre fibration, and its fibre over $\id$ is  homotopy equivalent to $\Diff_{\partial}(W)$. 
Taking classifying spaces on the corresponding fibration of groups gives the result. 

Now let us show the main step: given $\phi \in \Diffhalf(W)$,  we need to prove that $\phi|_{\mathcal{D}(W)}$ is isotopic to the identity relative to $\partial \mathcal{D}(W)$. 
Let $M:= W \natural W \in \Mhalf$, $\Phi:= \phi \natural \id_{W} \in \Diff(M)$, and $\mathbb{D}:=\overline{\partial M \setminus \mathcal{D}(W_0)}$, where $W_0$ denotes the left copy of $W$:
\begin{figure}[H]
    \begin{center}
    \begin{tikzpicture}[xscale=0.75, yscale=1.5] 
		\draw[densely dotted] 
		(2,0) -- (0,0) -- (0,1.5);
		\huequito[x=1,y=.8]
		 \draw (2,1.5) -- (1.8,1.5) 
		 .. controls (1.3,.5) and (.8,2.5) .. 
		 (.2,1.5) -- (0,1.5);
		 
	\draw[dashed] (2,0) -- (2,1.5);

    \begin{scope}[xshift=2cm]
		\draw[densely dotted] (2,1.5) -- (2,0) -- (0,0);
		\huequito[x=1,y=.8]
		\draw[densely dotted] (2,1.5) -- (1.8,1.5) 
		 .. controls (1.3,.5) and (.8,2.5) .. 
		 (.2,1.5) -- (0,1.5);
    \end{scope}
    \end{tikzpicture}
    \end{center}
    \caption{Manifold $M$, where the dotted part is $\mathbb{D} \subset \partial M$ and the dashed line is where the boundary connected sum is performed}
    \label{fig:my_label}
\end{figure}
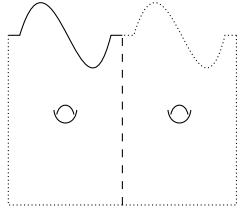

Let $\partial \Phi:= \Phi|_{\partial M} \in \Diff(\partial M, \mathbb{D})$, then it suffices to prove that $\partial \Phi$ is isotopic to $\id_{\partial M}$ relative to $\mathbb{D}$. 
Since $M \in \Mhalf$ has $g(M)=2g$, where $g=g(W)$, and $\Arf(M)=2\Arf(W)=0$, then by Proposition \ref{prop path components} and Remark \ref{rem representatives of path components} there is a diffeomorphism $f: M \xrightarrow{\cong} W_{2g,1}$. 
(Moreover, $f$ can be taken to be orientation-preserving by Proposition \ref{prop classifying space}.)
Thus, $\Phi$ induces a diffeomorphism 
\begin{equation*}
    \resizebox{\displaywidth}{!}{$W_{2g}:= W_{2g,1} \cup_{\id_{\partial D^{2n}}} D^{2n} \xrightarrow{f \circ \Phi \circ f^{-1} \cup \id_{D^{2n}}} W_{2g,1} \cup_{\partial f \circ \partial \Phi \circ \partial f^{-1}} D^{2n}=W_{2g} \# \Sigma_{[\partial f \circ \partial \Phi  \circ \partial f^{-1}]}.$}
\end{equation*}
By \cite[Theorem 3.1]{kosinski} the inertia group of $W_{2g}$ is trivial, and hence $\partial f \circ \partial \Phi \circ \partial f^{-1}$ is isotopic to $\id_{S^{2n-1}}$, so $\partial \Phi$ is isotopic to $\id_{\partial M}$. 
It suffices to show that there is an isotopy from $\partial \Phi$ to $\id_{\partial M}$ relative $\mathbb{D}$: 
fix an isotopy $H: I \rightarrow \Diff^+(\partial M)$ such that $H(0)=\id_{\partial M}$ and $H(1)= \partial \Phi$; then we get a commutative square

\centerline{\xymatrix{\partial I \ar[r]^-{\partial H} \ar[d] & \Diff(\partial M,\mathbb{D}) \ar@{^{(}->}[d] \\
I \ar[r]^-{H} \ar@{-->}[ru] & \Diff^+(\partial M)
}}

The obstruction to compress this map is
$$[(H,\partial H)] \in \pi_1(\Diff^+(\partial M),\Diff(\partial M, \mathbb{D}),\id).$$

\begin{claim}
 $\pi_1(\Diff^+(\partial M), \id) \rightarrow \pi_1(\Diff^+(\partial M),\Diff(\partial M, \mathbb{D}),\id)$ is surjective. 
\end{claim}

Given the claim we can finish the proof:
 pick a loop $\gamma: I \rightarrow \Diff^+(\partial M)$ based at $\id$ and mapping to 
 $-[(H,\partial H)] \in \pi_1(\Diff^+(\partial M),\Diff(\partial M, \mathbb{D}),\id)$, and then the new isotopy $H':= \gamma * H: I \rightarrow \Diff^+(\partial M)$ has a compression, and this will be the required isotopy. 

\begin{proof}[Proof of claim.]
 By isotopy extension $\Diff^+(\partial M) \xrightarrow{\res} \Emb^+(\mathbb{D},\partial M)$ is a fibration whose fibre over the inclusion $\mathbb{D} \subset \partial M$ is $\Diff(\partial M, \mathbb{D})$. 
Thus we can identify $\pi_1(\Diff^+(\partial M),\Diff(\partial M, \mathbb{D}),\id) \cong \pi_1(\Emb^+(\mathbb{D},\partial M), \incl)$.
Since $\mathbb{D}$ is diffeomorphic to a disc and  $\partial M \cong S^{2n-1}$ via $\partial f$ then $\Emb^+( \mathbb{D},\partial M)=\Emb^+(D^{2n-1},S^{2n-1}) \simeq \Fr^+(TS^{2n-1})= SO(2n)$.

The composition
\begin{equation*}
    \resizebox{\displaywidth}{!}{$SO(2n) \hookrightarrow \Diff^+(S^{2n-1}) \cong \Diff^+(\partial M) \rightarrow \Emb^+(\mathbb{D},\partial M) \simeq \Fr^+(TS^{2n-1})= SO(2n)$}
\end{equation*}
is a homotopy equivalence, so it induces an isomorphism on $\pi_1(-)$, giving the required surjectivity. 
\end{proof}
\end{proof}

\subsection{Comparing homological stability of \textbf{R} to diffeomorphism groups} \label{section 4.2}

Our ultimate goal is to show homological stability results about $B \Diff_{\partial}(W_{g,1})$ but as we saw in Lemma \ref{lem classifying space} the path-components of $\textbf{R}$ are models of $B \Diff_{\frac{1}{2}\partial}(W)$. 
The main result of this subsection says that homological stability of diffeomorphism groups relative the full boundary is equivalent to homological stability of $\textbf{R}$ itself. 
Before stating the precise statement let us introduce some notation: a map is called \textit{homologically K-connected} if its relative homology groups vanish in degrees $\leq K$, i.e. if it is a homology isomorphism in degrees smaller than $K$ and a homology surjection in degree $K$. 

\begin{note}
Our definition of homologically $K$-connected is called homologically $(K+1)$-connective in \cite{Ek}.
\end{note}

\begin{theorem} \label{theorem stab comparison}
Let $W_0 \in \textbf{R}$ and $W \in \Mhalf$, then
\begin{enumerate}[(i)]
    \item The stabilization map $B(\id_{W_0} \natural -): B \Diffhalf(W_1) \rightarrow B \Diffhalf(W_0 \natural W_1)$ is homologically $K$-connected (with some coefficients $\mathds{k}$) for some $K$ if and only if $B(\id_{W_0} \natural -): B \Diff_{\partial}(W_1) \rightarrow B \Diff_{\partial}(W_0 \natural W_1)$ is homologically $K$-connected (with $\mathds{k}$ coefficients).
    \item The map
    \begin{equation*}
    \begin{aligned}
        \coker\big(H_1(B\Diff_{\partial}(W)) \rightarrow
H_1(B\Diff_{\partial}(W_0 \natural W))\big) \xrightarrow{\cong} \\ \coker\big(H_1(B\Diffhalf(W)) \rightarrow H_1(B\Diffhalf(W_0 \natural W))\big)
\end{aligned}
    \end{equation*}
    is an isomorphism. 
\end{enumerate}
\end{theorem}

In order to prove the above result we need the following result on fibrations
\begin{lemma} \label{lem 2}
Suppose we have a homotopy-commutative diagram of (homotopy) fibrations

\centerline{\xymatrix{F \ar[r] \ar[d]^-{f} & E \ar[r] \ar[d]^-{g} & B \ar[d]^-{h} \\
F' \ar[r] & E' \ar[r] & B'
}}

where $B$, $B'$ are 1-connected and $h$ is a homotopy equivalence. \\
If $H_*(F',F;\mathds{k})=0$ for $* \leq K$ then $H_*(E',E;\mathds{k})=0$ for $* \leq K$ and furthermore $H_{K+1}(F',F;\mathds{k}) \xrightarrow{\cong} H_{K+1}(E',E;\mathds{k})$.
\end{lemma}

\begin{proof}

By pulling back the bottom fibration along the map $h$ and using the naturality of the Serre spectral sequence to compare this new one to the original bottom one, we reduce to the special case that $B=B'$ and $h=\id_B$.
Since $B$ is simply-connected, there is a relative Serre spectral sequence
$$E^2_{p,q}=H_p(B,H_q(F',F)) \Rightarrow H_{p+q}(E',E)$$
where we have removed the coefficients $\mathds{k}$ in the notation as they play no role in this proof. 

By assumption, $H_q(F',F)=0$ for $q \leq K$, and so $E^2_{p,q}=0$ for $p+q \leq K$, hence $H_d(E',E)=0$ for $d \leq K$. 
Moreover, on the line $p+q=K+1$ the only non-vanishing entry on the $E^2$-page of the spectral sequence is $E^2_{0,K+1}=H_{K+1}(F',F)$, so it suffices to show that this group survives to the $E^{\infty}$-page: 
it is immediate that all differentials vanish on this group, and any differential targeting position $(0,K+1)$ must come from a position of the form $(p,K+2-p)$ with $p \geq 2$, but all these entries already vanish in the $E^2$-page. 
\end{proof}

\begin{proof}[Proof of Theorem \ref{theorem stab comparison}]
By Theorem \ref{theorem fibration} there is a commutative diagram of (homotopy) fibrations with simply-connected base spaces

\label{diag 1}
\centerline{\xymatrix{ B\Diff_{\partial}(W_1) \ar[r]^-{B\incl} \ar[d]^-{B(\id_{W_0} \natural -)} & B\Diffhalf(W_1) \ar[r]^-{B\res} \ar[d]^-{B(\id_{W_0}\natural -)} & B\Diff_{\partial,0}(\mathcal{D}(W_1)) \ar[d]^-{B( \id_{\mathcal{D}(W_0)} \natural -)} \\
B\Diff_{\partial}(W_0 \natural W_1) \ar[r]^-{B\incl} & B\Diffhalf(W_0 \natural W_1) \ar[r]^-{B\res} & B\Diff_{\partial,0}(\mathcal{D}(W_0 \natural W_1))
}}

By Lemma \ref{lem 2} it suffices to show that the rightmost vertical map in the diagram is a homotopy equivalence. 
We will show that in fact
$$B(\id_{\mathcal{D}(W_0)} \natural -)): B \Diff_{\partial}(\mathcal{D}(W_1)) \rightarrow B \Diff_{\partial}(\mathcal{D}(W_0 \natural W_1))$$
is a homotopy equivalence:
there is a canonical identification 
$\mathcal{D}(W_0 \natural W_1)=\mathcal{D}(W_0) \natural \mathcal{D}(W_1)$, and both $\mathcal{D}(W_i)$ are abstractly diffeomorphic to $D^{2n-1}$, say via diffeomorphisms $\phi_i: \mathcal{D}(W_i) \xrightarrow{\cong} D^{2n-1}$.
Think of $D^{2n-1}$ as $I^{2n-1}$ and without loss of generality pick $\phi_0$ to be the identity in a neighbourhood of the rightmost face $\{1\} \times I^{2n-2}$ and $\phi_1$ to be the identity in a neighbourhood of the leftmost face $\{0\} \times I^{2n-2}$. 
Thus we get a commutative diagram in which the vertical maps are homeomorphisms

$\centerline{\xymatrix{
\Diff_{\partial}(\mathcal{D}(W_1)) \ar[r]^-{\id_{\mathcal{D}(W_0)} \natural -} \ar[d]^-{\phi_1 \circ - \circ \phi_1^{-1}} & \Diff_{\partial}(\mathcal{D}(W_0 \natural W_1)) \ar[d]^-{(\phi_0 \natural \phi_1) \circ - \circ (\phi_0 \natural \phi_1)^{-1}} \\
\Diff_{\partial}(D^{2n-1}) \ar[r]^-{\id \natural -} & \Diff_{\partial}(D^{2n-1} \natural D^{2n-1}) \ar[r]^-{=} & \Diff_{\partial}(D^{2n-1})
}}$

Thus, it suffices to show that $B(\id_{D^{2n-1}} \natural -)$ is a homotopy equivalence as a self map on $B\Diff_{\partial}(D^{2n-1})$, which is inmediate. 
\end{proof}

\subsection{Mapping class groups and arithmetic groups} \label{section 4.3}

In this subsection we study the first homology groups of $\textbf{R}$.
To do so, recall that by Lemma \ref{lem classifying space} each path-component of $\textbf{R}$ is the classifying space of a certain group of diffeomorphisms; and that for any topological group $G$, $H_1(BG;\mathbb{Z})$ is just the abelianization of $\pi_0(G)$. 
Thus, to understand the first homology of $\textbf{R}$ we need to understand some mapping class groups and their abelianizations. 
We will follow results of \cite{kreck} to achieve this. 

Given $W \in \Mhalf$, its \textit{mapping class group relative to half of the boundary} is $\Gamma_{\frac{1}{2}\partial}(W):= \pi_0(\Diffhalf(W))$. 
Let $\gamma(W):=\Aut(H_n(W),\lambda_W,q_W)$ and $g=g(W)$ be the genus of $W$. 
The arithmetic groups $\gamma(W)$ are well-known groups: 
in dimensions $n=3,7$, $\gamma(W)=\Sp_{2g}(\mathbb{Z})$ is a symplectic group, and for $n \neq 3,7$, $\gamma(W)$ is a quadratic symplectic group, in the sense of \cite[Section 1.1]{Sierra2022-st}: either $\Sp_{2g}^q(\mathbb{Z})$ if $\Arf(W)=0$ or $\Sp_{2g}^a(\mathbb{Z})$ if $\Arf(W)=1$. 

\begin{theorem}[Kreck, Krannich] \label{theorem kreck}
For $W \in \Mhalf$ there is a short exact sequence
$$1 \rightarrow \Hom_{\mathbb{Z}}(H_n(W),S\pi_n(SO(n))) \rightarrow \Gamma_{\frac{1}{2}\partial}(W) \rightarrow \gamma(W) \rightarrow 1$$
where $S\pi_n(SO(n)):=\im(\pi_{n}(SO(n)) \rightarrow \pi_{n}(SO(n+1)))$.
\end{theorem}

\begin{proof}
This proof is just a small generalization of the one given in \cite{krannichmcg} and \cite{kreck} to allow manifolds $W \in \mathbf{R}$ which are not $W_{g,1}$, so we will focus on explaining the differences and refer to the original papers for the details of the argument. 
By \cite[Proposition 3]{kreck} for $W \in \Mhalf$ there are short exact sequences
\begin{enumerate} [(i)]
    \item $1 \rightarrow \mathcal{I}_{\partial}(W) \rightarrow \Gamma_{\partial}(W) \rightarrow \gamma(W) \rightarrow 1$
    \item $1 \rightarrow \Theta_{2n+1} \rightarrow \mathcal{I}_{\partial}(W) \rightarrow \Hom_{\mathbb{Z}}(H_n(W),S\pi_n(SO(n))) \rightarrow 1$
\end{enumerate}
where $\Gamma_{\partial}(W):=\pi_0(\Diff_{\partial}(W))$ and $\mathcal{I}_{\partial}(W)$ is the kernel of the obvious map $\Gamma_{\partial}(W) \rightarrow \gamma(W)$.
Let us remark that \cite[Proposition 3]{kreck} assumes that $\partial W$ is the standard sphere, but the same proof works even when $\partial W$ is a homotopy sphere: 
all the ingredients that the proof uses are the handle structure of $W$ relative to a disc in its boundary, that $W$ is s-parallelizable, and that the inertia group of $W$ vanishes. 
The vanishing of the inertia group of $W$ is equivalent to the main step in the proof of Theorem \ref{theorem fibration}. 

Thus we get a short exact sequence
$$ 1 \rightarrow \Hom_{\mathbb{Z}}(H_n(W),S\pi_n(SO(n))) \rightarrow \frac{\Gamma_{\partial}(W)}{\Theta_{2n+1}} \rightarrow \gamma(W) \rightarrow 1$$

By \cite[Lemma 1.2]{krannichmcg} the image of $\Theta_{2n+1} \rightarrow \Gamma_{\partial}(W)$ is central and becomes trivial in $\Gamma_{\frac{1}{2}\partial}(W)$. 
Moreover, the induced map $\frac{\Gamma_{\partial}(W)}{\Theta_{2n+1}} \rightarrow \Gamma_{\frac{1}{2}\partial}(W)$ is an isomorphism. 
Only the case $W=W_{g,1}$ is treated in \cite{krannichmcg}, but the proof works in our context too by following the same steps and replacing \cite[equation (1.3)]{krannichmcg} by  Theorem \ref{theorem fibration}.
The result then follows. 
\end{proof}

The argument of \cite[Lemma 1.3]{krannichmcg} applies in our situation too, giving that the action of $\gamma(W)$ on $\Hom_{\mathbb{Z}}(H_n(W),S\pi_n(SO(n)))$ agrees with the standard action of $\gamma(W)$ on $H_n(W)$ and the trivial one on $S\pi_n(SO(n))$.  
Moreover, $\Hom_{\mathbb{Z}}(H_n(W),S\pi_n(SO(n))) \cong H_n(W) \otimes S\pi_n(SO(n))$ as $\gamma(W)$-modules by Poincaré duality. 
Thus Theorem \ref{theorem kreck} can be re-written in the following way. 

\begin{corollary} \label{cor kreck}
There is a short exact sequence
$$1 \rightarrow H_n(W) \otimes S\pi_n(SO(n)) \rightarrow \Gamma_{\frac{1}{2}\partial}(W) \rightarrow \gamma(W) \rightarrow 1$$
which is compatible with the $\gamma(W)$-action on the first group. 
\end{corollary}

The main result of this subsection is
\begin{theorem} \label{theorem stab reduction arithmetic}
Let $W_0, W \in \textbf{R}$, then the stabilization map by $W_0 \natural-$ induces an isomorphism
\begin{equation*}
    \resizebox{\displaywidth}{!}{$\coker(H_1(\Gamma_{\frac{1}{2}\partial}(W)) \rightarrow H_1(\Gamma_{\frac{1}{2}\partial}(W_0 \natural W))) \xrightarrow{\cong} \coker(H_1(\gamma(W)) \rightarrow H_1(\gamma(W_0 \natural W))).$}
\end{equation*}
\end{theorem}

\begin{proof}
By Corollary \ref{cor kreck} we get a commutative diagram of exact sequences
\begin{equation*}
\resizebox{\displaywidth}{!}{
\xymatrix{ (H_n(W) \otimes S\pi_n(SO(n)))_{\gamma(W)} \ar[d]^-{s} \ar[r] & H_1(\Gamma_{\frac{1}{2}\partial}(W)) \ar[d]^-{s} \ar[r] & H_1(\gamma(W)) \ar[r] \ar[d]^-{s} & 0 \\
(H_n(W_0 \natural W) \otimes S\pi_n(SO(n)))_{\gamma(W_0 \natural W)} \ar[r] & H_1(\Gamma_{\frac{1}{2}\partial}(W_0 \natural W)) \ar[r] & H_1(\gamma(W_0 \natural W)) \ar[r] & 0.
}}
\end{equation*}
By the snake lemma, it suffices to show that the leftmost vertical map is surjective. 
Since $W,W_0 \in \textbf{R}$ then $g(W_0 \natural W) \geq 2$, so it suffices to show that for any $V \in \mathbf{R}$ with $g(V) \geq 2$, $(H_n(V) \otimes S\pi_n(SO(n)))_{\gamma(V)}=0$. 
We will show that for any abelian group $A$ and $g \geq 2$ we have $(\mathbb{Z}^{2g} \otimes A)_{\gamma_g}=0$ for the following three cases
\begin{enumerate}[(i)]
    \item $\gamma_g=\Sp_{2g}(\mathbb{Z})$
    \item $\gamma_g=\Sp_{2g}^q(\mathbb{Z})$
    \item $\gamma_g=\Sp_{2g}^a(\mathbb{Z})$
\end{enumerate}
where in all the cases $\gamma_g$ has the standard action on $\mathbb{Z}^{2g}$ and acts trivially on $A$. 

Cases (i) and (ii) follow from \cite[Lemma A2]{krannichmcg}. 
For case (iii) fix a hyperbolic basis $e_0, f_0, \cdots e_{g-1}, f_{g-1}$ of $\mathbb{Z}^{2g}$ with $\lambda(e_i,e_j)=\lambda(f_i,f_j)=0$, $\lambda(e_i,f_j)= \delta_{i,j}$, and
$q(e_i)=q(f_i)= \left\{ \begin{array}{lcc}
             1  &  if \; i=0
             \\ 0 &  \text{otherwise}
             \end{array}
   \right.$ $\forall i,j$,
and write $[-]$ for the residue class of a given element of $\mathbb{Z}^{2g} \otimes A$ in the coinvariants.    
The permutations of the $g-1$ hyperbolic summands generated by each pair $e_i,f_i$ with $1 \leq i \leq g-1$ show that $[e_i \otimes a]= [e_j \otimes a]$ and $[f_i \otimes a]= [f_j \otimes a]$ for any $1 \leq i,j \leq g-1$ and any $a \in A$. 
Also, for each fixed $0 \leq i \leq g-1$ the transformation $e_i \mapsto f_i, \; f_i \mapsto -e_i$ lies in $\Sp^a_{2g}(\mathbb{Z})$, and so $[e_i \otimes a]= [f_i \otimes a]$ for any $0 \leq i \leq g-1$ and any $a \in A$. 
Since $g \geq 2$, the transformation given by $e_0 \mapsto e_1-f_1, \; f_0 \mapsto f_1+ e_0+ f_0, \; e_1 \mapsto e_0-e_1+f_1, \; f_1 \mapsto f_0+e_1-f_1$ exists and lies in $\Sp^a_{2g}(\mathbb{Z})$, and hence it implies that $[e_0 \otimes a]= [e_1 \otimes a]-[f_1 \otimes a]=0$ and $[e_1 \otimes a]= [e_0 \otimes a]-[e_1 \otimes a]+[f_1 \otimes a]= [e_0 \otimes a]$ for any $a \in A$. 
\end{proof}

Now we will need the following two inputs about the homology of arithmetic groups. 
The proof of the first result can be found in \cite[Lemma A.1,(i),(ii)]{krannichmcg}, and the proof of the second in \cite[Theorems 7.6, 7.7 and 7.8]{Sierra2022-st}. 
\begin{theorem}[Krannich] \label{theorem aritheoremetic 3,7}
The stabilization map
$$H_1(B\Sp_2(\mathbb{Z});\mathbb{Z}) \xrightarrow{\sigma^{g-1} \cdot -} H_1(B\Sp_{2g}(\mathbb{Z});\mathbb{Z})$$
is always surjective. 
\end{theorem}

Let us denote by $\Sp_{2g}^{\epsilon}(\mathbb{Z})$ the quadratic symplectic group of genus $g$ and Arf invariant $\epsilon \in \{0,1\}$ so that $\Sp_{2g}^{0}(\mathbb{Z})=\Sp_{2g}^{q}(\mathbb{Z})$ and $\Sp_{2g}^{1}(\mathbb{Z})=\Sp_{2g}^{a}(\mathbb{Z})$. 

\begin{theorem} \label{theorem aritheoremetic odd}
For any $\epsilon \in \{0,1\}$ the stabilisation map 
$$H_1(B\Sp_2^{\delta-\epsilon}(\mathbb{Z});\mathbb{Z}) \xrightarrow{\sigma_{\epsilon} \cdot \sigma_0^{g-2} \cdot -} H_1(B\Sp_{2g}^{\delta}(\mathbb{Z});\mathbb{Z})$$
\begin{enumerate}[(i)]
    \item Has cokernel isomorphic to $\frac{\mathbb{Z}}{2\mathbb{Z}}$ generated by $Q_{\mathbb{Z}}^1(\sigma_0)=Q_{\mathbb{Z}}^1(\sigma_1)$ if $\delta=0$ and $g=2$. 
    \item Is surjective if $\delta=1$ and $g=2$. 
    \item Is surjective if $g \geq 3$ for any $\delta$. 
\end{enumerate}

\end{theorem}
The above two results and Theorem \ref{theorem stab reduction arithmetic} imply
\begin{corollary} \label{cor reformulation stab}
\begin{enumerate}[(i)]
    \item If $n=3,7$ then the stabilization map
    $H_{1,1}(\textbf{R}) \rightarrow H_{g,1}(\textbf{R})$
    is surjective for any $g \geq 2$.
    \item If $n$ is odd, $n \neq 3,7$ and $\epsilon \in \{0,1\}$ then the stabilization map
    $$H_{(1,\delta-\epsilon),1}(\textbf{R}) \rightarrow H_{(g,\delta),1}(\textbf{R})$$
    \begin{enumerate}[(i)]
        \item Has cokernel isomorphic to $\frac{\mathbb{Z}}{2\mathbb{Z}}$ generated by $Q_{\mathbb{Z}}^1(\sigma_0)$ if $\delta=0$ and $g=2$. 
    \item Is surjective if $\delta=1$ and $g=2$. 
    \item Is surjective if $g \geq 3$ for any $\delta$. 
    In particular, $\sigma_{1-\epsilon} \cdot Q_{\mathbb{Z}}^1(\sigma_0)$ lies in the image of $\sigma_{\epsilon}^2 \cdot -:H_{(1,1-\epsilon),1}(\textbf{R}) \rightarrow H_{(3,1-\epsilon),1}(\textbf{R})$.
    \end{enumerate}
\end{enumerate}
\end{corollary}

\subsection{Results on rational homology of diffeomorphism groups} \label{section 4.4}

The aim of this subsection is to understand the rational homology of the diffeomorphism groups we are interested in.  

\begin{theorem} [Berglund-Madsen,Krannich] \label{berglundmadsen}
For $W \in \textbf{R}$ the stabilization map  
$$H_d(B \Diff_{\partial}(W);\mathbb{Q}) \rightarrow H_d(B \Diff_{\partial}(W_{1,1} \natural W);\mathbb{Q})$$
is surjective for $d \leq \min\{g(W),3n-7\}$ and an isomorphism for $d \leq \min \{g(W)-1,3n-7\}$.
\end{theorem}

Let us remark that the proof is just a small generalization of the one given in \cite[Theorem A]{krannich} to allow manifolds $W \in \mathbf{R}$ which are not $W_{g,1}$, so we will focus on explaining the differences and refer to the original paper for the details of the argument. 

\begin{proof}

By combining  \cite[Proposition 4.3, Theorem 4.1]{secondderivative} it follows that for both $V=W$ and $V=W_{1,1} \natural W$, $B \Diff_{\partial}(D^{2n}) \rightarrow B \Diff_{\partial}(V) \rightarrow B \widetilde{\Diff}_{\partial}(V)$ is a rational fibre sequence in degrees $\leq 3n-7$. 
Note that we need to generalize the above citation slightly to take into account the possibility that $V \in \textbf{R}$ is not a $W_{g,1}$, but the identical method of proof works in this case too. 
Moreover, the action of $\pi_{1}(B \widetilde{\Diff}_{\partial}(V))$ on $B \Diff_{\partial}(D^{2n})$ is trivial since any diffeomorphism can be isotoped to fix a disc near the boundary. 

Now consider the corresponding diagram of rational fibrations for both $V=W$ and $V=W_{1,1} \natural W$ with stabilization maps between the rows. 
The map on $B \Diff_{\partial}(D^{2n})$ is a self-homotopy equivalence. 
Thus, by the relative Serre spectral sequence with rational coefficients, similar to the one of Lemma \ref{lem 2}, it suffices to show that the stabilisation map 
$$H_d(B \widetilde{\Diff}_{\partial}(W);\mathbb{Q}) \rightarrow H_d(B \widetilde{\Diff}_{\partial}(W_{1,1} \natural W);\mathbb{Q})$$
is surjective for $d \leq g(W)$ and an isomorphism for $d \leq g(W)-1$.

In our context \cite[Theorem 1.1]{krannich} also applies, the only difference being that when $n \neq 3,7$ the group called $\mathbf{G_g}$ can be either $\Sp_{2g}^q$ or $\Sp_{2g}^a$, depending on the Arf invariant of the manifold. 
In order to get the surjectivity range in one degree higher than the isomorphism range we need the surjectivity range of \cite[Theorem 1.1]{krannich} to be one degree higher than the isomorphism range too.
We can improve the surjectivity range as follows: 
by \cite[Theorem 2]{tshishiku}, it suffices to show that $H_{g}(\Sp_{2g}^{\epsilon}(\mathbb{Z});\mathbb{Q}) \rightarrow H_{g}(\Sp_{2(g+1)}^{\epsilon}(\mathbb{Z});\mathbb{Q})$ is surjective when $\epsilon$ is either $q,a$ or nothing.
By transfer $H_{g}(\Sp_{2g}^{\epsilon}(\mathbb{Z});\mathbb{Q}) \rightarrow H_{g}(\Sp_{2g}(\mathbb{Z});\mathbb{Q})$ is surjective for $\epsilon$ as above, 
the stable rational homology of the quadratic symplectic groups is isomorphic to the stable rational homology of the usual symplectic groups by the work of Borel, and by \cite[Theorem 1.1]{krannich} the groups $H_{g}(\Sp_{2(g+1)}^{\epsilon}(\mathbb{Z});\mathbb{Q})$ are already stable. 
Thus, it suffices to show that $H_{g}(\Sp_{2g}(\mathbb{Z});\mathbb{Q}) \rightarrow H_{g}(\Sp_{2(g+1)}(\mathbb{Z});\mathbb{Q})$ is surjective, which is precisely \cite[Proposition 12]{improvingrationalrange}. 

The whole of \cite[Section 2]{krannich} also applies in our case: the methods of \cite{berglundmadsen} generalize to any $W \in \mathbf{R}$ and give the same expression for the homology of block diffeomorphisms homotopic to the identity as the one for $W_{g(W),1}$: 
this is because rationally both $W$ and $W_{g(W),1}$ have the same homology, intersection product and boundary map. 

Thus, the spectral sequence argument presented in \cite{krannich} also applies to our case, and the surjectivity result on arithmetic groups in one degree higher gives the required surjectivity range for Block diffeomorphisms too. 
\end{proof}

\subsection{Proof of Theorem \hyperref[theorem A]{A}} \label{section 4.5}

Now we will finally prove Theorem \hyperref[theorem A]{A} in two parts: one for $\mathbb{Z}$ and $\mathbb{Z}[1/2]$-coefficients, and another one for $\mathbb{Q}$-coefficients. 

\begin{theorem} \label{theorem integral stab}
For $n \geq 3$ odd, consider the stabilization map
$$H_d(B \Diff_{\partial}(W_{g-1,1});\mathds{k}) \rightarrow H_d(B \Diff_{\partial}(W_{g,1});\mathds{k}).$$
Then
\begin{enumerate}[(i)]
    \item If $n=3,7$ and $\mathds{k}=\mathbb{Z}$ it is surjective for $3d \leq 2g-1$ and an isomorphism for $3d \leq 2g-4$. 
    \item If $n \neq 3,7$ and $\mathds{k}=\mathbb{Z}$ it is surjective for $2d \leq g-2$ and an isomorphism for $2d \leq g-4$. 
    \item If $n \neq 3,7$ and $\mathds{k}=\mathbb{Z}[\frac{1}{2}]$ it is surjective for $3d \leq 2g-4$ and an isomorphism for $3d \leq 2g-7$. 
\end{enumerate}
\end{theorem}

\begin{proof}[Proof of Theorem \ref{theorem integral stab}]
By Remarks \ref{comparison stab R half bry}, \ref{rem representatives of path components} and Theorem \ref{theorem stab comparison} it suffices to prove the corresponding stability results for $\textbf{R}$. As we explained in Section \ref{section 1.2}, we can show them for $\textbf{R}_{\mathds{k}}$ instead. 

When $n=3,7$ apply \cite[Theorem 18.1]{Ek}: the assumptions to verify are the vanishing line in $E_2$-homology, which holds by Theorem \hyperref[theorem C]{C} and Corollary \ref{cor theorem key}, the surjectivity of the stabilisation map on first homology and grading 2, which holds by Corollary \ref{cor reformulation stab},(i), and the computation of zero-th homology, which holds by Corollary \ref{cor 0 connected}.

When $n \neq 3,7$ we take $\mathds{k}=\mathbb{Z}$ for part (ii) or $\mathds{k}=\mathbb{Z}[1/2]$ for part (iii). 
Then apply Lemmas \ref{lem stab 1}, \ref{lem stab 2}: the vanishing line in $E_2$-homology holds by Theorem \hyperref[theorem C]{C} and Corollary \ref{cor theorem key}, and the remaining assumptions hold by Corollary \ref{cor 0 connected} and by the universal coefficients theorem in homology and Corollary \ref{cor reformulation stab}. 
\end{proof}

\begin{theorem} \label{theorem rational stability}
For $n \geq 3$ odd, the stabilization maps 
$$H_d(B \Diff_{\partial}(W_{g-1,1});\mathbb{Q}) \rightarrow H_d(B \Diff_{\partial}(W_{g,1});\mathbb{Q})$$
are surjective for $d< \frac{3n-6}{3n-5}(g-c_n)$ and isomorphisms for $d< \frac{3n-6}{3n-5}(g-c_n)-1$, where $c_n=0$ for $n=3,7$ and $c_n=1$ otherwise. 
\end{theorem}

\begin{proof}
We proceed as in the above proof to reduce it to verifying the assumptions of Theorem \ref{theorem rational stability general} for $\textbf{X}=\mathbf{R}_{\mathbb{Q}}$ and $D=3n-6$. 
The vanishing line in $E_{2n-1}$-homology holds by Theorem \hyperref[theorem C]{C} and Corollary \ref{cor theorem key}. 
The existence of appropriate maps $\textbf{A} \rightarrow \textbf{R}_{\mathbb{Q}}$ follows from Remark \ref{rem application rational}. 
Now we need to verify the remaining part of assumption (i) if $n=3,7$ or assumption (ii) otherwise. 

When $g=1$ only the case $d=0$ needs to be considered, but this case is fine by Remark \ref{rem application rational}. 
When $g \geq 2$, Remark \ref{rem representatives of path components} says that the class $\sigma_0$ is generated by a model of $W_{1,1}$ so by Theorem \ref{theorem stab comparison} it suffices to check that for any $W \in \textbf{R}$ the stabilization map 
$$ H_d(B \Diff_{\partial}(W);\mathbb{Q}) \rightarrow H_d(B \Diff_{\partial}(W_{1,1} \natural W);\mathbb{Q})$$
is surjective for $d \leq \min\{g(W),3n-7\}$ and an isomorphism for $d \leq \min \{g(W)-1,3n-7\}$, which follows from Theorem \ref{berglundmadsen}.
\end{proof}

\subsection{Proof of Theorem \hyperref[theorem B]{B}} \label{section 4.6}

The proof of Theorem \hyperref[theorem B]{B} is based on \cite[Corollary 3.5]{Sierra2022-st} which says the following: 
let $\mathbf{X} \in \Alg_{E_2}(\mathsf{sMod}_{\FF}^{\mathsf{H}})$ satisfy the assumptions of Lemma \ref{lem stab 2} (but with $\mathds{k}=\FF$ instead), and that $\sigma_0 \cdot Q_2^1(\sigma_0)$ destabilizes twice by $\sigma_0$, then one of the following two options holds
\begin{enumerate}[(i)]
     \item $H_{(4k,0),2k}(\mathbf{\overline{X}}/\sigma_{\epsilon}) \neq 0$ for all $k \geq 1$, and in particular the optimal slope for the stability is $1/2$. 
    \item $H_{x,d}(\mathbf{\overline{X}}/\sigma_{\epsilon})=0$ for $3d \leq 2\rk(x)-6$, so $\mathbf{X}$ satisfies homological stability of slope at least $2/3$ with respect to $\sigma_{\epsilon}$. 
\end{enumerate}

The proof of this result is based on a secondary stability result given in \cite[Theorem 2.3]{Sierra2022-st}.

Now we can finally prove Theorem \hyperref[theorem B]{B}.

\begin{proof}
We let $\mathbf{X}=\mathbf{R}_{\FF}$ and $\epsilon=0$, then it satisfies all the assumptions needed to apply \cite[Corollary 3.5]{Sierra2022-st} by Theorem \hyperref[theorem C]{C}, Corollaries \ref{cor theorem key}, \ref{cor 0 connected} and by the universal coefficients theorem and Corollary \ref{cor reformulation stab}.
Now we have two cases to consider. 

\textbf{Case (i)} Suppose $H_{(4k,0),2k}(\overline{\mathbf{R}_{\FF}}/\sigma_0) \neq 0$ for all $k \geq 1$. 
By Remark \ref{rem representatives of path components} and Theorem \ref{theorem stab comparison} this implies that for each $k \geq 1$ there is some $d(k) \leq 2k$ for which $H_{d(k)}(B \Diff_{\partial}(W_{4k,1}), B \Diff_{\partial}(W_{4k-1,1});\FF) \neq 0$. 
By Theorem \hyperref[theorem A]{A}(ii) and the universal coefficients theorem we have $d(k) \geq 2k$, and thus 
$$H_{2k}(B \Diff_{\partial}(W_{4k,1}), B \Diff_{\partial}(W_{4k-1,1});\mathbb{Z}) \neq 0$$
for all $k \geq 1$ by another application of the universal coefficients theorem. 

\textbf{Case (ii)} Suppose $H_{x,d}(\overline{\mathbf{R}_{\FF}}/\sigma_0)=0$ for $3d \leq 2\rk(x)-6$.
Then by Remark \ref{rem representatives of path components} and Theorem \ref{theorem stab comparison} we have 
$$H_d(B \Diff_{\partial}(W_{g,1}),B \Diff_{\partial}(W_{g-1,1});\FF)=0$$
for $3d \leq 2g-6$. 
Moreover, by Theorem \hyperref[theorem A]{A}(iii), 
$$H_d(B \Diff_{\partial}(W_{g,1}),B \Diff_{\partial}(W_{g-1,1});\mathbb{Z}[1/2])=0$$
for $3d \leq 2g-4$.
Finally, the required integral vanishing follows from the universal coefficients theorem and the finite generation of $H_d(B \Diff_{\partial}(W_{g,1});\mathbb{Z})$ for any $g,d$, which follows from \cite[Theorem 6.1, Remark 6.2]{finitenessproperties}. 
\end{proof}

\begin{rem} \label{remark quantised}
 In \cite{Sierra2022-st} we find ``quantisation stability results'' using \cite[Theorem A, Theorem B, Corollary 3.5]{Sierra2022-st}, but they are consequence of a secondary stability result. 
 However, our case is different: the secondary stabilisation map lives in the algebra $\mathbf{R}$ itself, but it is not defined in $\bigsqcup_{g}{B \Diff_{\partial}(W_{g,1})}$. 
 However, Theorem \ref{theorem stab comparison} allows us to ``pull-back'' the quantisation stability from $\mathbf{R}$ even if we cannot pull-back the secondary stability. 
\end{rem}

Let us also mention that quoting the finite generation of homology groups from \cite[Theorem 6.1, Remark 6.2]{finitenessproperties} is not really needed: one can use the ideas of the proofs of \cite[Theorems 2.1, 2.2, 2.3]{Sierra2022-st} to remove the hypothesis of finite generation and the expense of working with a more elaborated CW $E_2$-algebra model.  

\section{Splitting complexes and the proof of Theorem \hyperref[theorem C]{C}} \label{section 5}

In this section we will prove Theorem \hyperref[theorem C]{C}. 
The overall argument is based on \cite[Section 4]{E2} but we will need some additional steps to deal with issues associated to the non-discreteness of the diffeomorphism groups.

\subsection{The arc complex}

We will make use of a high-dimensional analogue of an ``arc complex'', inspired by \cite[Definition 4.7]{E2}. 
The intuition is to replace the choice of two points in the boundary by an isotopy class of embeddings of $S^{n-1}$, to replace arcs by embeddings of $D^n$, and to replace ``non-separating'' with the assumption that the corresponding ``cut manifold'' remains $(n-1)$-connected.   
Observe that when $n=1$ we basically recover the arc complex for surfaces, except that in the surface case arcs are only taken up to isotopy. 

\begin{definition} \label{definition valid geometric data}
A \textit{valid geometric data} is a pair $(W,\Delta)$ where $W \in \mathcal{M}[A]$, and $\Delta$ is an isotopy class of embeddings $ S^{n-1} \hookrightarrow \{1\} \times \int(A) \subset \partial W$.
\end{definition}

\begin{definition}\label{definition arc complex}
Given a valid geometric data $(W,\Delta)$, we define the \textit{arc complex} $\mathcal{A}(W,\Delta)$ to be the following simplicial complex:

\begin{enumerate}[(1)]
   \item A vertex is an embedding 
   $a: D^n \hookrightarrow W$ such that
   \begin{enumerate}[(i)]
       \item Its boundary $\partial a:=a|_{\partial D^n}$ has image contained in $\int(A) \subset W$ and lies in the isotopy class $\Delta$.
       \item $a$ intersects $\partial W$ only in $\partial a$ and the intersection is transversal. 
       \item The cut manifold $W \setminus a:= W \setminus a(D^n)$ is $(n-1)$-connected.
   \end{enumerate}
   
\item Vertices $a_0, \cdots, a_p$ span a $p$-simplex if and only if 
    \begin{enumerate}[(i)]
        \item The images of the embeddings $a_i$ are pairwise disjoint. 
        \item The (jointly) cut manifold $W \setminus \{a_0,\cdots,a_p\}:= W \setminus \bigsqcup_{i=0}^{p}{a_i(D^n)}$ is $(n-1)$-connected. 
    \end{enumerate}
\end{enumerate}
\end{definition}

The key property we need about the arc complex is the following result, which is an analogue of \cite[Theorem 4.8]{E2}. 
Its proof will be delayed to Section \ref{section 6}. 

\begin{theorem}\label{theorem arc complex connectivity}
If $(W,\Delta)$ is valid geometric data then $\mathcal{A}(W,\Delta)$ is $(g(W)-2)$-connected.
\end{theorem}

\subsection{Definition of the splitting complexes and posets}

\begin{definition}\label{definition splitting complex}
For $W \in \mathcal{M}[A]$  we define the \textit{($E_1$-)splitting complex} $S_{\bullet}^{E_1}(W)$ to be the following semisimplicial space: 
\begin{enumerate}[(1)]
    \item The space of $0$-simplices, $S_{0}^{E_1}(W)$, is given as a set by the collection of triples $(\omega,t,\epsilon)$, called ``walls'', where $\omega: [t-\epsilon,t+\epsilon] \times I^{2n-1} \hookrightarrow W$ is an embedding, $0<t-\epsilon<t<t+\epsilon<1$, such that
    \begin{enumerate}[(i)]
        \item $\omega$ agrees pointwise with the inclusion $[t-\epsilon,t+\epsilon] \times I^{2n-1} \hookrightarrow I^{2n}$ in a neighbourhood of $[t-\epsilon,t+\epsilon] \times \overline{\partial I^{2n-1} \setminus I^{2n-2} \times \{1\}}$ in $[t-\epsilon,t+\epsilon] \times I^{2n-1}$. 
        In other words, $\omega$ looks standard near its boundary except the top face. 
        Moreover, $\omega$ maps the top face $[t-\epsilon,t+\epsilon] \times I^{2n-2} \times \{1\}$ inside the interior of the top face $\mathcal{D}(W)=\partial W \cap (I^{2n-1} \times \{1\} \times \mathbb{R}^{\infty})$, and $\omega$ maps $[t-\epsilon,t+\epsilon] \times \int(I^{2n-1})$ inside the interior of $W$. 
        \item $\omega$ is compatible with the product structure near its top face in the sense of Definition \ref{definition moduli}. 
        In particular, $\omega$ is transversal to $\partial W$.  
        \item $W \setminus \im(\omega|_{\{t\} \times I^{2n-1}})$ has precisely two path components. 
        We denote by $W_{\leq \omega}$ the \textit{left region}, i.e. the closure of the path-component containing $\{0\} \times I^{2n-1}$, and by $W_{\geq \omega}$ \textit{right region}, i.e. the closure of the other path-component. 
         \item Both $H_n(W_{\leq \omega})$ and $H_n(W_{\geq \omega})$ are non-zero. 
    \end{enumerate}
    We then topologyse it as a subspace of $\Emb([-1,1] \times I^{2n-1},I^{2n} \times \mathbb{R}^{\infty}) \times (0,1)^2$ by using the parameters $t,\epsilon$ to reparametrise the embedding $\omega$.  
    \item For $p \geq 0$, $S_{p}^{E_1}(W) \subset (S_{0}^{E_1}(W))^{p+1}$ consists of the space of tuples of $0$-simplices $((\omega_0,t_0,\epsilon_0), \cdots,(\omega_p,t_p,\epsilon_p))$ such that $t_i+\epsilon_i<t_{i+1}-\epsilon_{i+1}$, $\im(\omega_i) \subset \int(W_{\leq \omega_{i+1}})$, and the region of $W$ in between these, denoted $W_{\omega_i \leq - \leq \omega_{i+1}}$, has $H_n(W_{\omega_i \leq - \leq \omega_{i+1}})$ non-zero for $0 \leq i \leq p-1$.
    \item The i-th face map is given by forgetting $(\omega_i,t_i,\epsilon_i)$. 
\end{enumerate}
We denote by $S^{E_1}(W):=||S_{\bullet}^{E_1}(W)||$ its geometric realization. 
\end{definition}

The thickness of the walls will be relevant because it gives $W_{\leq \omega}$, $W_{\geq \omega}$ and $W_{\omega \leq - \leq \omega'}$ product structures near their boundaries, but from an intuitive point of view one should think of a wall as a codimension one disc cutting the manifold into two ``nice'' pieces. 
In fact, we will often abuse notation and remove $t,\epsilon$ and say ``$\omega \in S^{E_1}(W)$ is a wall''. 

\begin{figure}[H]
   \begin{center}
     
\begin{tikzpicture}[scale=3]
	\begin{scope}[xshift=.885cm,yshift=.115cm,scale=1.09]
    \draw[rotate=45,scale=0.5] (0,1)
    -- (.1,.9) 
    to[out=0, in=180] (.5,.5)
    to[out=0, in=180] (1,0)
    to[out=0, in=180] (1.5,-.5)
    to[out=0, in=180] (2,-1)
    to[out=0, in=150] (2.5,-1.5)
    --
    (2.6,-1.6)
    ;
    \end{scope}

    \draw  (1,0) -- (0,0) -- (0.5,0.5) ;
	\churro[scale=0.25, x=0.65, y=0.25]
    
	\begin{scope}[xshift=1cm]
    \draw (1,0) -- (0,0);
	\churro[scale=0.25, x=0.65, y=0.25]
    \end{scope}
    
	\begin{scope}[xshift=1.5cm]
    \draw[dotted,thick,rotate=90,scale=0.5] (0,1)
    -- (.1,.9) 
    to[out=0, in=180] (.5,.5)
    to[out=0, in=180] (.9,.1)
    -- (.98,0.02);
    \end{scope}
    
	\begin{scope}[xshift=2.5cm]
    \draw[rotate=90,scale=0.5] (0,1)
    -- (.1,.9) 
    to[out=0, in=180] (.5,.5)
    to[out=0, in=180] (.9,.1)
    -- (1,0);
    \end{scope}

\end{tikzpicture}

   \end{center}
   \caption{A $0$-simplex represented by the dotted line.}
\end{figure}

In order to understand the above semisimplicial space it will be useful to consider its levelwise discretization too: 

\begin{definition}
For $W \in \mathcal{M}[A]$  we define the \textit{discretized ($E_1$-)splitting complex} $S_{\bullet}^{E_1,\delta}(W)$ to be the semisimplicial set obtained by taking the levelwise discretization of $S_{\bullet}^{E_1}(W)$. 
\end{definition}

We will view the above semisimplicial objects as nerves of certain posets. 
To do so, let us firstly define what we mean by topological posets and their nerves. 

\begin{definition}\label{top poset}
A \textit{(non-unital) topological poset} $(P,<)$ is a topological space $P$ with a strict partial ordering $<$ on the underlying set of $P$. 
\end{definition}

Any (non-unital) topological poset $(P,<)$ can be viewed as a non-unital topological category in the sense of \cite[Definition 3.1]{ssSpaces} by taking the space of objects to be $P_0=P$, the space of morphisms to be $P_1=\{(a,b) \in P^2: a <b\}$, and the source and target maps given by the projections onto the first and second coordinates respectively. 
The composition map is then canonically defined by the axioms of a partial ordering.
Thus, the nerve $N_{\bullet} P$ defines a semisimplicial space whose layers can be explicitly described via $P_p= N_p P:=\{x_0< x_1<\cdots<x_p: x_i \in P\}$, topologized as a subspace of $P^{p+1}$, where the i-th face map forgets $x_i$. 

\begin{exmp}
The $E_1$-splitting complex $S_{\bullet}^{E_1}(W)$ of Definition \ref{definition splitting complex} is the nerve of the non-unital topological poset $\mathcal{S}^{E_1}(W)$ whose space of elements is the space of $0$-simplices $S_{0}^{E_1}(W)$, and whose partial ordering $<$ is given by the space of $1$-simplices. 
\end{exmp}

\begin{exmp}
Given any topological poset $(P,<)$ we can construct its discretization $(P^{\delta},<)$, whose underlying set is $P$ with the discrete topology and where we use the same strict partial ordering. 
By definition $P^{\delta}_{\bullet}$ is the semisimplicial set given by levelwise discretizing $P_{\bullet}$. 
In particular, we can form the \textit{discretized ($E_1$)-splitting poset} $\mathcal{S}^{E_1,\delta}(W)$ whose nerve is the discretized splitting complex. 
\end{exmp}

\subsection{Recollection on simplicial complexes and posets}

To state some of the intermediate results to prove Theorem \hyperref[theorem C]{C} we need
some definitions and constructions about simplicial complexes and posets. All the complexes and posets appearing in this section are assumed to be discrete. 

Let us begin by recalling the relationship between posets and simplicial complexes: 
any poset gives a simplicial complex whose vertices are the elements of the poset, and where a set of $p+1$ vertices defines a $p$-simplex if and only if they are strictly ordered. 
This simplicial complex has the same geometric realization as the nerve of the poset. 
Conversely, any simplicial complex has an associated \textit{face poset} whose elements are the simplices and the partial ordering is given by inclusion. 
The associated simplicial complex to the face poset of any complex is the barycentric subdivision of the original complex, and hence homeomorphic to the original one. 

For the rest of the paper we will refer to topological properties of a poset to mean the properties of its associated simplicial complex. 
Thus, the topological properties of a simplicial complex agree with the ones of its face poset; in the rest of the paper we will not make distinctions between topological properties holding for a simplicial complex or its face poset. 

If $(P,<)$ is a poset and $x \in P$ we let 
$$\dim(x):=\sup\{n: \; \text{there is a chain} \; x_1<\cdots <x_{n} < x\},$$
and we define the \textit{dimension} of $(P,<)$ via $\dim(P):=\sup_{x \in P}\{\dim(x)\} \in \mathbb{N} \cup \{\infty\}$, which agrees with the dimension of the corresponding simplicial complex.  

\begin{definition}
For a given function of sets $f: P \rightarrow \mathbb{Z}$ we say that $(P,<)$ is \textit{$f$-weakly Cohen-Macaulay of dimension} $n$ if the following holds
\begin{enumerate}[(i)]
    \item $P$ is $(n-1)$-connected. 
    \item For each $x \in P$ the poset $P_{<x}$ is $(f(x)-2)$-connected.
    \item For each $x \in P$ the poset $P_{>x}$ is $(n-2-f(x))$-connected. 
    \item For each $x<y$ in $P$ the poset $P_{x<-<y}=P_{>x} \cap P_{<y}$ is $(f(y)-f(x)-3)$-connected.
\end{enumerate}
When $f=\dim(-)$ we say that $P$ is weakly Cohen-Macaulay of dimension $n$. 
\end{definition}

We will say that a simplicial complex is weakly Cohen-Macaulay of dimension $n$ if its associated face poset has the same property. 
In this case conditions (ii) and (iv) above are automatic as we take $f=\dim(-)$, and condition (iii) is equivalent to saying that  for a $p$-simplex $\sigma$, its link $\Lk(\sigma)$ is $(n-p-2)$-connected. 

\subsection{Some technical results}

One of the main steps in proving Theorem \hyperref[theorem C]{C} will be Corollary \ref{cor weakly cohen macaulay}, which says that for $W \in \Mhalf$ the discretized splitting poset $\mathcal{S}^{E_1,\delta}(W)$ is $f$-weakly Cohen-Macaulay of a certain dimension. 
In this section we will study the posets $\mathcal{S}^{E_1,\delta}(W)_{<\omega}$, $\mathcal{S}^{E_1,\delta}(W)_{>\omega}$ and $\mathcal{S}^{E_1,\delta}(W)_{\omega<-<\omega'}$ and show that they are isomorphic to splitting complexes of certain manifolds. 
Finally we will show a lemma about splitting complexes of cut manifolds, which will be used in the proof of Corollary \ref{cor weakly cohen macaulay}. 

\begin{lemma} \label{lemma technical}
Let $W,W' \in \mathcal{M}[A]$ lie in the same path-component.
Then  $\mathcal{S}_0^{E_1}(W)$ and $\mathcal{S}_0^{E_1}(W')$ are isomorphic topological posets. 
Thus, their discretizations are also isomorphic. 
\end{lemma}

\begin{proof}
Let $\phi: W \xrightarrow{\cong} W'$ be a diffeomorphism fixing pointwise a neighbourhood of $\partial^-W=\partial^-W'$ and preserving the product structures. 
Then we get an induced map $\phi_*: \mathcal{S}_0^{E_1}(W) \rightarrow \mathcal{S}_0^{E_1}(W')$, $(\omega,t,\epsilon) \mapsto (\phi \circ \omega,t,\epsilon)$, which is order-preserving and has an inverse given by $(\phi^{-1})_*$. 
\end{proof}

The next technical result will allow us to view the manifolds $W_{\leq \omega}$, $W_{\geq \omega}$ as elements in the moduli spaces of Definition \ref{definition moduli}. 

\begin{lemma} \label{lemma pieces splitting complex} 
Let $W \in \mathcal{M}[A]$ and $(\omega,t,\epsilon) \in S_0^{E_1}(W)$, then
\begin{enumerate}[(a)]
    \item The inclusion $W_{\leq \omega} \hookrightarrow I^{2n} \times \mathbb{R}^{\infty}$ is isotopic to an embedding $e: W_{\leq \omega}\hookrightarrow I^{2n} \times \mathbb{R}^{\infty}$ such that
    \begin{enumerate}[(i)]
        \item $\im(e) \in \Mhalf$. 
        \item On a neighbourhood of 
        $$\partial W_{\leq \omega} \cap (J_{2n-1} \times \mathbb{R}^{\infty})=J_{2n-1} \cap \{p=(x_1,\cdots,x_{2n},y) \in I^{2n}\times \mathbb{R}^{\infty}: \; 0 \leq x_1 \leq t \}$$
        it agrees with $(x_1,\cdots,x_{2n},0) \mapsto (x_1/t,x_2,\cdots,x_{2n},0)$. 
        \item $e \circ \omega|_{\{t\} \times I^{2n-1}}$ agrees with the standard inclusion $I^{2n-1} \cong \{1\} \times I^{2n-1} \times \{0\} \subset I^{2n} \times \mathbb{R}^{\infty}$.
        \item $e$ gives a diffeomorphism from $\mathcal{D}(W) \cap \partial W_{\leq \omega}$ to $\mathcal{D}(\im(e))$. 
        \item $e$ preserves the product structures (in the sense of Definition \ref{definition moduli}).
    \end{enumerate}
    Moreover, $[\im(e)] \in \pi_0(\Mhalf)$ is independent of the choice of embedding $e$ satisfying all the above conditions. 
    \item The inclusion $W_{\geq \omega} \hookrightarrow I^{2n} \times \mathbb{R}^{\infty}$ is isotopic to an embedding $e: W_{\geq \omega}\hookrightarrow I^{2n} \times \mathbb{R}^{\infty}$ such that
    \begin{enumerate}[(i)]
        \item $\im(e) \in \mathcal{M}[A]$. 
        \item On a neighbourhood of 
        $$\partial W_{\geq \omega} \cap (J_{2n-1} \times \mathbb{R}^{\infty})=J_{2n-1} \cap \{p=(x_1,\cdots,x_{2n},y) \in I^{2n}\times \mathbb{R}^{\infty}: \; t \leq x_1 \leq 1 \}$$
        it agrees with $(x_1,\cdots,x_{2n},0) \mapsto ((x_1-t)/(1-t),x_2,\cdots,x_{2n},0)$. 
        \item $e \circ \omega|_{\{t\} \times I^{2n-1}}$ agrees with the standard inclusion $I^{2n-1} \cong \{0\} \times I^{2n-1} \times \{0\} \subset I^{2n} \times \mathbb{R}^{\infty}$.
        \item $e$ gives a diffeomorphism from $\mathcal{D}(W) \cap \partial W_{\geq \omega}$ to $\mathcal{D}(\im(e))$. 
        \item $e$ preserves the product structures. 
    \end{enumerate}
    Moreover, $[\im(e)] \in \pi_0(\mathcal{M}[A])$ is independent of the choice of embedding $e$ satisfying all the above conditions. 
\end{enumerate}
\end{lemma}

\begin{proof}
The proofs of (a) and (b) are almost identical, so we will focus on part (a). 
Let us show the existence of $e$ first and then the uniqueness of the path-component. 

We can decompose $\partial W_{\leq \omega}$ as $\partial W_{\leq \omega} \cap (J_{2n-1} \times \mathbb{R}^{\infty}) \cup \im(\omega|_{\{t\} \times I^{2n-1}}) \cup \mathcal{D}(W) \cap \partial W_{\leq \omega}$, and conditions (a)(ii) and (a)(iii) tell us what $e$ should do on the first two pieces of the decomposition. 
Moreover condition (a)(v) forces the behaviour of $e$ on a closed neighbourhood of the first two pieces of the boundary, and by construction this partial embedding $e$ is isotopic to the inclusion where defined. 
Now we use isotopy extension to extend it to the remaining piece of the boundary $\mathcal{D}(W) \cap \partial W_{\leq \omega}$ in such a way that $e$ sends it inside $I^{2n-1} \times \{1\} \times \mathbb{R}^{\infty}$. 
We then use condition (a)(v) again to extend $e$ to a closed neighbourhood of the whole boundary $\partial W_{\leq \omega}$. 
Finally we use isotopy extension again to define $e$ on the whole of $W_{\leq \omega}$ such that its interior maps to $\int(I^{2n}) \times \mathbb{R}^{\infty}$. 

It remains to check that the manifold $\im(e) \subset I^{2n} \times \mathbb{R}^{\infty}$ lies in $\Mhalf$: 
the $(n-1)$-connectivity follows by Seifert-Van Kampen and Mayer-Vietoris applied to the decomposition $W= W_{\leq \omega} \cup_{I^{2n-1}} W_{\geq \omega}$, the $s$-parallelizability is immediate as $W$ is itself $s$-parallelizable. 
Conditions (ii),(iii),(iv)(a),(v) and (vi) of Definition \ref{definition moduli} hold by construction; and hence it remains to verify condition (iv)(b), i.e. that $\mathcal{D}(\im(e))$ is contractible. 
This follows by another application of Seifert-Van Kampen and Mayer-Vietoris, this time to $\mathcal{D}(W)$ decomposed as a union of two pieces along $I^{2n-2}$. 

For the uniqueness of the path-component suppose that $e'$ is another embedding, then $\partial^-(\im(e))=\partial^-(\im(e'))=J_{2n-1} \cup \{1\} \times I^{2n-1}$ and $e' \circ e^{-1} \in \Emb^p_{\partial^- \im(e)}(\im(e),I^{2n} \times \mathbb{R}^{\infty})$ by condition a(v). 
Since the embedding space is path-connected by the Whitney embedding theorem then the result follows. 
\end{proof}

As a consequence of Lemmas \ref{lemma technical} and \ref{lemma pieces splitting complex} we can meaningfully write $\mathcal{S}^{E_1}(W_{\leq \omega})$, $\mathcal{S}^{E_1}(W_{\geq \omega})$ and their discretized analogues to really mean the corresponding splitting complexes on $\im(e)$ for $e$ as in Lemma \ref{lemma pieces splitting complex}. 
By consecutively applying parts (a) and (b) Lemma \ref{lemma pieces splitting complex} one can also make sense of
$\mathcal{S}^{E_1}(W_{\omega\leq- \leq \omega'})$.

\begin{corollary} \label{cor pieces of splitting complex}
For $W \in \mathcal{M}[A]$ and $(\omega,t,\epsilon) \in S_0^{E_1}(W)$ we have isomorphism of posets
\begin{enumerate}[(i)]
    \item $\mathcal{S}^{E_1,\delta}(W)_{<\omega} \cong \mathcal{S}^{E_1,\delta}(W_{\leq \omega})$. 
    \item $\mathcal{S}^{E_1,\delta}(W)_{>\omega} \cong \mathcal{S}^{E_1,\delta}(W_{\geq \omega})$.
\end{enumerate}
Moreover for $(\omega',t',\epsilon')>(\omega,t,\epsilon)$ we have $\mathcal{S}^{E_1,\delta}(W)_{\omega<-<\omega'} \cong \mathcal{S}^{E_1,\delta}(W_{\omega \leq -\leq \omega'})$.
\end{corollary} 

\begin{proof}
Part (i): the map is given by sending $(\Tilde{\omega}<\omega) \mapsto e \circ \Tilde{\omega}$ and reparametrizing its ``thickening coordinate'' so that its centre is at coordinate $\Tilde{t}/t$ and is compatible with the product. 
Let us check that it indeed defines a wall: 
conditions (i), (ii) and (iii) of Definition \ref{definition splitting complex} hold automatically. 
Condition (iv) holds because $W_{\leq \Tilde{\omega}}= (W_{\leq \omega})_{\leq \Tilde{\omega}}$ and $W_{\Tilde{\omega} \leq - \leq \omega}= (W_{\leq \omega})_{\geq \Tilde{\omega}}$, so both have non-zero n-th homology. 
 
The fact that the correspondence is a bijection is because we can write an inverse map by precomposing with $e^{-1}$ and reparametrizing. 

Part (ii) is similar to part (i). 

For the moreover part we firstly apply part (i) to identify $\mathcal{S}^{E_1,\delta}(W)_{<\omega'} \cong \mathcal{S}^{E_1,\delta}(W_{\leq \omega'})$ and then part (ii) on the right hand side with $\omega$. 
\end{proof}

In a similar flavour we have the following lemma which allows us to make sense of the splitting complex of the manifold obtained by cutting a collection of arcs, and to give an interpretation to it. 
It will be very useful in the proof of Theorem \ref{theorem splitting poset connectivity}.
Before stating it we will need a small variation of the discretized splitting complex which will also be useful in the next section. 

\begin{definition}
For $W \in \mathcal{M}[A]$ we define the discrete poset
$\Tilde{\mathcal{S}}(W)$ to be same as $\mathcal{S}^{E_1,\delta}(W)$ but changing condition (1).(iv) on each wall $\omega$ to $H_n(W_{\leq \omega}) \neq 0$ only, so there is no condition on $H_n(W_{\geq \omega})$. 
The ordering relation is identical to the one of Definition \ref{definition splitting complex}.
\end{definition}

\begin{lemma} \label{lemma splitting cut manifold}
Let $(W,\Delta)$ be a valid geometric data where $W \in \mathcal{M}[A]$ and let $\alpha=\{a_0,\cdots,a_p\}$ be a $p$-simplex in $\mathcal{A}(W,\Delta)$. 
Denote by $A'$ the result of performing surgery on $A$ along some appropriate framings of the boundaries $\partial a_i$ for $0 \leq i \leq p$; then $A'$ is valid, there is $W' \in \mathcal{M}[A']$ such that $\mathcal{D}(W')=\mathcal{D}(W)$ and
\begin{enumerate}[(i)]
    \item The interiors of $W'$ and $W \setminus \alpha$ are diffeomorphic via a diffeomorphism $\phi$ which is the identity in a neighbourhood of $J_{2n-1} \cup \mathcal{D}(W)$. 
    \item Extending $\phi$ by the identity on $J_{2n-1} \cup \mathcal{D}(W)$ induces a bijection between the poset of walls $\omega \in \mathcal{S}^{E_1,\delta}(W)$ such that $\im(a_i)$ lies strictly to the right of $\im(\omega)$ for $0 \leq i \leq p$ and the poset $\Tilde{\mathcal{S}}(W')$.  
\end{enumerate}
\end{lemma}

The reason why we need $\Tilde{\mathcal{S}}(W')$ in the statement is that the wall $\omega'$ we get in $W'$ might satisfy that $H_n(W'_{\geq \omega'})=0$. 
However, condition $H_n(W'_{\leq \omega'}) \neq 0$ will be guaranteed because by the construction of the bijection in the following proof we have $W'_{\leq \omega'} \cong W_{\leq \omega}$. 

\begin{proof}
For each arc $a_i$ pick a trivialization of its normal bundle to get an embedding $\Tilde{a_i}: D^n \times D^n \hookrightarrow W$ such that $\Tilde{a_i}|_{D^n \times \{0\}}=a_i$, $\Tilde{a_i}(\partial D^n \times D^n) \subset \int(A)$, $\Tilde{a_i}$ intersects $\partial W$ transversally and $\Tilde{a_i}^{-1}(\partial W)= \partial D^n \times D^n$. 
Shrink the thickenings if needed to ensure that the images of the $\Tilde{a_i}$ are pairwise disjoint. 
Let $\Tilde{W}:= W \setminus \bigsqcup_{i=0}^p{\Tilde{a_i}(D^n \times \int(D^n))}$. 
Then $\Tilde{W}$ is a compact $2n$-manifold with boundary, whose boundary is the result of performing surgeries to $\partial W$ along the embeddings $\Tilde{a_i}|_{\partial D^n \times D^n}$; and hence we can decompose it as $J_{2n-1} \cup \mathcal{D}(W) \cup A'$, where $A'$ is the result of performing the surgeries on $A$. 
Since $A$ is $(n-2)$-connected then so is $A'$, and since the surgeries take place in the interior of $A$ then $A$ and $A'$ agree on a neighbourhood of their boundary. 
Moreover, $\int(\Tilde{W})$ is diffeomorphic to $\int(W \setminus \alpha)$ via a diffeomorphism $\Tilde{\phi}$ which can be chosen to be the identity near $J_{2n-1} \cup \mathcal{D}(W)$. 

Using isotopy extension, we pick an isotopy from the inclusion $A' \hookrightarrow I^{2n} \times \mathbb{R}^{\infty}$ to an embedding with image contained in the rightmost face $\{1\} \times I^{2n-1} \times \mathbb{R}^{\infty}$ such that the isotopy is constant near the boundary of $A'$. 
We will abuse notation and call its image $A'$ too. 
We then extend the isotopy to one on $\Tilde{W}$ such that it is constant on $J_{2n-1} \cup \mathcal{D}(W)$, and we let $W' \subset I^{2n} \times \mathbb{R}^{\infty}$ be the resulting manifold at the end of the isotopy. 
One can proceed as in Lemma \ref{lemma pieces splitting complex} to ensure that $W'$ has a product structure near the boundary compatible with the one on $J_{2n-1} \cup \mathcal{D}(W)$. 
We let $\phi$ be the composition of $\Tilde{\phi}$ with the end map of the isotopy, so that property (i) is satisfied. 

For part (ii) the bijection is given as follows: 
take a wall $\omega \in \mathcal{S}^{E_1}(W)$ such that $\im(a_i)$ lies strictly to the right of $\im(\omega)$ for all $i$, then  $\im(\omega) \subset W \setminus \alpha$, so we can precompose it with $\phi^{-1}$ so that it lies in $W'$, and then it will define a wall there because $\phi$ is the identity near $J_{2n-1} \cup \mathcal{D}(W)$, so the standard part of the wall and the product structure are preserved. 
The inverse map is given analogously but by precomposing with $\phi$ instead. 
\end{proof}

\subsection{Connectivity of discretized splitting complexes} \label{section 5.5}

The goal of this section is to show that for any $W \in \Mhalf$, its discretized splitting poset is $f$-weakly Cohen-Macaulay of dimension $g(W)-2$ for some appropriate $f$. 
In order to do so we will do an inductive argument on $W,$ and we will allow $W \in \mathcal{M}[A]$ at the inductive step for a general $A$ as in Definition \ref{definition moduli}; so we will in fact prove a more general result. 

In order to state this general result let us define \textit{genus} for a general $W \in \mathcal{M}[A]$, generalizing the case $A=I^{2n-1}$ treated in  Section \ref{section 3.4}: the pair
$(H_n(W),\lambda_W)$ defines a skew-symmetric bilinear form, and the \textit{genus} of $W$ is
$$g(W):= \sup\{g: \exists \; \text{morphism} \; \phi: H^{\oplus g} \rightarrow (H_n(W),\lambda_W)\},$$
which recovers the previous definition when $A=I^{2n-1}$. 
We will study skew symmetric forms in more detail and show some basic properties of the genus in Section \ref{section skew sym}; but for now the property that we need is: 
$g(W)=g(W_{\leq \omega})+g(W_{\geq \omega})$, which is a consequence of the additivity of the genus under orthogonal direct sums, see Section \ref{section skew sym}, plus the decomposition $(H_n(W),\lambda_W) \cong (H_n(W_{\leq \omega}), \lambda_{W_{\leq \omega}}) \oplus (H_n(W_{\geq \omega}), \lambda_{W_{\geq \omega}})$ from Mayer-Vietoris.

\begin{theorem} \label{theorem splitting poset connectivity}
Let $W \in \mathcal{M}[A]$, then 
\begin{enumerate}[(i)]
    \item $\mathcal{S}^{E_1, \delta}(W)$ is $(g(W)-3+C(A))$-connected, where $C(A)=0$ if $H_{n-1}(A)=0$ and $C(A)=1$ otherwise.  
    \item $\Tilde{\mathcal{S}}(W)$ is $(g(W)-2)$-connected. 
\end{enumerate}
\end{theorem} \label{rem splitting poset CM}

Before proving it let us show a technical lemma. 

\begin{lemma} \label{lem augmented splitting complex}
\begin{enumerate}[(i)]
    \item If $H_{n-1}(A) \neq 0$ then $\Tilde{\mathcal{S}}(W)= S^{E_1,\delta}(W)$.
    \item If $H_{n-1}(A)=0$ then $\mathcal{S}^{E_1,\delta}(W)$ is a subposet of $\Tilde{\mathcal{S}}(W)$ and the poset $\Tilde{\mathcal{S}}(W) \setminus \mathcal{S}^{E_1,\delta}(W)$ has no relations, and for any $\omega \in \Tilde{\mathcal{S}}(W) \setminus \mathcal{S}^{E_1,\delta}(W)$ we have $\Tilde{\mathcal{S}}(W)_{>\omega}= \emptyset$ and $\Tilde{\mathcal{S}}(W)_{< \omega} \cong \mathcal{S}^{E_1,\delta}(W_{\leq \omega})$. 
    \item If $H_{n-1}(A)=0$ then the inclusion $\mathcal{S}^{E_1,\delta}(W) \hookrightarrow \Tilde{\mathcal{S}}(W)$ induces the zero map in homotopy groups.  
\end{enumerate}
\end{lemma}

Parts (i),(ii) are some basic properties, whereas part (iii) is more technical and it will be used in the proof of Theorem \ref{theorem splitting poset connectivity} to deduce that $\Tilde{\mathcal{S}}(W)$ has more connectivity than $\mathcal{S}^{E_1,\delta}(W)$ when $H_{n-1}(A)=0$. 

\begin{proof}
There is a natural inclusion of posets $\mathcal{S}^{E_1,\delta}(W) \subset \Tilde{\mathcal{S}}(W)$. 

\textbf{Part (i).} We will show that the inclusion is surjective. 
Let $\omega \in \Tilde{\mathcal{S}}(W)$, we will show that $\omega \in \mathcal{S}^{E_1,\delta}(W)$, i.e. that $H_n(W_{\geq \omega}) \neq 0$. 

By Lemma \ref{lemma pieces splitting complex} we can view $W_{\geq \omega}$ as an element in $\mathcal{M}[A]$, so its boundary is the union of $A$ and a contractible $(2n-1)$-manifold along $\partial A = \partial I^{2n-1}$, hence $H_{n-1}(\partial W_{\geq \omega}) \cong H_{n-1}(A) \neq 0$. 
Since $W_{\leq \omega}$ is $(n-1)$-connected, the homology long exact sequence of $(W_{\leq \omega},\partial W_{\leq \omega})$ gives that $H_n(W_{\leq \omega},\partial W_{\leq \omega}) \rightarrow H_{n-1}(A)$ is surjective, and so by Poincaré-Lefschetz it follows that $H^n(W_{\geq \omega}) \neq 0$. 
Thus $H_n(W_{\geq \omega}) \neq 0$ by the universal coefficients theorem and the $(n-1)$-connectivity of $W_{\geq \omega}$. 

\textbf{Part (ii).}
For any $\omega \in \Tilde{\mathcal{S}}(W) \setminus \mathcal{S}^{E_1,\delta}(W)$ we have $H_{n}(W_{\geq \omega})=0$, and since the piece in between two walls has non-zero n-th homology it follows that $\mathcal{S}^{E_1,\delta}(W)_{> \omega}=\emptyset$ and that $\Tilde{\mathcal{S}}(W) \setminus \mathcal{S}^{E_1,\delta}(W)$ has no relations. 
By an analogous argument to Corollary \ref{cor pieces of splitting complex}, $\Tilde{\mathcal{S}}(W)_{< \omega} \cong \mathcal{S}^{E_1,\delta}(W_{\leq \omega})$. 

\textbf{Part (iii).}
By simplicial approximation and considering the map of associated simplicial complexes it suffices to show that given any finite collection of elements $\omega_1, \cdots, \omega_u \in \mathcal{S}^{E_1,\delta}(W)$ we can pick $\omega \in \Tilde{\mathcal{S}}(W)$ such that $\omega_i < \omega$ for $1 \leq i \leq u$. 

By (1).(i) in Definition \ref{definition splitting complex}, each $\omega_i$ is disjoint from $A$, hence we can use the product structure of $W$ to pick $\epsilon>0$ such that $[1-\epsilon,1] \times A \subset W$ is contained in $W_{\geq \omega_i}$ for all $i$. 
Since by assumption $A$ is $(2n-1)$-dimensional, $(n-2)$-connected, has boundary $\partial I^{2n-1}$ and $H_{n-1}(A)=0$ then Poincaré-Lefschetz duality and Whitehead's theorem imply that $A$ is contractible. 
By the $h$-cobordism theorem we find that $A$ is diffeomorphic to $I^{2n-1}$, but not necessarily diffeomorphic to $I^{2n-1}$ relative to $\partial A= \partial I^{2n-1}$ (because of the existence of exotic spheres in generic even dimensions). 

Fix a diffeomorphism $\phi: I^{2n-1} \rightarrow A$, which by isotopy extension we can assume to be the identity on a neighbourhood of $\overline{\partial I^{2n-1} \setminus I^{2n-2} \times \{1\}}$.
Moreover, by geodesic interpolation we can further isotope it so that it preserves the natural product structure of $I^{2n-1}$ near its top face $I^{2n-2} \times \{1\}$, as in Lemma \ref{lem classifying space}. 
Then the composition $\omega$ given by $[1-\epsilon/2,1-\epsilon/4] \times I^{2n-1} \xrightarrow{\id \times \phi} [1-\epsilon/2,1-\epsilon/4] \times A \hookrightarrow W$ defines a wall in $\Tilde{\mathcal{S}}(W)$ lying to the right of all the $\omega_i$.

Finally, to check that $\omega_i<\omega$ we just need to check that $H_n(W_{\omega_i \leq - \leq \omega}) \neq 0$: we have $H_n(W_{\geq \omega_i})=H_n(W_{\omega_i \leq - \leq \omega}) \oplus H_n(W_{\geq \omega})$, and the left-hand-side is non-zero by assumption whereas the second summand of the right-hand-side vanishes, giving the result. 
\end{proof}

\begin{rem} \label{remark D(W)}
The proof of part (iii) explains why in Definition \ref{definition moduli} we wanted to allow $\mathcal{D}(W)$ to be a homotopy disc in general: the wall $\omega$ constructed above cannot be chosen to be standard near all its boundary if $A$ is not diffeomorphic to $I^{2n-1}$ relative its whole boundary. 
Thus we need to allow the walls to be non-standard on the top face, which means that the boundaries of the left pieces of the walls are allowed to be homotopy spheres. 
As we will see in the proof of Proposition \ref{prop hofib} this suggests that the $E_k$-algebra under study should contain manifolds whose boundary is a homotopy sphere too. 
\end{rem}

Now we can finally proof the main result of the section.
Before doing so, let us remark that the proof will make use not only of the arc complex connectivity but also of Proposition \ref{prop cut manifold}, which gives bounds for the genus and the rank of the n-th homology of the cut manifold.
As we will see in Remark \ref{rem n odd} the proof of Proposition \ref{prop cut manifold} uses that $n$ is odd in a fundamental way. 

\begin{proof}[Proof of Theorem \ref{theorem splitting poset connectivity}]
We will prove the result by induction on $\rk(H_n(W))$:
when $\rk(H_n(W))=0$ both (i) and (ii) are vacuously true since $g(W)=0$ too. 

\textbf{Step 1.} We will show part (ii) assuming that part (i) holds for any $W' \in \mathcal{M}[A']$ such that $\rk(H_n(W')) \leq \rk(H_n(W))$, where $A'$ is any valid choice. 
We have two cases: 

If $H_{n-1}(A) \neq 0$ then Lemma \ref{lem augmented splitting complex} (i) gives $\Tilde{\mathcal{S}}(W)=\mathcal{S}^{E_1,\delta}(W)$, which by induction  is $(g(W)-2)$-connected, as required. 

If $H_{n-1}(A)=0$ then view $\mathcal{S}^{E_1,\delta}(W)$ and $\Tilde{\mathcal{S}}(W)$ as simplicial complexes and consider the inclusion $\mathcal{S}^{E_1,\delta}(W) \hookrightarrow \Tilde{\mathcal{S}}(W)$. 
The smaller complex is a full subcomplex of the other one; and if $\sigma$ is a $p$-simplex in $\Tilde{\mathcal{S}}(W)$ with no vertex in $\mathcal{S}^{E_1,\delta}(W)$ then we must have $p=0$ by Lemma \ref{lem augmented splitting complex}(ii) so that $\sigma=\{\omega\}$.
Also, $\Lk(\sigma) \cap \mathcal{S}^{E_1,\delta}(W)$ is the simplicial complex corresponding to the poset $\mathcal{S}^{E_1,\delta}(W_{\leq \omega})$ by Lemma \ref{lem augmented splitting complex}(ii), which is $(g(W_{\leq \omega})-3)$-connected by induction. 
Moreover, additivity of the genus and $\omega \notin \mathcal{S}^{E_1,\delta}(W)$ implies that $g(W_{\leq \omega})=g(W)$. 
Therefore the inclusion $\mathcal{S}^{E_1,\delta}(W) \hookrightarrow \Tilde{\mathcal{S}}(W)$ is $(g(W)-2)$-connected by \cite[Proposition 2.5]{high}. 
Thus, by Lemma \ref{lem augmented splitting complex}(iii) we get that $\Tilde{\mathcal{S}}(W)$ is $(g(W)-2)$-connected, as required. 

\textbf{Step 2.} We will show part (i) assuming both parts (i) and (ii) hold for any $W' \in \mathcal{M}[A']$ such that $\rk(H_n(W')) < \rk(H_n(W))$, where $A'$ is any valid choice. 
We will apply the ``nerve theorem'', \cite[Corollary 4.2]{E2}. 
To do so, we will use the arc complex to index an appropriate family of closed subposets of 
$\mathcal{S}^{E_1,\delta}(W)$: 

Pick $\delta$ to be a generator of a direct summand of $H_{n-1}(\partial W) \cong H_{n-1}(A)$ of maximal order, so $\delta=0$ when $H_{n-1}(A)=0$, and use that $A$ is $(n-2)$-connected and the Hurewicz theorem to represent $\delta$ by a continuous map $f: S^{n-1} \rightarrow \int(A)$. 
Then use transversality to homotope $f$ to a smooth embedding and let $\Delta$ be its isotopy class. 
In principle this isotopy class depends on the choice of smooth embedding, but we just pick one.

Consider the functor of posets
$$F: \mathcal{A}(W,\Delta)^{\text{op}} \rightarrow \{\text{closed subposets of} \; \mathcal{S}^{E_1,\delta}(W)\} $$
given by sending $\alpha= \{a_0,\cdots,a_p\}$ to the subposet of $\mathcal{S}^{E_1,\delta}(W)$ consisting of walls $\omega$ such that all the $a_i$'s lie strictly to the right of $\omega$. 

Now let us verify the hypotheses of \cite[Corollary 4.2]{E2} with $n=g(W)-2+C(A)$ (not to be confused with half the dimension of the manifolds in this proof), $t_{\mathcal{A}}(\alpha)=p$ where $\alpha=\{a_0, \cdots, a_p\}$, and $t_{\mathcal{S}}(\omega)=g(W_{\leq \omega})-1$.

(i) By Theorem \ref{theorem arc complex connectivity} the poset $\mathcal{A}(W,\Delta)$ is $g(W)-2$-connected, and in particular $(n-1)$-connected. 

(ii) Let $\alpha=\{a_0, \cdots,a_p\} \in \mathcal{A}(W,\Delta)$, then $\mathcal{A}(W,\Delta)_{< \alpha}$ is the boundary of a $p$-simplex, and so $(p-2)=(t_{\mathcal{A}}(\alpha)-2)$-connected. 

On the other hand, by Lemma \ref{lemma splitting cut manifold} the poset $F(\alpha)$ is isomorphic to $\Tilde{\mathcal{S}}(W \setminus \alpha)$. 
By induction hypothesis, which applies by Proposition \ref{prop cut manifold}(i), this is  $(g(W \setminus \alpha)-2)$-connected.
By our choice of $\Delta$ and Proposition \ref{prop cut manifold}(ii) it follows that $g(W \setminus \alpha) \geq g(W)-(p+1)+C(A)$.
Thus, $F(\alpha)$ is $(n-(t_{\mathcal{A}}(\alpha)+1))$-connected, as required. 

(iii) Let $\omega \in \mathcal{S}^{E_1,\delta}(W)$.
By Corollary \ref{cor pieces of splitting complex}, $\mathcal{S}^{E_1,\delta}(W)_{<\omega} \cong \mathcal{S}^{E_1,\delta}(W_{\leq\omega})$. 
Now we claim that $\rk(H_n(W_{\leq \omega}))<\rk(H_n(W))$. 

Since $H_n(W)=H_n(W_{\leq \omega}) \oplus H_n(W_{\geq \omega})$ it suffices to show that $\rk(H_n(W_{\geq \omega}))>0$. 
Since $\omega \in \mathcal{S}^{E_1,\delta}(W)$ then $H_n(W_{\geq \omega}) \neq 0$, so it suffices to show that this group is torsion-free. 
By the universal coefficients theorem it is enough to show that $H^{n+1}(W_{\geq \omega})= 0$, and by Poincaré-Lefschetz duality it suffices to show that $H_{n-1}(W_{\geq \omega},\partial W_{\geq \omega}) = 0$, which follows from the $(n-1)$-connectivity of $W_{\geq \omega}$ and the $(n-2)$-connectivity of $\partial W_{\geq \omega}$ (see Lemma \ref{lemma pieces splitting complex}). 

Thus, induction hypothesis applies to $W_{\leq \omega}$ giving that $\mathcal{S}^{E_1,\delta}(W)_{<\omega}$ is $(g(W_{\leq \omega})-3)$-connected, i.e. $(t_{\mathcal{S}}(\omega)-2)$-connected. 

On the other hand, the poset $\mathcal{A}(W,\Delta)_{\omega}:=\{\alpha \in \mathcal{A}(W,\Delta): \; \omega \in F(\alpha)\}$ is isomorphic to $\mathcal{A}(W_{\geq \omega},\Delta)$.
Thus, by Theorem \ref{theorem arc complex connectivity} it is $(g(W_{\geq \omega})-2)$-connected. 
By additivity of the genus we find that $\mathcal{A}(W,\Delta)_{\omega}$ is $(n-C(A)-(t_{\mathcal{S}}(\omega)+1))$-connected, as required since $C(A) \leq 1$. 

The nerve theorem applies giving that $\mathcal{S}^{E_1,\delta}(W)$ is $(g(W)-3+C(A))$-connected, completing the induction step.  
\end{proof}

\begin{corollary} \label{cor weakly cohen macaulay}
For any allowed $A$ and $W \in \mathcal{M}[A]$, $\mathcal{S}^{E_1,\delta}(W)$ is $f$-weakly Cohen-Macaulay of dimension $g(W)-2+C(A)$, where $f(\omega)=g(W_{\leq \omega})-1$. 
\end{corollary}

\begin{proof}
By Theorem \ref{theorem splitting poset connectivity} we already know that $\mathcal{S}^{E_1,\delta}(W)$ is $(g(w)-3+C(A))$-connected. 
From Theorem \ref{theorem splitting poset connectivity} and Corollary \ref{cor pieces of splitting complex} it follows that 

(i) For any $\omega \in \mathcal{S}^{E_1,\delta}(W)$, $\mathcal{S}^{E_1,\delta}(W)_{\leq \omega}$ is $(g(W_{\leq \omega})-3)$-connected, i.e. $(f(\omega)-2)$-connected. 

(ii) For any $\omega \in \mathcal{S}^{E_1,\delta}(W)$, $\mathcal{S}^{E_1,\delta}(W)_{\geq \omega}$ is $(g(W_{\geq \omega})-3+C(A))$-connected, which by additivity of the genus is $((g(W)-2+C(A))-2-f(\omega))$-connected. 

(iii) For any $\omega < \omega' \in \mathcal{S}^{E_1,\delta}(W)$, $\mathcal{S}^{E_1,\delta}(W)_{\omega <-<\omega'}$ is $(g(W_{\omega \leq - \leq \omega'})-3)$-connected. 
By additivity of the genus this is $(f(\omega')-f(\omega)-3)$-connected. 
\end{proof}

\begin{rem}
It can be shown that $f(\omega)=\dim(\omega)$ and $\dim(\mathcal{S}^{E_1,\delta}(W))=g(W)-2+C(A)$, so actually $\mathcal{S}^{E_1,\delta}(W)$ is Cohen-Macaulay of dimension $g(W)-2+C(A)$. 
However this fact will not be used in the rest of this paper. 
\end{rem}

\subsection{The discretization argument and the connectivity of splitting complexes}

In this section we give a general tool which allows to show high connectivity of the nerve of a topological poset by proving that its discretization is $f$-weakly Cohen-Macaulay.
The statement and proof of the following result are inspired by the proof of \cite[Theorem 5.6]{high}. 

\begin{theorem}\label{poset discretization}
Let $(P,<)$ be a (non-unital) topological poset such that $P$ is a Hausdorff space and the relation $<$ is an open subset of $P^2$.
If $(P^{\delta},<)$ is $f$-weakly Cohen-Macaulay of dimension $n$ for some function $f$, then $||N_{\bullet}P||$ is $(n-1)$-connected.  
\end{theorem}

\begin{proof}
Let $\iota: P^{\delta} \rightarrow P$ be the identity map, viewed as a continuous map of posets. 

Consider the bi-semisimplicial space $A_{\bullet,\bullet}$ whose space of $(p,q)$-simplices is given by 
\begin{equation*}
    \begin{split}
       A_{p,q}:=\{(x_0 < x_1 <\cdots < x_p), (y_0< y_1< \cdots < y_q): \; \text{the set} \; \{x_0,\cdots,x_p,y_0,\cdots,y_q\} \\
        \text{is totally ordered}\} 
    \end{split}
\end{equation*}
topologized as a subspace of $P^{p+1} \times (P^{\delta})^{q+1}$, and whose face maps are given by forgetting the $x_i$'s and the $y_j$'s. 

This bi-semisimplicial space has two augmentations $\epsilon_{\bullet}: A_{\bullet,\bullet} \rightarrow P_{\bullet}=N_{\bullet}P$ and $\delta_{\bullet}: A_{\bullet,\bullet} \rightarrow P^{\delta}_{\bullet}=N_{\bullet}P^{\delta}$. 

\begin{claim}
The diagram 

\centerline{\xymatrix{ & ||A_{\bullet,\bullet}|| \ar[ld]_-{||\delta_{\bullet}||} \ar[rd]^-{||\epsilon_{\bullet}||} & \\
||P^{\delta}_{\bullet}|| \ar[rr]^-{||\iota_{\bullet}||} & & ||P_{\bullet}|| }}

is homotopy-commutative. 
\end{claim}

This claim is similar to \cite[Lemma 4.2]{ssSpaces}

\begin{proof}[Proof of Claim.]
For each $p,q \geq 0$ we will construct a map $H_{p,q}: I \times A_{p,q} \times \Delta^p \times \Delta^q \rightarrow ||P_{\bullet}||$ in such a way that the $H_{p,q}$'s are compatible with the face maps, and hence glue to produce the required homotopy. \\
We define $H_{p,q}$ on a given $(p,q)$-simplex $ \sigma= ((x_0<\cdots<x_p),(y_0<\cdots<y_q)) \in A_{p,q}$ to be the map $H_{p,q}(\sigma): I \times \Delta^p \times \Delta^q \rightarrow \Delta^{p+q+1}$ given as follows:

$\epsilon_{\bullet}$ at simplex $\sigma$ gives an inclusion $a:\Delta^p \hookrightarrow \Delta^{p+q+1}$ as a $p$-face corresponding to the positions of the $x_i$'s in the chain defined by $\{x_0, \cdots, x_p\} \sqcup \{y_0, \cdots, y_q\}$, and similarly $\iota_{\bullet} \circ \delta_{\bullet}$ at simplex $\sigma$ gives an inclusion $b:\Delta^q \hookrightarrow \Delta^{p+q+1}$ as a $q$-face. 

The homotopy $H$ is then given by $H(t,u,v):=ta (u)+(1-t) b(v)$ for $u \in \Delta^p$, $v \in \Delta^q$ and $t \in I$. 
\end{proof}

Thus, using that $P^{\delta}_{\bullet}$ is $(n-1)$-connected, it suffices to show that $||\epsilon_{\bullet}||$ is $(n-1)$-connected. 

By \cite[Proposition 2.7]{high} it suffices to show that for each $p$ the map $\epsilon_p: |A_{p,\bullet}| \rightarrow P_p$ is $(n-1-p)$-connected. 
This will be shown by applying \cite[Corollary 2.9]{high} with $Y_{\bullet}=P^{\delta}_{\bullet}$, which is a semisimplicial set as ${P^{\delta}}$ is discrete, $Z=P_{p}$, which is Hausdorff as $P$ is, and $X_{\bullet}= A_{p,\bullet} \subset Z \times Y_{\bullet}$, which is an open subset on each degree as $<$ is open. 
The only condition to check is that for any given $x_0<\cdots<x_p \in P_p$ the semisimplicial set $X_{\bullet}(x_0<\cdots<x_p)$ defined via $\{x_0<\cdots<x_p\} \times X_{\bullet}(x_0<\cdots<x_p)= X_{\bullet} \cap (\{x_0<\cdots<x_p\} \times Y_{\bullet})$ is $(n-1-p-1)$-connected. 

By definition $X_{\bullet}(x_0<\cdots<x_p)=(P^{\delta}_{<x_0}*P^{\delta}_{x_0<-<x_1}* \cdots * P^{\delta}_{x_{p-1}<-<x_p} * P^{\delta}_{>x_p})_{\bullet}$, and since
$P^{\delta}$ is $f$-weakly Cohen-Macaulay of dimension $n$ then $P^{\delta}_{<x_0}$ is $(f(x_0)-2)$-connected, each $P^{\delta}_{x_{i-1}<-<x_i}$ is $(f(x_i)-f(x_{i-1})-3)$-connected and $P^{\delta}_{>x_p}$ is $(n-2-f(x_p))$-connected. 
Thus the join is $(n-2-2-3p+2(p+1))$-connected, i.e. $(n-2-p)$-connected, as required.
\end{proof}

From Corollary \ref{cor weakly cohen macaulay} and the above result it follows that

\begin{corollary} \label{cor key} 
For each $W \in \Mhalf$ the $E_1$-splitting complex $S^{E_1}(W)$ is $(g(W)-3)$-connected.
More generally, for any allowed $A$ and $W \in \mathcal{M}[A]$ the $E_1$-splitting complex $S^{E_1}(W)$ is $(g(W)-3+C(A))$-connected.
\end{corollary}

\subsection{Splitting complexes and $E_1$-homology} \label{section 5.7}
In this section we will relate the required vanishing in the $E_1$-homology of $\textbf{R}$ of Theorem \hyperref[theorem C]{C} to the connectivity bound on the splitting complex of Corollary \ref{cor key}.
This will be based on results of \cite[Section 13]{Ek}, which express the $E_1$-indecomposables in terms of bar constructions. 
We will make use of different bar constructions defined in \cite[Section 9, Section 13]{Ek}. 

Recall also \cite[Section 12.2.1]{Ek}: 
given an $E_1$-algebra $\textbf{X}$ we can define a unital associative replacement $\mathbf{\overline{X}}$ whose underlying object is given by $([0,\infty) \times \mathds{1}) \sqcup ((0,\infty) \times \mathbf{X})$, where $\mathds{1}$ is the monoidal unit of the underlying category. 
A small modification of this construction gives a non-unital associative replacement that we shall denote $\mathbf{X'}$, whose underlying object is $(0,\infty) \times \mathbf{X}$, and such that the inclusion $\mathbf{X'} \rightarrow \mathbf{\overline{X}}$ preserves the associative product. 
In particular, $\mathbf{X'}$ becomes an strict $\mathbf{\overline{X}}$-bimodule in the obvious way. 

Our first lemma relates the $E_1$-homology of $\textbf{R}$ to a certain bar construction

\begin{lemma} \label{lem Q vs B}
The homotopy cofibre of $B(\mathbf{R'},\mathbf{\bR},\mathbf{R'}) \rightarrow B(\mathbf{\bR},\mathbf{\bR},\mathbf{R'}) \simeq \mathbf{R'}$ is homologically $(\rk(x)-2)$-connected in grading $x \in \mathsf{G}_n$ if and only if $H_{x,d}^{E_1}(\mathbf{R})=0$ for $d<\rk(x)-1$. 
\end{lemma}

\begin{proof}
By \cite[Corollary 9.17,Theorem 13.7]{Ek} and using that $\mathbf{R'}_{\mathbb{Z}} \rightarrow \mathbf{R}_{\mathbb{Z}}$ is a weak equivalence of $\mathbf{R}_{\mathbb{Z}}$-modules we get that 
$Q_{\mathbb{L}}^{E_1}(\mathbf{R}_{\mathbb{Z}}) \simeq B(\mathds{1},\mathbf{\overline{R_{\mathbb{Z}}}},\mathbf{R_{\mathbb{Z}}'})$.
Since
$$B(\mathbf{R_{\mathbb{Z}}'},\mathbf{\overline{R_{\mathbb{Z}}}},\mathbf{R_{\mathbb{Z}}'}) \rightarrow B(\mathbf{\overline{R_{\mathbb{Z}}}},\mathbf{\overline{R_{\mathbb{Z}}}},\mathbf{R_{\mathbb{Z}}'}) \simeq \mathbf{R_{\mathbb{Z}}'} \rightarrow B(\mathds{1},\mathbf{\overline{R_{\mathbb{Z}}}},\mathbf{R_{\mathbb{Z}}'})$$
is a cofibration then $Q_{\mathbb{L}}^{E_1}(\mathbf{R_{\mathbb{Z}}})$ is the homotopy cofibre of $B(\mathbf{R_{\mathbb{Z}}'},\mathbf{\overline{R_{\mathbb{Z}}}},\mathbf{R_{\mathbb{Z}}'}) \rightarrow B(\mathbf{\overline{R_{\mathbb{Z}}}},\mathbf{\overline{R_{\mathbb{Z}}}},\mathbf{R_{\mathbb{Z}}'}) \simeq \mathbf{R_{\mathbb{Z}}'}$.  

Now perform the associative replacements $\mathbf{R'}$ and $\mathbf{\overline{R}}$ in the original category $\mathsf{Top}^{\mathsf{G}_n}$ and consider the map  $B(\mathbf{R'},\mathbf{\bR},\mathbf{R'}) \rightarrow B(\mathbf{\bR},\mathbf{\bR},\mathbf{R'}) \simeq \mathbf{R'}$. 
By construction, if we apply the functor $(-)_{\mathbb{Z}}$ then we get $B(\mathbf{R_{\mathbb{Z}}'},\mathbf{\overline{R_{\mathbb{Z}}}},\mathbf{R_{\mathbb{Z}}'}) \rightarrow B(\mathbf{\overline{R_{\mathbb{Z}}}},\mathbf{\overline{R_{\mathbb{Z}}}},\mathbf{R_{\mathbb{Z}}'}) \simeq \mathbf{R_{\mathbb{Z}}'}$. 

Thus, $H_{x,d}^{E_1}(\mathbf{R})=0$ for $d<\rk(x)-1$ if and only if $Q_{\mathbb{L}}^{E_1}(\mathbf{R_{\mathbb{Z}}})$ is homologically $(\rk(x)-2)$-connected in grading $x$, which is equivalent to the map $B(\mathbf{R_{\mathbb{Z}}'},\mathbf{\overline{R_{\mathbb{Z}}}},\mathbf{R_{\mathbb{Z}}'}) \rightarrow B(\mathbf{\overline{R_{\mathbb{Z}}}},\mathbf{\overline{R_{\mathbb{Z}}}},\mathbf{R_{\mathbb{Z}}'}) \simeq \mathbf{R_{\mathbb{Z}}'}$ being homologically $(\rk(x)-2)$-connected in grading $x$, which in turn is equivalent to the map $B(\mathbf{R'},\mathbf{\bR},\mathbf{R'}) \rightarrow B(\mathbf{\bR},\mathbf{\bR},\mathbf{R'}) \simeq \mathbf{R'}$ being itself homologically $(\rk(x)-2)$-connected in grading $x$. 
\end{proof}

In $\mathsf{Top}^{\mathsf{G}_n}$ we can also make sense of the bar construction $B_{\bullet}(\mathbf{R'},\mathbf{R'},\mathbf{R'})$, which has an augmentation $B_{\bullet}(\mathbf{R'},\mathbf{R'},\mathbf{R'}) \rightarrow \mathbf{R'}$ which factors trough the augmentation $B_{\bullet}(\mathbf{R'},\mathbf{\bR},\mathbf{R'}) \rightarrow \mathbf{R'}$.
Moreover, we also have a weak equivalence $\mathbf{R'} \rightarrow \mathbf{R}$ induced by projection. 

\begin{lemma} \label{lem B}
The map $B(\mathbf{R'},\mathbf{R'},\mathbf{R'}) \rightarrow B(\mathbf{R'},\mathbf{\bR},\mathbf{R'})$ induced by the inclusion $\mathbf{R'} \rightarrow \mathbf{\bR}$ is a weak equivalence. 
Thus, the homotopy cofibres of $B_{\bullet}(\mathbf{R'},\mathbf{R'},\mathbf{R'}) \rightarrow \mathbf{R}$ and $B(\mathbf{R'},\mathbf{\bR},\mathbf{R'}) \rightarrow \mathbf{R'}$ are equivalent. 
\end{lemma}

\begin{proof}
The object $(\mathbf{R'})^+:= \mathds{1} \sqcup \mathbf{R'}$ has an obvious unital associative algebra structure, and the inclusion $(\mathbf{R'})^+ \rightarrow \mathbf{\overline{R}}$ is a weak equivalence and a map of unital associative algebras, so it induces a map of semisimplicial objects $B_{\bullet}(\mathbf{R'},(\mathbf{R'})^+,\mathbf{R'}) \rightarrow B_{\bullet}(\mathbf{R'},\mathbf{\bR},\mathbf{R'})$, which is levelwise a weak equivalence, and hence a weak equivalence on geometric realizations by \cite[Theorem 2.2]{ssSpaces}.

Finally, $B_{\bullet}(\mathbf{R'},(\mathbf{R'})^+,\mathbf{R'})$ admits the structure of a simplicial graded space where degeneracies are given by inserting the unit $\mathds{1}$, which consists of a point $*$ in grading $0 \in \mathsf{G}_n$. 
Thus, this simplicial object is given by freely adding degeneracies to the semisimplicial object $B_{\bullet}(\mathbf{R'},\mathbf{R'},\mathbf{R'})$.
By \cite[Lemma 2.6]{ssSpaces} the map $B_{\bullet}(\mathbf{R'},\mathbf{R'},\mathbf{R'}) \rightarrow B_{\bullet}(\mathbf{R'},(\mathbf{R'})^+,\mathbf{R'})$ is a weak equivalence in geometric realizations, giving the result. 
\end{proof}

Denote by $\pi_{\bullet}$ the map $B_{\bullet}(\mathbf{R'},\mathbf{R'},\mathbf{R'}) \rightarrow \mathbf{R}$, and for each $W \in \textbf{R}$ let 
$$F(W):= \hofib_W(B(\mathbf{R'},\mathbf{R'},\mathbf{R'}) \xrightarrow{\pi} \mathbf{R})$$
be the corresponding homotopy fibre. 

\begin{proposition} \label{prop hofib}
For $W \in \textbf{R}$, $F(W)$ is weakly equivalent to $S^{E_1}(W)$.
\end{proposition}

\begin{proof}
For each $p \geq 0$ let $F_p(W):= \hofib_W(\pi_p: B_p(\mathbf{R'},\mathbf{R'},\mathbf{R'}) \rightarrow \textbf{R})$, where $\pi_p: (\mathbf{R'})^{\otimes p+2} \rightarrow \textbf{R}$ is given by using the associative product in $\mathbf{R'}$ followed by the weak equivalence $\mathbf{R'} \rightarrow \textbf{R}$. 
We will find explicit functorial models of the $F_p(W)$'s so that they define a semisimplicial space $F_{\bullet}(W)$, and then $F(W) \simeq || F_{\bullet}(W)||$. 

The space $\Emb_{\partial^- W}^p(W, I^{2n} \times \mathbb{R}^{\infty})$ considered in Definition \ref{definition moduli} is weakly contractible by the Whitney embedding theorem, and the map $\Emb_{\partial^- W}^p(W, I^{2n} \times \mathbb{R}^{\infty}) \rightarrow \mathcal{M}[I^{2n-1}]$ is a fibration by the isotopy extension theorem. 
Thus, we can fix a model for $F_p(W)$ via the pullback square 

\centerline{
\xymatrix{F_p(W) \ar[d] \ar[r] & \pi_p^{-1}(\mathcal{M}[I^{2n-1};W]) \subset (\mathbf{R'})^{\otimes p+2} \ar[d] \\
\Emb_{\partial^- W}^p(W, I^{2n} \times \mathbb{R}^{\infty}) \ar[r] & \mathcal{M}[I^{2n-1};W]}
}

where the bottom right corner means the path-component of $W$ in $\Mhalf$, which by Proposition \ref{prop classifying space} agrees with the image of $\Emb_{\partial^- W}^p(W, I^{2n} \times \mathbb{R}^{\infty}) \rightarrow \mathcal{M}[I^{2n-1}]$. 

Hence a model for $F_p(W)$ is the set of tuples
$((s_0,W_0), \cdots,(s_{p+1},W_{p+1});e)$
where $(s_i,W_i) \in \mathbf{R'}=(0,\infty) \times \mathbf{R}$ for all $i$, $e \in \Emb_{\partial^- W}^p(W, I^{2n} \times \mathbb{R}^{\infty})$, and 
$$(s_0,W_0) \bullet \cdots \bullet (s_{p+1},W_{p+1})= (s_0+\cdots+s_{p+1},\im(e)) \in \mathbf{R'}$$
where $\bullet$ denotes the product in $\mathbf{R'}$. 
Changing variables $t_i:=\frac{s_0+ \cdots +s_i}{s_0+ \cdots s_{p+1}} \in (0,1)$ for $0 \leq i \leq p$, we can also view $F_p(W)$ as the collection of tuples 
$$(0<t_0< \cdots <t_p< 1;W_0, \cdots, W_{p+1}; e) \in \Delta^{p} \times \textbf{R}^{p+2} \times \Emb_{\partial^- W}^p(W, I^{2n} \times \mathbb{R}^{\infty})$$
such that the $E_1$-multiplication of the $W_i$'s using the partition 
$$[0,t_0] \sqcup [t_0,t_1] \sqcup \cdots \sqcup [t_p,1] \in \mathcal{C}_1(p+2)$$
is precisely $\im(e) \in \mathcal{M}[I^{2n-1};W]$. 

With this explicit description, it is clear that $F_{\bullet}(W)$ defines a semisimplicial space, where the i-th face map forgets $t_i$ and glues together $W_i$ and $W_{i+1}$ using the $E_1$-product with partition $[0,\frac{t_{i}-t_{i-1}}{t_{i+1}-t_{i-1}}] \sqcup [\frac{t_{i}-t_{i-1}}{t_{i+1}-t_{i-1}},1]$. 

To finish the proof we will construct a zigzag of semisimplicial spaces
$$F_{\bullet}(W) \xleftarrow{U_{\bullet}} F'_{\bullet}(W) \xrightarrow{\Psi_{\bullet}} S_{\bullet}^{E_1}(W)$$
which are levelwise weak equivalences, and hence weak equivalences on geometric realizations by \cite[Lemma 2.6]{ssSpaces}. 

The semisimplicial space $F'_{\bullet}(W)$ is defined as follows: 
$F'_p(W)$ is the collection of tuples 
$$(0<t_0< \cdots t_p< 1;W_0, \cdots, W_{p+1}; e; 0<\epsilon_0,\cdots,\epsilon_p<1) \in F_p(W) \times (0,1)^{p+1}$$
such that 
$0<t_0-\epsilon_0< t_0+\epsilon_0<t_1-\epsilon_1< \cdots < t_p-\epsilon_p < t_p+\epsilon_p <1$ and for $0 \leq i \leq p$, $\im(e)$ agrees pointwise with $I^{2n} \times \{0\} \subset I^{2n} \times \mathbb{R}^{\infty}$ in an open neighbourhood of $[t_i-\epsilon_i,t_i+\epsilon_i] \times I^{2n-1} \times \mathbb{R}^{\infty}$. 

The map $U_p: F'_p(W) \rightarrow F_p(W)$ is given by forgetting the extra data $\epsilon_0,\cdots,\epsilon_p$. 
$U_p$ has contractible fibres, and it is a microfibration because if we pick some valid $(\epsilon_0,\cdots,\epsilon_1)$ for a point in $F_p(W)$ then the same choice of $(\epsilon_0,\cdots,\epsilon_1)$ works in a neighbourhood of that point. 
Thus, $U_p$ is a weak equivalence by \cite[Lemma 2.2]{weiss}. 

The map $\Psi_p$ is defined via
$$(0<t_0< \cdots t_p< 1; W_0, \cdots, W_{p+1}; e; 0<\epsilon_0,\cdots,\epsilon_p<1) \mapsto ((\omega_0,t_0,\epsilon_0), \cdots , (\omega_p,t_p,\epsilon_0))$$
where the wall
$\omega_i$ is given by the composition $[t_i-\epsilon_i,t_i+\epsilon_i] \times I^{2n-1} \hookrightarrow \im(e) \xrightarrow{e^{-1}} W$. 
By the isotopy extension theorem, $\Psi_p: F'_p(W) \rightarrow
{S}_p^{E_1}(W)$ is a Serre fibration, so it suffices to check that it has weakly contractible fibres. 

The fibre over a given $((\omega_0,t_0,\epsilon_0), \cdots , (\omega_p,t_p,\epsilon_0)) \in S^{E_1}_p(W)$ is the subspace of $e \in \Emb_{\partial^- W}^p(W, I^{2n} \times \mathbb{R}^{\infty})$ such that the compositions $e \circ \omega_i$ agree with the standard inclusions $[t_i-\epsilon_i,t_i+\epsilon_i] \times I^{2n-1} \times \{0\} \subset I^{2n} \times \mathbb{R}^{\infty}$ for all $i$.  
This space is then weakly contractible by the Whitney embedding theorem since the images of the $\omega_i$ form a closed submanifold. 
\end{proof}

Now we can finally prove Theorem \hyperref[theorem C]{C}.

\begin{proof}
For each $W \in \Mhalf$ the $E_1$-splitting complex $S^{E_1}(W)$ is $(g(W)-3)$-connected by Corollary \ref{cor key}.
Hence Proposition \ref{prop hofib} implies that $F(W)$ is $(g(W)-3)$-connected for any $W \in \textbf{R}$. 
Thus, the cofibre of $B(\mathbf{R'},\mathbf{R'},\mathbf{R'}) \xrightarrow{\pi} \mathbf{R}$ is homologically $(\rk(x)-2)$-connected in degree $x$. 
Then Lemmas \ref{lem Q vs B} and \ref{lem B} give the result. 
\end{proof}

\section{The connectivity of the arc complex} \label{section 6}
In order to finish the proof of Theorem \hyperref[theorem C]{C} we need to show Theorem \ref{theorem arc complex connectivity} and Proposition \ref{prop cut manifold}, which are the only results used in Section \ref{section 5} that are left to prove. 
To do so we will firstly construct an algebraic model for the arc complex, and then we will compare it to the geometric one. 

\subsection{The algebraic arc complex} \label{section alg complex}

\begin{definition} \label{definition valid algebraic data}
A \textit{valid algebraic data} consists of a triple $(M,\lambda,\delta)$, where
\begin{enumerate}[(i)]
    \item $M$ is a finitely generated free $\mathbb{Z}$-module.
    \item $\lambda: M \otimes M \rightarrow \mathbb{Z}$ is a skew-symmetric bilinear form on $M$. We write $\lambda^{\vee}: M \rightarrow M^{\vee}$ for the corresponding map $m \mapsto \lambda(m,-)$. 
    \item $\delta \in \partial (M,\lambda)$ is an element, where $\partial (M,\lambda):= \coker(\lambda^{\vee})=\frac{M^{\vee}}{\lambda^{\vee}(M)}$.
\end{enumerate}
\end{definition}

We will usually remove $\lambda$ from the notation in all future expressions to make them easier to read, for example we shall write $\partial M$ instead of $\partial (M,\lambda)$.

Any valid geometric data $(W,\Delta)$ as in Definition \ref{definition valid geometric data} gives a valid algebraic data $(H_n(W),\lambda_W,\delta)$ as follows: compactness of $W$ implies that $H_n(W)$ is finitely generated, 
the long exact sequence of the pair $(W,\partial W)$ and the connectivity assumptions on $W$ and $\partial W$ give an exact sequence
$$ 0 \rightarrow H_{n}(\partial W) \rightarrow H_n(W) \rightarrow H_n(W,\partial W) \rightarrow H_{n-1}(\partial W) \rightarrow 0$$
and that $H_{n-1}(W,\partial W)=0$. 
By Poincaré-Lefschetz duality we have isomorphisms $0=H_{n-1}(W,\partial W) \cong H^{n+1}(W)$ and
$H_n(W,\partial W) \cong H^n(W)$.
The universal coefficients theorem implies that $H_n(W)$ is torsion-free and $H^n(W) \cong H_n(W)^{\vee}$.
Composing the isomorphisms $H_n(W,\partial W) \cong H^n(W) \cong H_n(W)^{\vee}$, the inclusion $H_n(W) \rightarrow H_n(W,\partial W)$ becomes $\lambda_W^{\vee}: H_n(W) \rightarrow H_n(W)^{\vee}$.
Thus, the above exact sequence implies that $H_n(\partial W) \cong \ker(\lambda_W^{\vee})$ and that $H_{n-1}(\partial W) \cong \partial H_n(W)$ in such a way that the first map becomes the inclusion of the kernel and the last one becomes the projection to the cokernel of $\lambda_W^{\vee}$. 
In particular, the isotopy class $\Delta$ gives a well-defined element $\delta \in \partial H_n(W)$.

\begin{definition} \label{definition alg complex}
For a valid algebraic data $(M,\lambda,\delta)$ the algebraic arc complex $\mathcal{A}^{\text{alg}}(M,\lambda,\delta)$ is the simplicial complex with vertex set
$$\{\alpha \in M^{\vee}: \;  \alpha \; \text{is unimodular and} \; \alpha \mod \lambda^{\vee}(M) =\delta \}$$
and where a set $\{\alpha_0, \cdots, \alpha_p\}$ of (distinct) vertices spans a $p$-simplex if and only if it is unimodular in $M^{\vee}$. 
\end{definition}

The geometric interpretation of the above definition is explained by the following result. 

\begin{proposition} \label{prop simplicial map}
If $(W,\Delta)$ is valid geometric data then there is a simplexwise injective simplicial map 
$$\Phi: \mathcal{A}(W,\Delta) \rightarrow \mathcal{A}^{\text{alg}}(H_{n}(W),\lambda_W,\delta)$$
given (on vertices) by $a \mapsto \alpha:=a_*([D^n , \partial D^n]) \in H_n(W,\partial W) \cong H_n(W)^{\vee}$.
\end{proposition}

\begin{proof}
Firstly we claim that a set $\{a_0, \cdots, a_p\}$ of pairwise disjoint arcs satisfies that the cut manifold $W':=W \setminus \{a_0, \cdots, a_p\}$ is $(n-1)$-connected if and only if the corresponding set $\{\alpha_0,\cdots,\alpha_p\} \subset H_n(W)^{\vee}$ is unimodular and has size precisely $p+1$, which implies the simplexwise injectivity part. 

Indeed, the long exact sequence of the pair $(W,W')$ gives an exact sequence
$$0 \rightarrow H_n(W') \rightarrow H_n(W) \rightarrow H_n(W,W') \rightarrow H_{n-1}(W') \rightarrow 0.$$
The pairwise disjointness of the arcs and excision imply that $H_n(W,W') \cong H_n(\bigsqcup_{i=0}^p{D^n \times D^n},\bigsqcup_{i=0}^p{ D^n \times \partial D^n}) \cong \mathbb{Z}^{p+1}$ in such a way that $H_n(W) \rightarrow H_n(W,W')$ agrees with $\bigoplus_{i=0}^{p}{\alpha_i}: H_n(W) \rightarrow \mathbb{Z}^{p+1}$, hence giving the result. 

Secondly, the boundary map $H_n(W,\partial W) \rightarrow H_{n-1}(W)$ agrees with the quotient map $H_n(W)^{\vee} \rightarrow \partial H_n(W)$ and hence any vertex $a$ of the arc complex will map to $\alpha \in H_n(W)^{\vee}$ such that $[\alpha]=\delta \in \partial H_n(W)$. 
\end{proof}

\subsubsection{The algebraic argument}

In this section we will prove that the algebraic arc complex is weakly Cohen-Macaulay of a certain dimension. 
To state the precise result we need some notation: let $t(M):=\max\{\rk(U): U \subset \lambda^{\vee}(M) \; \text{is a direct summand of} \; M^{\vee}\}$. 

\begin{theorem}\label{theorem connectivity alg complex}
If $(M,\lambda, \delta)$ is valid algebraic data then the algebraic arc complex $\mathcal{A}^{\text{alg}}(M,\lambda,\delta)$ is weakly Cohen-Macaulay of dimension $t(M)-2$.
\end{theorem}

We will deduce this result from a more general one inspired by \cite[Theorem 1.4]{Nina}
We will use the following notation from \cite{Nina} to state it and prove it:
for $X$ a set we let $\mathcal{O}(X)$ denote the poset of non-empty ordered finite sequences of elements of $X$, with partial order given by refinement. 
For $V$ a $\mathbb{Z}$-module let $\mathcal{U}(V)$ denote the subposet of $\mathcal{O}(V)$ consisting of the unimodular sequences. 
If $F \subset \mathcal{O}(V)$ is a subposet and $(v_1,\cdots,v_k) \in \mathcal{U}(V)$, we write $F \cap \mathcal{U}(V)_{(v_1,\cdots,v_k)}$ for the poset of sequences $(w_1,\cdots,w_l) \in F$ such that $(v_1,\cdots,v_k,w_1,\cdots,w_l) \in \mathcal{U}(V)$. 

\begin{theorem} \label{theorem technical}
Let $N':=d_1\mathbb{Z} \oplus \cdots \oplus d_r \mathbb{Z} \oplus \mathbb{Z}^t$, where $d_i \geq 2$. 
Let $N:= \mathbb{Z}^{r+t+l}$, with standard basis elements $x_1,\cdots,x_{r+t+l}$, so that $d_1x_1, \cdots,d_r x_r, x_{r+1},\cdots,x_{r+t}$ is a basis for $N'$. 
Let $N^{\infty}:= N \oplus \mathbb{Z}^{\infty}$, and let the standard basis of $\mathbb{Z}^{\infty}$ be $e_1,e_2,\cdots$.
Fix an element $\delta_0 \in N$.

For any $(v_1,\cdots,v_k) \in \mathcal{U}(N^{\infty})$ with $k \geq 1$
\begin{enumerate}[(1)]
    \item $\mathcal{O}(\delta_0+ N') \cap \mathcal{U}(N^{\infty})$ is $(t-3)$-connected. 
    \item $\mathcal{O}(\delta_0+ N') \cap \mathcal{U}(N^{\infty})_{(v_1,\cdots,v_k)}$ is $(t-3-k)$-connected. 
\end{enumerate}
\begin{enumerate}[(a)]
     \item $\mathcal{O}(\delta_0+ N' \cup \delta_0+ N'+e_1) \cap \mathcal{U}(N^{\infty})$ is $(t-2)$-connected. 
    \item $\mathcal{O}(\delta_0+ N' \cup \delta_0 + N'+e_1) \cap \mathcal{U}(N^{\infty})_{(v_1,\cdots,v_k)}$ is $(t-2-k)$-connected. 
\end{enumerate}
\end{theorem}

In order to state the application that will be relevant for us we need to fix some extra notation:
given a finitely generated free $\mathbb{Z}$-module $N$ and a submodule $N' \subset N$ we let 
\begin{equation*}
    \resizebox{\displaywidth}{!}{$t(N,N'):= \max\{\rk(U): U \subset N' \; \text{is a submodule such that} \; U \subset N \; \text{is a direct summand}\}$}
\end{equation*}
For example, if we take $N=M^{\vee}$, $N'=\lambda^{\vee}(M)$ then $t(N,N')=t(M)$. 

Given the additional data of an element $\delta \in N/N'$ we define $U^{\text{unord}}(N,N',\delta)$ to be the simplicial complex whose vertices are unimodular elements $x \in N$ such that $x \mod N' = \delta$ and whose $p$-simplices are sets of (distinct) $p+1$ vertices $\{x_0,\cdots,x_p\}$ which are unimodular in $N$. 

\begin{corollary}
Let $N$ be a finitely generated free $\mathbb{Z}$-module, let $N' \subset N$ be a submodule and let $\delta \in N/N'$ be an element. 
Then $U^{\text{unord}}(N,N',\delta)$ is weakly Cohen-Macaulay of dimension $t(N,N')-2$. 
\end{corollary}

This implies Theorem \ref{theorem connectivity alg complex} by taking $N=M^{\vee}$, $N'=\lambda^{\vee}(M)$ and the given $\delta \in \partial M= N/N'$.

\begin{proof}
By Smith Normal Form it suffices to consider the situation $N'=d_1\mathbb{Z} \oplus \cdots \oplus d_r \mathbb{Z} \oplus \mathbb{Z}^t \subset N= \mathbb{Z}^{r+t+l}$, where $d_i \geq 2$ for all $i$. 
Then $t=t(N,N')$, and we can pick a representative $\delta_0 \in N$ of $\delta$.

Let $\mathcal{U}(N,N',\delta)$ be the poset of non-empty finite unimodular sequences in $N$ whose elements lie in the coset $\delta= \delta_0+N'$. 
Since a finite sequence of elements in $N$ is unimodular in $N$ if and only if it is unimodular in $N^{\infty}:=N \oplus \mathbb{Z}^{\infty}$ then $\mathcal{
U}(N,N',\delta)=\mathcal{O}(\delta_0+ N') \cap \mathcal{U}(N^{\infty})$ is $(t-3)$-connected Theorem \ref{theorem technical}(1). 

Now view $U^{\text{unord}}(N,N',\delta)$ as a poset, and pick a total ordering $<$ in the set of one-element sequences in $\mathcal{U}(N,N',\delta)$.
There are poset maps 
$${U}^{\text{unord}}(N,N',\delta) \xrightarrow{f} \mathcal{U}(N,N',\delta) \xrightarrow{g} {U}^{\text{unord}}(N,N',\delta)$$
whose composition is the identity, where $f$ sends $\{v_0,\cdots,v_p\} \in {U}^{\text{unord}}(N,N',\delta)$ to the sequence $(v_{i_0},\cdots,v_{i_p})$ with $v_{i_0}< \cdots < v_{i_p}$ and where $g$ sends a sequence to its underlying set. 
Thus, the poset ${U}^{\text{unord}}(N,N',\delta)$ is also $(t-3)$-connected, as required. 

Moreover, for a $(k-1)$-simplex $\{v_1,\cdots, v_k\}$ of $U^{\text{unord}}(N,N',\delta)$, we have $(v_1,\cdots,v_k) \in \mathcal{U}(N^{\infty})$ so $\mathcal{O}(\delta_0+ N') \cap \mathcal{U}(N^{\infty})_{(v_1,\cdots,v_k)}$ is $(t-3-k)$-connected by Theorem \ref{theorem technical}(2).  
Also,
$$f \big(\Lk_{{U}^{\text{unord}}(N,N',\delta)}(\{v_1,\cdots,v_k\})\big) \subset \mathcal{O}(\delta_0+ N') \cap \mathcal{U}(N^{\infty})_{(v_1,\cdots,v_k)}$$ 
and 
$$g \big(\mathcal{O}(\delta_0+ N') \cap \mathcal{U}(N^{\infty})_{(v_1,\cdots,v_k)} \big) = \Lk_{{U}^{\text{unord}}(N,N',\delta)}(\{v_1,\cdots,v_k\}).$$ 
Thus $\Lk_{{U}^{\text{unord}}(N,N',\delta)}(\{v_1,\cdots,v_k\})$ is $(t-3-k)$-connected. 
\end{proof}

The proof of Theorem \ref{theorem technical} is almost identical to the one of \cite[Theorem 1.4]{Nina}, so we will refer to it for the details and only indicate the differences: we always work with the ring $R=\mathbb{Z}$, $sr(R)=2$. 
The analogue of the variable $g$ in \cite{Nina} is what we call $t$. 
When $\delta_0=0$, $r=l=0$ we recover the exact same proof as in \cite{Nina}. 

\begin{proof}[Proof of Theorem \ref{theorem technical}]
We will show the result by induction on $t$. 
For $t \leq 1$ parts (1),(2) and (b) are vacuously true since $k \geq 1$ and any poset is $(-2)$-connected. 

Statement (a) is also true for $t \leq 1$ since $\mathcal{O}(\delta_0+ N' \cup \delta_0+ N'+e_1) \cap \mathcal{U}(N^{\infty})$ contains the one element sequence $(\delta_0+e_1)$ and hence is non-empty, so $(-1)$-connected. 

Thus we will assume that $t \geq 2$ throughout the proof. 
Without loss of generality we can assume that $\delta_0 \in \langle x_1,\cdots,x_r,x_{r+t+1},\cdots,x_{r+t+l} \rangle$, so that its projection to $\mathbb{Z}^t=\langle x_{r+1},\cdots,x_{r+t} \rangle$ vanishes. 

\textbf{Proof of (b).} Let $Y=\delta_0+N' \cup \delta_0+N'+e_1$ and $F:=\mathcal{O}(Y) \cap \mathcal{U}(N^{\infty})_{(v_1,\cdots,v_k)}$. 
We need to show that $F$ is $(t-2-k)$-connected. 
Since $GL_t(\mathbb{Z})$ acts transitively on the set of unimodular vectors of $\mathbb{Z}^t$ for any $t \geq 1$ we don't need to consider two separate cases as in \cite{Nina}, and instead we can use the argument for the case ``$g>sr(R)$''. 

We need to find $f \in GL(N^{\infty})$ such that $f(Y)=Y$ and the projection of $f(v_1)$ to $\mathbb{Z}^{t}=\langle x_{r+1},\cdots,x_{r+t}\rangle$, denoted $f(v_1)|_{\mathbb{Z}^t}$, is unimodular. 
Once $f$ is found we can proceed as in \cite{Nina} to inductively prove the result.
The construction of $f$ itself is analogous to the one in \cite{Nina}: the given $f$ preserves both the set $Y$ and the element $\delta_0$. 

\textbf{Proof of (2).}
We let 
$$X= (\delta_0+ d_1 \mathbb{Z} \oplus \cdots \oplus d_r \mathbb{Z} \oplus \mathbb{Z}^{t-1} \oplus 0) \cup (\delta_0+ d_1 \mathbb{Z} \oplus \cdots \oplus d_r \mathbb{Z} \oplus \mathbb{Z}^{t-1} \oplus x_{r+t})$$
and 
$$F=\mathcal{O}(\delta_0+N') \cap \mathcal{U}(N^{\infty})_{(v_1,\cdots,v_k)}$$
and then mimic the corresponding step in \cite{Nina}. 

\textbf{Proof of (1) and (a).}
For (1) we take the same $X$ and $F$ as above. 
For (2) we take instead $X=\delta_0+N'$, $F=\mathcal{O}(\delta_0+N' \cup \delta_0+N'+e_1) \cap \mathcal{U}(N^{\infty})_{(v_1,\cdots,v_k)}$ and $y_0=\delta_0+e_1$. 
In both cases we then follow the exact same proof as in \cite{Nina}.
\end{proof}

\subsubsection{Skew symmetric forms and genus} \label{section skew sym}

In this section we state a classification result of skew symmetric forms over the integers and use it derive some algebraic results that will be useful. 

The following classification theorem can be found in \cite[Theorem IV.I]{skew}: 

\begin{theorem} \label{classification}
Any pair $(M,\lambda)$ where $M$ is a finitely generated free $\mathbb{Z}$-module and $\lambda: M \otimes M \rightarrow \mathbb{Z}$ is a skew-symmetric form is isomorphic to a unique canonical form
$$H^{\oplus g} \oplus \bigoplus_{i=1}^r{\begin{pmatrix} 0 & d_i \\ -d_i & 0\end{pmatrix}} \oplus (\mathbb{Z}^b,0)$$
where $H$ is the standard hyperbolic form with matrix $\begin{pmatrix} 0 & 1 \\ -1 & 0\end{pmatrix}$, $g \geq 0$, $d_1| d_2 \cdots |d_r$ and $d_i \geq 2$, $r \geq 0$, $(\mathbb{Z}^b,0)$ denotes the zero form of rank $b \geq 0$, and $\oplus$ denotes orthogonal direct sum. 
\end{theorem}

The number $g=g(M)$ is called the \textit{genus} of the form, and can also be defined as $g(M):=\max\{k: H^{\oplus k} \hookrightarrow (M,\lambda)\}$, hence generalizing the definition given in Section \ref{section 5.5}. 
The integer $b$ agrees with the rank of the radical $\rad(M):=\ker(\lambda^{\vee})$, which is also equal to $\rk(\partial M)$ by rank-nullity. 
The values of $d_i$ are certain invariants of the form; and the number $r=r(M)$ satisfies that $2r$ is the minimum number of generators of $\Tors(\partial M) \cong \bigoplus_{i=1}^r{{\mathbb{Z}}/{d_i \mathbb{Z}}\oplus {\mathbb{Z}}/{d_i \mathbb{Z}}}$.
In fact, when $(M,\lambda)$ is written in canonical form then $\partial M= \bigoplus_{i=1}^r{({\mathbb{Z}}/{d_i \mathbb{Z}}\oplus {\mathbb{Z}}/{d_i \mathbb{Z}})} \oplus \mathbb{Z}^{\rk(M)-2r-2g}$.

\begin{rem}\label{rem algebraic formulae}
Canonical form gives the following formulae for $g(M)$ and $t(M)$ which will be useful later:
$g(M)=\frac{1}{2}(\rk(M)-\rk(\partial M)-2 r(M))$ and $t(M)=2g(M)$. 
\end{rem}

The additivity of the genus under orthogonal direct sum follows from Theorem \ref{classification}: 
by taking Smith Normal Form of the canonical from matrix we see that $g(M)$ is twice the number of $1$'s in the diagonal. 
The number of $1$'s in Smith Normal Form is additive under direct sum of matrices. 

\subsubsection{Properties of the cut manifold}

In this section we will show the properties of the cut manifold that were needed in the proof of Theorem \ref{theorem splitting poset connectivity}. 

\begin{lemma} \label{lem technical cutting one arc}
Suppose $(M,\partial,\delta)$ is valid algebraic data such that $\partial M \neq 0$ and $\delta \in \partial M$ is a generator of a cyclic direct summand of $\partial M$ of maximum order.
Let $\alpha \in M^{\vee}$ be unimodular such that $[\alpha]=\delta \in \partial M$, and denote $M'=\ker(\alpha)$ with $\lambda':=\lambda|_{M'}$. Then $g(M')= g(M)$. 
\end{lemma}

\begin{proof}
Let $\delta' \in \partial M'$ be defined as follows: $\alpha \in M^{\vee}$ is unimodular, so there is $x \in M$ such that $\alpha(x)=1$, then $\lambda(x,-)|_{M'} \in (M')^{\vee}$ and we let $\delta':= [\lambda(x,-)|_{M'}] \in \partial M'$. 
This element is well-defined, i.e. independent of the choice of $x$, because if we are given $x'$ with $\alpha(x')=1$ then $x-x' \in \ker(\alpha)=M'$ so $\lambda(x-x',-)|_{M'} \in \lambda^{\vee}(M')$. 
We claim that there is an isomorphism ${\partial M}/{\langle \delta \rangle} \cong {\partial M'}/{\langle \delta' \rangle}$.

Indeed, consider the composition of surjections $\phi: M^{\vee} \rightarrow (M')^{\vee} \rightarrow {\partial M'}/{\langle \delta' \rangle}$. 
It suffices to show that $\ker(\phi)$ is the subgroup of $M^{\vee}$ generated by $\lambda^{\vee}(M)$ and $\alpha$.
To do so we check both inclusions:
if $\beta \in \ker(\phi)$ then $\beta|_{M'}=\lambda(y,-)|_{M'}+k \lambda(x,-)|_{M'}$ for some $y \in M'$ and $k \in \mathbb{Z}$, since $\delta'=[\lambda(x,-)|_{M'}] \in \partial M'$. 
Thus, $\beta= \lambda(y+k x,-)-\lambda(y,x) \alpha$ because both sides of this equation agree on $M'$ and on $x$ by construction, and $M$ is generated by $M'$ and $x$; so $\beta$ lies in the subgroup of $M^{\vee}$ generated by $\lambda^{\vee}(M)$ and $\alpha$. 
Conversely, if $\beta=\lambda(z,-)+l \alpha$ for some $z \in M$ and $l \in \mathbb{Z}$ then we can write $z=z'+\alpha(z) x$ where $z' \in M'$ and then $\beta|_{M'}=\lambda(z',-)|_{M'}+\alpha(z) \lambda(x,-)|_{M'}+0$, so $\phi(\beta)=0$. 

By rank-nullity we have $\rk(M')=\rk(M)-1$.
Now we have two cases to consider: 

(i)$\delta$ has infinite order.
     We claim that the evaluation map $\rad(M) \rightarrow \Hom_{\mathbb{Z}}(\partial M, \mathbb{Z})$, $x \mapsto (\beta \mapsto \beta(x))$ is an isomorphism.
     Indeed, $\Hom_{\mathbb{Z}}(\partial M, \mathbb{Z})=\Ann(\lambda^{\vee}(M) \subset M^{\vee})$, and under the evaluation isomorphism $M \xrightarrow{\cong} (M^{\vee})^{\vee}$ this annihilator is precisely $\rad(M)$. (This is the algebraic analogue of Poincaré duality in $\partial W$.) 
    Since $\delta \in \partial M$ generates a free $\mathbb{Z}$-summand then there is $x \in \rad(M)$ such that $\alpha(x)=1$. But then $\delta'=0$ since $x \in \rad(M)$. 
    Thus $\partial M' \cong \frac{\partial M}{\langle \delta \rangle}$ and so $\rk(\partial M')=\rk(\partial M)-1$ and $r(M)=r(M')$, giving $g(M)=g(M')$ by Remark \ref{rem algebraic formulae}. 
    
(ii) $\delta$ has finite order, so $\partial M$ must be torsion, and by taking canonical form we can assume that $\delta$ generates the last $\frac{\mathbb{Z}}{d_r \mathbb{Z}}$ summand. 
    We  claim that $\delta' \in \partial M'$ must have infinite order: 
    otherwise let $N \in \mathbb{Z}_{>0}$ be such that $N \delta'=0$, then there is some $x' \in M'$ such that $N \lambda(x,-)|_{M'}=\lambda(x',-)|_{M'}$, where $x \in M$ satisfies $\alpha(x)=1$. 
    $\delta$ has finite order $d_r$ by assumption so $d_r \alpha= \lambda(t,-)$ for some $t \in M$. 
    Then, $0 \neq N d_r= \alpha(N d_r x)= \lambda(t, N x)=-N \lambda(x,t)=-\lambda(x',t)=\lambda(t,x')=d_r \alpha(x')=0$, where we have used that $x' \in M'=\ker(\alpha)$ and that $t \in M'$ because $d_r \alpha(t)=\lambda(t,t)=0$. 
    This gives a contradiction and hence shows the claim. 
    Since $\partial M$ is torsion then so is $\partial M/\langle \delta \rangle$, which is isomorphic to $\partial M'/ \langle \delta' \rangle$; thus  $\rk(\partial M')=1$, and $\Tors(\partial M') \subset \Tors({\partial M'}/{\langle \delta' \rangle})= \partial M/\langle \delta \rangle = \Big(\bigoplus_{i=1}^{r-1}{{\mathbb{Z}}/{d_i \mathbb{Z}}\oplus {\mathbb{Z}}/{d_i \mathbb{Z}}} \Big) \oplus {\mathbb{Z}}/{d_r \mathbb{Z}}$. 
    In particular $r(M')<r(M)$ as the right hand side of the previous formula has less than $2r=2r(M)$ summands, so $r(M) \geq r(M')+1$ and hence the formula for genus in Remark \ref{rem algebraic formulae} gives $g(M') \geq g(M)$, hence the result as the reverse inequality is trivial because $M' \subset M$. 
\end{proof}

\begin{rem} \label{rem n odd}
The above proof uses Remark \ref{rem algebraic formulae} plus the fact that $\lambda(t,t)=0=\lambda(x,x)$ for $t,x$ as in the proof. 
These facts use that the form is skew-symmetric, and one can check that the analogue of Lemma \ref{lem technical cutting one arc} for symmetric forms is false.
This is the key step in which we need to take $n$ odd: Lemma \ref{lem technical cutting one arc} is used in proving Proposition \ref{prop cut manifold}, which in turn is used in the inductive step of the proof of Theorem \ref{theorem splitting poset connectivity}, which is needed to prove Theorem \hyperref[theorem C]{C}. 
\end{rem}

\begin{corollary} \label{cor cut manifold properties}
Let $(M,\partial,\delta)$ be valid algebraic data such that $\delta \in \partial M$ generates a cyclic direct summand of $\partial M$ of maximum order. 
Let $\{\alpha_0,\cdots,\alpha_p\}$ be a $p$-simplex of $\mathcal{A}^{\text{alg}}(M,\lambda,\delta)$ and let $M':= \bigcap_{i=0}^{p}{\ker(\alpha_i)}$ and $\lambda'=\lambda|_{M'}$, then 
\begin{enumerate}
    \item $\rk(M')=\rk(M)-(p+1)$. 
    \item $g(M') \geq g(M)-(p+1)$. Moreover, if $\partial M \neq 0$ then $g(M') \geq g(M)-p$. 
\end{enumerate}
\end{corollary}

\begin{proof}
Rank-nullity implies (i) since $\alpha_0,\cdots,\alpha_p$ is unimodular in $M^{\vee}$. 
The first part of (ii) is shown in \cite[Corollary 4.2]{high} for the case $p=0$, and the general case follows by iterating. 
Lemma \ref{lem technical cutting one arc} gives the second part of (ii) in the special case $p=0$, and the general case $p \geq 1$ is a consequence of \cite[Corollary 4.2]{high}. 
\end{proof}

Now we will show the geometric analogue of the above result. 

\begin{proposition} \label{prop cut manifold}
Let $(W,\Delta)$ be valid geometric data and $\{a_0, \cdots, a_p\} \in \mathcal{A}(W,\Delta)$ be a $p$-simplex. Then 
\begin{enumerate}[(i)]
    \item $\rk(H_n(W \setminus \{a_0, \cdots, a_p\}))=\rk(H_n(W))-(p+1)<\rk(H_n(W))$. 
    \item $g(W \setminus \{a_0, \cdots, a_p\}) \geq g(W)-(p+1)$. 
     Moreover if $H_{n-1}(A)= H_{n-1}(\partial W) \neq 0$ and $\delta$ generates a direct summand of maximum order then $g(W \setminus \{a_0, \cdots, a_p\}) \geq g(W)-p$.
\end{enumerate}
\end{proposition}

\begin{proof}
Let $W'=W \setminus a$ denote the cut manifold and consider the exact sequence of the proof of Proposition \ref{prop simplicial map}
$$0 \rightarrow H_n(W') \rightarrow H_n(W) \xrightarrow{\bigoplus_{i=0}^p{\alpha_i}} \mathbb{Z}^{p+1} \rightarrow H_{n-1}(W') \rightarrow 0$$
where $\alpha_i=(a_i)_*[D^n ,\partial D^n] \in H_n(W,\partial W) \cong H_n(W)^{\vee}$. 
Then $H_n(W')=\bigcap_{i=0}^p \ker(\alpha_i)$ and $\lambda_{W'}$ is the restriction of $\lambda_W$ by naturality of the intersection product. 

The triple $(H_n(W),\lambda_W,\delta)$ defines valid algebraic data, and by Proposition \ref{prop simplicial map} $\{\alpha_0,\cdots,\alpha_p\}$ is a $p$-simplex of $\mathcal{A}^{\text{alg}}(H_n(W),\lambda_W,\delta)$.
Thus, Corollary \ref{cor cut manifold properties} gives the result.
\end{proof}

\subsection{The connectivity of the arc complex} \label{section arc complex}

In this section we will prove Theorem \ref{theorem arc complex connectivity} saying that the arc complex is highly connected. 
The proof is inspired by the one of \cite[Lemma 5.5]{high}.

\begin{proof}[Proof of Theorem \ref{theorem arc complex connectivity}.]

We can assume that $g(W) \geq 1$ as otherwise the result is vacuously true. 

Let $k$ be fixed, $0 \leq k \leq g(W)-2$, we will show that any continuous map $f: S^k \rightarrow |\mathcal{A}(W,\Delta)|$ is nullhomotopic. 
Firstly by the simplicial approximation theorem we can suppose that $S^k \cong |L|$ for $L$ a PL triangulation and that $f: L \rightarrow \mathcal{A}(W,\Delta)$ is simplicial. 
Consider the composition 
$\Phi \circ f: S^k \xrightarrow{f} |\mathcal{A}(W,\Delta)| \rightarrow |\mathcal{A}^{\text{alg}}(H_{n}(W),\lambda_W,\delta)|$.

By Theorem \ref{theorem connectivity alg complex}, $\mathcal{A}^{\text{alg}}(H_n(W),\lambda_W,\delta)$ is weakly Cohen-Macaulay of dimension $t(H_n(W))-2$, and by Remark \ref{rem algebraic formulae} $t(H_n(W))=2g(W)$. 
Since we are assuming that $g(W) \geq 1$ then $2g(W)-2 \geq g(W)-1$, so $\mathcal{A}^{\text{alg}}(H_n(W),\lambda_W,\delta)$ is also weakly Cohen-Macaulay of dimension $g(W)-1$. 
Since $k \leq g(W)-2$ then by \cite[Theorem 2.4]{high} $L$ can be extended to a PL triangulation $K$ on $D^{k+1}$ with the property that the star of each vertex $v \in K \setminus L$ intersects $L$ at a single (possibly empty) simplex; and such that that there is a simplicial map $g: K \rightarrow \mathcal{A}^{\text{alg}}(H_n(W),\lambda_W,\delta)$ extending $\Phi \circ f$ with the additional property that  $g(\Lk_K(v)) \subset \Lk_{\mathcal{A}^{\text{alg}}(H_n(W),\lambda_W,\delta)}(g(v))$ for any vertex $v \in K \setminus L$. 

We will prove that $g: K \rightarrow \mathcal{A}^{\text{alg}}(H_{n}(W),\lambda_W,\delta)$ lifts along $\Phi$ to a nullhomotopy $G: K \rightarrow \mathcal{A}(W,\Delta)$ of $f$.
Observe that $G|_L=f$ is fixed. 

Choose an enumeration of the vertices of $K \setminus L$ as $v_1,\cdots,v_N$.
For each $1 \leq i \leq N$ we shall inductively pick smooth embeddings $a_i: D^{n} \hookrightarrow W$ satisfying: 
\begin{enumerate}[(i)]
    \item Each $\partial a_i$ has image in $\int(A) \subset W$ and lies in the isotopy class $\Delta$.
    \item $a_i$ is transverse to $\partial W$ and intersects it precisely along $\partial a_i$.
    \item $(a_i)_*([D^n,\partial D^n])= (\Phi \circ f)(v_i) \in \mathcal{A}^{\text{alg}}(H_{n}(W),\lambda_W,\delta)$.  
    \item If $i \neq j$ then $\im({a_i})$ and $\im({a_j})$ are disjoint.
    \item If $w \in L$ then $\im({a_i})$ and $\im(f(w))$ are disjoint. 
\end{enumerate}

Suppose that the embeddings ${a_1},\cdots,{a_{s-1}}$ have already been chosen and satisfy the above properties, we will construct the embedding ${a_s}$. 

Pick an embedding $\iota: S^{n-1} \hookrightarrow A \subset \partial W$ representing the isotopy class $\Delta$ and make it transverse, and hence disjoint, to all the $\partial f(w)$ for $w \in L$ and all the $\partial a_i$ for $i<s$. 
We have $(\Phi \circ f)(v_s) \in H_n(W,\partial W) \cong \pi_n(W,\partial W)$ such that $(\Phi \circ f)(v_s) \mapsto \delta \in H_{n-1}(\partial W) \cong \pi_{n-1}(\partial W)$ and hence there is a continuous map ${a_s}: D^n \rightarrow W$ such that $\partial {a_s}=\iota$ and ${a_s}$ represents the homology class $(\Phi \circ f)(v_s)$. 
By transversality we can make ${a_s}$ to be smooth, immersed with at worse double point singularities, transverse to all the ${a_i}$ for $i<s$, to all the $f(w)$ for $w \in L$ and to $\partial W$, and intersecting $\partial W$ only along $\partial {a_s}$.
Moreover we can also make it be in general position with respect to all the ${a_i}$ for $i<s$, to all the $f(w)$ for $w \in L$, so that there are no triple intersection points. 

Now, by the half-Whitney trick, see \cite[Claim 6.18]{stablemoduli}, we can isotope $a_s$ to cancel any intersection point with ${a_i}$, $i<s$, or with $f(w)$ for $w \in L$.
This procedure does not create any new intersection points each time we cancel because we can pick each Whitney disc disjoint from the $a_i$'s and $f(w)$'s not in play by transversality, using that $2+n<2n$ and that there are no triple intersections. 
This isotopy changes $\partial a_s$ but it keeps it in $\int(A)$ and in the isotopy class $\Delta$, and moreover we can always use transversality at the end of the procedure to ensure condition (ii) above still holds. 

We can use Haefliger's trick to homotope $a_s$ so that it becomes embedded, without changing the above disjointness conditions and the behaviour of $a_s$ near the boundary: 
for each self-intersection point the cancelling procedure takes place in a small neighbourhood of that point, and $a_s$ was already embedded near the boundary. 

Thus, the new collection ${a_1},\cdots,{a_s}$ will satisfy all the properties (i)-(v) above, as required. 

Finally set $G(v_i):=a_i$ for $1 \leq i \leq N$.  
By conditions (i),(ii) and (iii) on the choice of the ${a_i}$'s, $G$ gives a lift of $g$ along $\Phi$ on the vertices of $K$, so to finish the proof it suffices to check that each $a_i$ is indeed non-separating so that it defines a vertex in the arc complex, and that such $G$ is simplicial. 
Let $\tau$ be a simplex in $K$, we will show that $G(\tau)$ is a simplex in $\mathcal{A}(W,\Delta)$.
The pairwise disjointness condition holds because $f$ is itself simplicial and the embeddings $a_i$ are pairwise disjoint and disjoint to all the $f(w)$ for $w \in L$. 
The connectivity of the jointly cut manifold holds because it is equivalent to the unimodularity of the corresponding elements in the algebraic arc complex by the first step in the proof of Proposition \ref{prop simplicial map}, and by assumption $g(\tau)$ is a simplex as $g$ is simplicial. 
\end{proof}

\bibliographystyle{amsalpha}
\bibliography{bibliography}

\end{document}